\documentclass{article}

\usepackage{microtype}
\usepackage{graphicx}
\usepackage{subfigure}
\usepackage{booktabs} % for professional tables

\usepackage{hyperref}

\usepackage{appendix}
\usepackage{amsthm}
\usepackage{amsmath}

\usepackage{amsfonts}
\usepackage{multirow}
\usepackage{color}

\newtheorem{theorem}{Theorem}
\newtheorem{lemma}{Lemma}

\newtheorem{definition}{Definition}
\newtheorem{assumption}{Assumption}
\newtheorem{remark}{Remark}

\usepackage[accepted]{icml2019}

\icmltitlerunning{Faster Stochastic ADMM for Nonconvex Optimization}

\begin{document}

\twocolumn[
\icmltitle{ Faster Stochastic Alternating Direction Method of Multipliers \\ for Nonconvex Optimization }




\begin{icmlauthorlist}
\icmlauthor{Feihu Huang}{1}
\icmlauthor{Songcan Chen}{2,3}
\icmlauthor{Heng Huang}{1,4}
\end{icmlauthorlist}

\icmlaffiliation{1}{Department of Electrical \& Computer Engineering, University of Pittsburgh, PA 15261, USA}
\icmlaffiliation{2}{College of Computer Science \& Technology, Nanjing University of Aeronautics and Astronautics, Nanjing 211106, China}
\icmlaffiliation{3}{MIIT Key Laboratory of Pattern Analysis \& Machine Intelligence}
\icmlaffiliation{4}{JD Finance America Corporation}

\icmlcorrespondingauthor{Feihu Huang}{huangfeihu2018@gmail.com}
\icmlcorrespondingauthor{Heng Huang}{heng.huang@pitt.edu}

\icmlkeywords{Machine Learning, ICML}

\vskip 0.3in
]



\printAffiliationsAndNotice{}  % leave blank if no need to mention equal contribution

\vspace*{-8pt}
\begin{abstract}
\vspace*{-2pt}
In this paper, we propose a faster stochastic alternating direction method of multipliers (ADMM) for nonconvex optimization by using a new stochastic path-integrated differential estimator (SPIDER), called as SPIDER-ADMM. Moreover, we prove that the SPIDER-ADMM achieves a record-breaking incremental first-order oracle (IFO) complexity of $\mathcal{O}(n+n^{1/2}\epsilon^{-1})$ for finding an $\epsilon$-approximate stationary point, which improves the deterministic ADMM by a factor $\mathcal{O}(n^{1/2})$, where $n$ denotes the sample size. As one of major contribution of this paper, we provide a new theoretical analysis framework for nonconvex stochastic ADMM methods with providing the optimal IFO complexity. Based on this new analysis framework, we study the unsolved optimal IFO complexity of the existing non-convex SVRG-ADMM and SAGA-ADMM methods,
 and prove they have the optimal IFO complexity of $\mathcal{O}(n+n^{2/3}\epsilon^{-1})$. Thus, the SPIDER-ADMM improves the existing stochastic ADMM methods by a factor of $\mathcal{O}(n^{1/6})$. Moreover, we extend SPIDER-ADMM to the online setting, and propose a faster online SPIDER-ADMM. Our theoretical analysis shows that the online SPIDER-ADMM has the IFO complexity of $\mathcal{O}(\epsilon^{-\frac{3}{2}})$, which improves the existing best results by a factor of $\mathcal{O}(\epsilon^{-\frac{1}{2}})$. Finally, the experimental results on benchmark datasets validate that the proposed algorithms have faster convergence rate than the existing ADMM algorithms for nonconvex optimization.
\end{abstract}

\vspace*{-12pt}
\section{Introduction}
\vspace*{-6pt}
Alternating direction method of multipliers (ADMM) \cite{gabay1976dual,boyd2011distributed} is a powerful optimization tool
for the composite or constrained problems in machine learning. In general, it
considers the following optimization problem:
\vspace*{-4pt}
\begin{equation} \label{eq:1}
\min_{x,y} f(x) + g(y), \quad \mbox{s.t.} \ Ax + By =c, \nonumber
\vspace*{-4pt}
\end{equation}
where $f(x):\mathbb{R}^d\rightarrow \mathbb{R}$ and $g(y):\mathbb{R}^p\rightarrow \mathbb{R}$ are convex functions.
For example, in machine learning, $f(x)$ can be used for the empirical loss, $g(y)$ for the structure regularizer,
and the constraint for encoding the structure pattern of model parameters.
Due to the flexibility in splitting the objective function into loss $f(x)$ and regularizer $g(y)$,
the ADMM can relatively easily solve some complicated structure problems in machine learning,
such as the graph-guided fused lasso \cite{kim2009multivariate} and the overlapping group lasso,
which are too complicated for the other popular optimization methods
such as proximal gradient methods \cite{nesterov2005smooth,beck2009fast}.
Thus, the ADMM has been extensively studied in recent years \cite{boyd2011distributed,nishihara2015general,xu2017admm}.

\vspace*{-2pt}
The above deterministic ADMM generally needs to compute the gradients of empirical loss function
on all examples at each iteration,
which makes it unsuitable for solving big data problems.
Thus, the online and stochastic versions of ADMM \cite{wang2012online,suzuki2013dual,ouyang2013stochastic}
are developed. However, due to large variance of stochastic gradients,
these stochastic methods suffer from a slow convergence rate.
Recently, some fast stochastic ADMM methods \cite{zhong2014fast,suzuki2014stochastic,zheng2016fast} have been proposed
by using the variance reduced (VR) techniques.

\begin{table*}
\vspace*{-0.4cm}
  \centering
  \caption{  IFO complexity comparison of the \textbf{non-convex ADMM} methods for finding
  an $\epsilon$-approximate stationary point of the problem \eqref{eq:2},
  i.e., $\mathbb{E}\|\nabla \mathcal {L}(x,y_{[m]},z)\|^2\leq \epsilon$. $n$ denotes the sample size.}
  \label{tab:1}
\begin{tabular}{c|c|c|c}
  \hline
  \textbf{Problem} & \textbf{Algorithm} & \textbf{Reference}  & \textbf{IFO}  \\ \hline
  \multirow{4}*{Finite-sum} & ADMM & \citet{jiang2019structured}    &  $\mathcal{O}(n\epsilon^{-1})$ \\ \cline{2-4}
  &SVRG-ADMM & \citet{huang2016stochastic,zheng2016stochastic}    &  $\mathcal{O}(n+n^{\frac{2}{3}}\epsilon^{-1})$ \\ \cline{2-4}
  &SAGA-ADMM & \citet{huang2016stochastic}    & $\mathcal{O}(n+n^{\frac{2}{3}}\epsilon^{-1})$ \\ \cline{2-4}
  &SPIDER-ADMM  & Ours  & {\color{red}{$\mathcal{O}( n+n^{\frac{1}{2}}\epsilon^{-1})$ }}\\ \hline
  \multirow{2}*{Online} & SADMM  &\citet{huang2018mini}   & $\mathcal{O}(\epsilon^{-2})$ \\ \cline{2-4}
  &Online SPIDER-ADMM  & Ours  & {\color{red}{$\mathcal{O}(\epsilon^{-\frac{3}{2}})$}}  \\
  \hline
\end{tabular}
\vspace*{-0.5cm}
\end{table*}

\vspace*{-2pt}
So far, the above discussed ADMM methods build on the convexity of objective functions.
In fact, ADMM is also highly successful in solving various nonconvex problems
such as tensor decomposition \cite{kolda2009tensor} and training neural networks \cite{Taylor2016Training}.
Thus, some works \cite{li2015global,wang2015convergence,wang2015global,hong2016convergence,jiang2019structured}
have devoted to studying the non-convex ADMM methods.
More recently, for solving the big data problems,
the nonconvex stochastic ADMMs \cite{huang2016stochastic,zheng2016stochastic} have been
proposed with the VR techniques such as the SVRG \cite{johnson2013accelerating} and
the SAGA \cite{defazio2014saga}. In addition, \citet{huang2018mini}
have extended the online/stochastic ADMM \cite{ouyang2013stochastic}
to the nonconvex setting.

\vspace*{-2pt}
Although these works have studied the convergence of nonconvex stochastic ADMMs and proved
these methods have $\mathcal{O}(\frac{c}{T})$ convergence rate, where $T$ denotes number of iteration and
$c$ a constant independent on $T$, they have not provided the \textbf{optimal} incremental/stochastic first-order oracle (IFO/SFO \cite{Ghadimi2013StochasticFA}) complexity
for these methods yet. In other words, they have only proved these stochastic ADMMs have the same convergence rate to
the deterministic ADMM \cite{jiang2019structured}, but don't tell us whether these stochastic ADMMs have less IFO complexity than the deterministic ADMM,
which is a key assessment criteria of the first-order stochastic methods \cite{Reddi2016Prox}.
For example, from the existing noncovex SAGA-ADMM and SVRG-ADMM \cite{zheng2016stochastic,huang2016stochastic},
we only obtain a \textbf{rough} IFO complexity of $\mathcal{O}(n + bc\epsilon^{-1})$ for finding an $\epsilon$-approximate stationary point,
where $b$ denotes the mini-batch size. In their convergence analysis, to ensure the convergence of these methods, they need to choose a small step size $\eta$
and a large penalty parameter $\rho$.
Under this case, we maybe have $bc\geq n$, so that these stochastic ADMMs have no less IFO complexity than the deterministic ADMM.
Thus, there still exist two important problems to be addressed:
\vspace*{-4pt}
\begin{itemize}
\vspace*{-6pt}\item \emph{ Does the stochastic ADMM have less IFO complexity than the deterministic ADMM for nonconvex optimization?}
\vspace*{-6pt}\item \emph{ If the stochastic ADMM improves IFO complexity, how much can it improve?}
\end{itemize}
\vspace*{-10pt}
In the paper,  we answer the above challenging questions
with positive solutions and propose a new faster stochastic
ADMM method (\emph{i.e.,} SPIDER-ADMM) to solve the following nonconvex nonsmooth problem:
\vspace*{-3pt}
\begin{align} \label{eq:2}
\min_{x,\{y_j\}_{j=1}^m} & f(x) := \left\{
\begin{aligned}
 & \frac{1}{n}\sum_{i=1}^n f_i(x)  \ \mbox{(finite-sum)}\\
 & \mathbb{E}_{\zeta} [f(x;\zeta)]  \ \mbox{(online)}
 \end{aligned} \right.  +  \sum_{j=1}^m g_j(y_j) \nonumber \\
 \mbox{s.t.} & \quad Ax + \sum_{j=1}^mB_jy_j =c,
\end{align}
where $A\in \mathbb{R}^{l\times d}$, $B_j\in \mathbb{R}^{l\times p}$ for all $j\in [m]$, and 
$\zeta$ is a random variable following an unknown distribution.
Here $f(x): \mathbb{R}^d\rightarrow \mathbb{R}$ is a \emph{nonconvex} and smooth
function, and $g_j(y_j): \mathbb{R}^p\rightarrow \mathbb{R}$ is a convex and possibly
\emph{nonsmooth} function for all $j\in [m], \ m\geq 1$.
In machine learning, $f(x)$ can be used for losses such as activation functions of neural networks, $\sum_{j=1}^m g_j(y_j)$ can
be used for not only single structure penalty (\emph{e.g.,} sparse, low rank) but also
superposition structures penalties (\emph{e.g.}, sparse + low rank, sparse + group sparse), which are
widely applied in robust PCA \cite{candes2011robust}, subspace clustering \cite{liu2010robust}, and dirty models \cite{jalali2010dirty}.
For the problem \eqref{eq:2}, its finite-sum subproblem generally arises from the empirical loss minimization and M-estimation.
While its online subproblem comes from the expected loss minimization. To address the online subproblem, we extend the SPIDER-ADMM to
the online setting, and propose an online SPIDER-ADMM.

\subsection*{Challenges and Contributions}
Our SPIDER-ADMM methods build on the variance-reduced technique of SPIDER \cite{fang2018spider} and SpiderBoost \cite{wang2018spiderboost},
which is a variant of stochastic
recursive gradient algorithm (SARAH \cite{nguyen2017sarah,nguyen2017stochastic}) and reaches the state-of-the-art IFO complexity
as the SNVRG \cite{zhou2018stochastic}.
Although the SPIDER and SpiderBoost have shown good performances in the stochastic gradient descent (SGD) and proximal SGD methods,
applying these techniques to the nonconvex ADMM method \emph{is not a trivial task}.
There exist the following two main \textbf{challenges}:
\vspace{-6pt}
\begin{itemize}
 \vspace*{-3pt} \item Due to failure of the Fej\'{e}r monotonicity of iteration, the convergence analysis of the nonconvex ADMM is generally quite difficult \cite{wang2015convergence}.
                 With using the inexact stochastic gradient, this difficulty is greater in the nonconvex stochastic ADMM methods;
 \vspace*{-10pt} \item To obtain the optimal IFO complexity of our methods, we need to design a new effective \emph{Lyapunov} function, which can not follow
                 the existing nonconvex stochastic ADMM methods \cite{huang2016stochastic}.
 \vspace*{-10pt}
\end{itemize}
In this paper, thus, we will fill this gap between
the nonconvex ADMM and the SPIDER/SpiderBoost methods.
Our main \textbf{contributions} are summarized as follows:
\vspace{-6pt}
\begin{itemize}
\setlength\itemsep{0em}
\vspace*{-4pt}
\item[1)] We propose a faster stochastic ADMM ( \emph{i.e.,} SPIDER-ADMM ) method for nonconvex optimization based on the SPIDER/SpiderBoost.
          Moreover, we prove that the SPIDER-ADMM achieves a lower
          IFO complexity of $\mathcal{O}(n+n^{1/2}\epsilon^{-1})$
          for finding an $\epsilon$-approximate stationary point,
          which improves the deterministic ADMM by a factor $\mathcal{O}(n^{1/2})$.
 \vspace*{-4pt}
\item[2)] We extend the SPIDER-ADMM method to the online setting, and propose a faster online SPIDER-ADMM for
          nonconvex optimization. Moreover, we prove that the online SPIDER-ADMM achieves a lower
          IFO complexity of $\mathcal{O}(\epsilon^{-\frac{3}{2}})$,
          which improves the existing best results by a factor of $\mathcal{O}(\epsilon^{-\frac{1}{2}})$.
\vspace*{-4pt}
\item[3)] We provide a useful theoretical analysis framework for nonconvex stochastic ADMM methods
          with providing the optimal IFO complexity.
          Based on our new analysis framework, we also prove that the existing nonconvex SVRG-ADMM and SAGA-ADMM
          have the optimal IFO complexity of $\mathcal{O}(n+n^{2/3}\epsilon^{-1})$.
          Thus, our SPIDER-ADMM improves the existing stochastic ADMMs by a factor of $\mathcal{O}(n^{1/6})$.
\end{itemize}
\subsection*{Notations}
 \vspace*{-6pt}
Let $y_{[m]}= \{y_1,\cdots,y_m\}$ and $y_{[j:m]}= \{y_{j},\cdots,y_m\}$ for $j\in[m]=\{1,2,\cdots,m\}$.
Given a positive definite matrix $G$, $\|x\|^2_G = x^TGx$;
$\sigma_{\max}(G)$ and $\sigma_{\min}(G)$
denote the largest and smallest eigenvalues of matrix $G$, respectively;
$\kappa_G = \frac{\sigma_{\max}(G)}{\sigma_{\min}(G)} \geq 1$.
$\sigma^A_{\max}$ and $\sigma^A_{\min}$ denote the largest and smallest eigenvalues of matrix $A^T A$, respectively.
Given positive definite matrices $\{H_j\}_{j=1}^m$, let $\sigma^H_{\min} = \min_j \sigma_{\min}(H_j)$
and $\sigma^H_{\max} = \max_j \sigma_{\max}(H_j)$. $I_d$ denotes a $d\times d$ identity matrix.
\section{Preliminaries}
 \vspace*{-6pt}
In the section, we introduce some preliminaries regarding problem \eqref{eq:2}.
First, we restate the standard
$\epsilon$-approximate stationary point of the nonconvex problem \eqref{eq:2}
used in \cite{jiang2019structured,zheng2016stochastic}.

\begin{definition} \label{def:1}
Given $\epsilon>0$, the point $(x^*,y_{[m]}^*,z^*)$ is said to be an
$\epsilon$-stationary point of the problem \eqref{eq:2},
if it holds that
\begin{align}
 \mathbb{E}\big[ \mbox{dist}(0,\partial L(x^*,y^*_{[m]},z^*))^2 \big] \leq \epsilon,
\end{align}
where $L(x,y_{[m]},z)=f(x) + \sum_{j=1}^m g_j(y_j) - \langle z, Ax + \sum_{j=1}^mB_jy_j-c \rangle$,
\vspace*{-6pt}
\begin{align}
   \partial L(x,y_{[m]},z) = \left [ \begin{matrix}
     \nabla_x L(x,y_{[m]},z) \\
     \partial_{y_1} L(x,y_{[m]},z) \\
      \cdots \\
     \partial_{y_m} L(x,y_{[m]},z) \\
     -Ax-\sum_{j=1}^mB_jy_j+c
 \end{matrix}
 \right ], \nonumber
\end{align}
\vspace*{-6pt}
and $\mbox{dist}(0,\partial L)=\inf_{L'\in \partial L} \|0-L'\|.$
\end{definition}

Next, we give some standard
assumptions regarding problem \eqref{eq:2} as follows:
\begin{assumption}
Each loss function $f_i(x)$ is $L$-smooth such that
\vspace*{-6pt}
\begin{align}
\|\nabla f_i(x)-\nabla f_i(y)\| \leq L \|x - y\|, \ \forall x,y \in \mathbb{R}^d, \nonumber
\end{align}
which is equivalent to
\vspace*{-6pt}
\begin{align}
f_i(x) \leq f_i(y) + \nabla f_i(y)^T(x-y) + \frac{L}{2}\|x-y\|^2.  \nonumber
\end{align}
\end{assumption}
\vspace*{-6pt}
\begin{assumption}
Full gradient of loss function $f(x)$ is bounded, i.e., there exists a constant $\delta >0$ such that for all $x$,
it follows $\|\nabla f(x)\|^2 \leq \delta^2$.
\end{assumption}
\vspace*{-6pt}
\begin{assumption}
$f(x) $ and $g_j(y_j)$ for all $j\in [m]$ are all lower bounded, and let
$f^*=\inf_x f(x) > - \infty$ and $g_j^*=\inf_{y_j} g_j(y_j) > - \infty$.
\end{assumption}
\vspace*{-6pt}
\begin{assumption}
$A$ is a full row or column rank matrix.
\end{assumption}

Assumption 1 imposes smoothness on the individual loss functions, which is commonly used in convergence analysis of
the nonconvex algorithms \cite{Ghadimi2013StochasticFA,ghadimi2016mini}.
Assumption 2 shows full gradient of loss function have a bounded norm, which is used in
the stochastic gradient-based and ADMM-type methods \cite{boyd2011distributed,suzuki2013dual,hazan2016introduction}.
Assumptions 3 and 4 have been used in the study of nonconvex ADMMs
\cite{hong2016convergence,jiang2019structured,zheng2016stochastic}.
Assumptions 3 guarantees the feasibility of the problem \eqref{eq:2}.
Assumption 4 guarantees the matrix $A^TA$ or $AA^T$ is non-singular.
Since there exist multiple regularizers in the above problem \eqref{eq:2},
$A$ is general a full column rank matrix.
Without loss of generality, we will use the full column rank matrix $A$ below.
\section{ Fast SPIDER-ADMM Method }
In the section, we propose a new faster stochastic ADMM algorithm, \emph{i.e.,} SPIDER-ADMM, to solve
the finite-sum problem \eqref{eq:2}.
We begin with giving the augmented Lagrangian function of the problem \eqref{eq:2}:
\vspace*{-6pt}
\begin{align}
 \mathcal {L}_{\rho}(x,y_{[m]},z) = & f(x) + \sum_{j=1}^mg_j(y_j) - \langle z, Ax + \sum_{j=1}^m B_jy_j-c\rangle \nonumber \\
 &  + \frac{\rho}{2} \|Ax+ \sum_{j=1}^mB_jy_j-c\|^2,
\end{align}
where $z \in \mathbb{R}^l$ and $\rho >0$ denote the dual variable and penalty parameter, respectively.
Algorithm \ref{alg:1} gives the SPIDER-ADMM algorithmic framework.

In Algorithm \ref{alg:1}, we use the proximal method to update the variables $\{y_j\}_{j=1}^m$.
At the step 9 of Algorithm \ref{alg:1}, we update the variables $\{y_j\}_{j=1}^m$ by solving the following subproblem, for all $j\in [m]$
\begin{align}
y^{k+1}_j \!=\! \mathop{\arg\min}_{y_j\in \mathbb{R}^p} \mathcal {L}_{\rho}(x_k,y^{k+1}_{[j-1]},y_j,y^{k}_{[j+1:m]},z_k)
\!+\! \frac{1}{2}\|y_j\!-\!y^k_j\|_{H_j}^2, \nonumber
\end{align}
where
\begin{align}
\vspace*{-6pt}
& \mathcal {L}_{\rho} (x_k,y^{k+1}_{[j-1]},y_j,y^{k}_{[j+1:m]},z_k) =  f(x_k) + \sum_{i=1}^{j-1} g_i(y^{k+1}_{i}) \nonumber  \\
&+ g_j(y_{j}) \!+\! \sum_{i=j+1}^m g_j(y^{k}_{j}) \!-\! z_k^T( B_jy_j +\tilde{c}) \!+\! \frac{\rho}{2}\| B_jy_j +\tilde{c} \|^2 \nonumber
\vspace*{-6pt}
\end{align}
with $\tilde{c} = Ax_k+\sum_{i=1}^{j-1}B_iy_i^{k+1}+ \sum_{i=j+1}^{m}B_iy_i^{k}-c$ and $H_j\succ 0$.
When set $H_j = r_j I_p - \rho B_j^TB_j\succeq I_p$ with $r_j \geq \rho
\sigma_{\max}(B^T_j B_j) + 1$ for all $j\in [m]$ to linearize the term $\frac{\rho}{2}\|B_jy_j + \tilde{c}\|^2$,
then we can use the following proximal operator to update $y_j$, for all $j\in [m]$
\begin{align}
 y^{k+1}_j = \mathop{\arg\min}_{y_j\in \mathbb{R}^p} \frac{r_j}{2}\|y_j-w^k_j\|^2 +  g_j(y_j),
\end{align}
where $w^k_j=\frac{1}{r_j}\big( H_jy^k_j - \rho B^T_j\tilde{c} + B^T_jz_k\big)$.

To update $x$, we define
an approximated function over $x_k$ as follows:
\vspace*{-6pt}
\begin{align}
 &\hat{\mathcal {L}}_{\rho} (x,y^{k+1}_{[m]}, z_k,v_k) =
 f(x_k) + v_k^T(x-x_k)+ \frac{1}{2\eta}\|x-x_k\|^2_G \nonumber \\
 & + \sum_{j=1}^m g_j(y^{k+1}_{j}) - z_k^T(Ax+ \sum_{j=1}^m B_jy^{k+1}_j-c) \nonumber \\
 & + \frac{\rho}{2}\|Ax+\sum_{j=1}^mB_jy_j^{k+1}-c\|^2,
\end{align}
where $\eta>0$ is a step size; $v_k$ is a stochastic gradient over $x_k$;
$G\succ 0$ is a positive matrix.
In updating $x$, to avoid computing inverse of $\frac{G}{\eta} + \rho A^TA$,
we can set $G = r I_d - \rho \eta A^TA \succeq I_d$ with $r \geq \rho \eta
\sigma^A_{\max} + 1 $ to linearize term $\frac{\rho}{2}
\|Ax+ \sum_{j=1}^mB_jy^{k+1}_{j}-c\|^2$. Then at the step 10 of Algorithm \ref{alg:1}, we have
\begin{align}
 x_{k+1} = \frac{Gx_k}{r} - \frac{\eta v_k}{r} - \frac{\eta\rho}{r}A^T(\sum_{j=1}^m B_jy_j^{k+1} - c - \frac{z_k}{\rho}). \nonumber
\end{align}

In Algorithm \ref{alg:1}, after setting $v_0 = \nabla f(x_0)$, for each subsequent iteration $k$, we have:
\vspace*{-6pt}
\begin{align}
 v_{k} = \nabla f_{\mathcal{I}_k}(x_k) - \nabla f_{\mathcal{I}_k}(x_{k-1}) + v_{k-1},
\end{align}
where $\nabla f_{\mathcal{I}_k}(x_{k}) = \frac{1}{|\mathcal{I}_k|} \sum_{i\in\mathcal{I}_k }\nabla f_i(x_t)$.
It is easy to check $\mathbb{E}[v_k|x_0] = \nabla f(x_k)$,
\emph{i.e.,} an unbiased estimate gradient over $x_k$.
Comparing the existing SVRG-ADMM,
our SPIDER-ADMM constructs stochastic gradient $v_k$ based on
the information $x_{k-1}$ and $v_{k-1}$, while the SVRG-ADMM constructs $v_k$ based on
the information $x_{0}$ and $v_{0}$ (\emph{i.e.,} the initalization information of each outer loop).
Due to using more fresh information, thus, SPIDER-ADMM
can yield more accurate estimation of the full gradient than SVRG-ADMM.
Simultaneously, it does not require to additional computation and memory,
so it costs less memory than the existing SAGA-ADMM.
\begin{algorithm}[htb]
   \caption{ SPIDER-ADMM Algorithm }
   \label{alg:1}
\begin{algorithmic}[1]
   \STATE {\bfseries Input:} $b$, $q$, $K$, $\eta>0$ and $\rho>0 $;
   \STATE {\bfseries Initialize:} $x_0 \in \mathbb{R}^d$, $y^0_j \in \mathbb{R}^p, \ j\in [m]$ and $z_0 \in \mathbb{R}^l$;
   \FOR {$k=0,1,\cdots,K-1$}
   \IF {$\mbox{mod}(k,q) = 0$}
   \STATE{} Compute $v_k = \nabla f(x_k)$;
   \ELSE
   \STATE{} Uniformly randomly pick a mini-batch $\mathcal{I}_k$ (with replacement) from $\{1,2,\cdots,n\}$ with $|\mathcal{I}_k|=b$,
            and compute $$ v_{k} = \nabla f_{\mathcal{I}_k}(x_k) - \nabla f_{\mathcal{I}_k}(x_{k-1}) + v_{k-1};$$
   \ENDIF
   \STATE{}  $ y^{k+1}_j= \arg\min_{y_j} \big\{ \mathcal {L}_{\rho}(x_k,y^{k+1}_{[j-1]},y_j,y^{k}_{[j+1:m]},z_k) + \frac{1}{2}\|y_j-y^k_j\|_{H_j}^2 \big\}$ for all $j\in [m]$;
   \STATE{}  $ x_{k+1}= \arg\min_x \hat{\mathcal {L}}_{\rho}\big( x,y^{k+1}_{[m]},z_k,v_k \big)$;
   \STATE{}  $z_{k+1} = z_k- \rho(Ax_{k+1} + \sum_{j=1}^mB_jy^{k+1}_j - c)$;
   \ENDFOR
   \STATE {\bfseries Output \ (in theory):} Chosen uniformly random from $\{x_{k},y_{[m]}^{k},z_k\}_{k=1}^{K}$.
   \STATE {\bfseries Output \ (in practice):} $\{x_{K},y_{[m]}^{K},z_K\}$.
\end{algorithmic}
\end{algorithm}

\section{ Fast Online SPIDER-ADMM Method }
In the section, we propose an online SPIDER-ADMM to
solve the online problem \eqref{eq:2}, which is equivalent to the following stochastic constrained problem:
\begin{align} \label{eq:9}
\min_{x,\{y_j\}_{j=1}^m} & \mathbb{E}_{\zeta} [ f(x;\zeta) ] + \sum_{j=1}^m g_j(y_j),  \nonumber \\
 \mbox{s.t.} & \ Ax + \sum_{j=1}^mB_jy_j =c,
\end{align}
where $f(x)=\mathbb{E}_{\zeta} [ f(x;\zeta) ]$ denotes a population risk
over an underlying data distribution. The problem \eqref{eq:9} can be viewed as
having infinite samples, so we cannot evaluate the full gradient $\nabla f(x)$.
For solving the problem \eqref{eq:9}, so we use stochastic sampling to evaluate the full gradient.
Algorithm \ref{alg:2} shows the algorithmic framework of online SPIDER-ADMM method.
In Algorithm \ref{alg:2}, we use the mini-batch samples to estimate
the full gradient, and update the variables $\{x,y_j,z\}_{j=1}^m$, which is the same as
in Algorithm \ref{alg:1}.

\begin{algorithm}[htb]
   \caption{ Online SPIDER-ADMM Algorithm }
   \label{alg:2}
\begin{algorithmic}[1]
   \STATE {\bfseries Input:} $b_1$, $b_2$, $q$, $K$, $\eta>0$ and $\rho>0$;
   \STATE {\bfseries Initialize:} $x_0 \in \mathbb{R}^d$, $y^0_j \in \mathbb{R}^p, \ j\in [m]$ and $z_0 \in \mathbb{R}^l$;
   \FOR {$k=0,1,\cdots,K-1$}
   \IF {$\mbox{mod}(k,q) = 0$}
   \STATE{} Draw $S_1$ samples with $|S_1|=b_1$, and compute $v_k = \frac{1}{b_1}\sum_{i\in S_1}\nabla f_i(x_k)$;
   \ELSE
   \STATE{} Draw $S_2$ samples with $|S_2|=b_2=\sqrt{b_1}$, and compute
            $$v_k = \frac{1}{b_2} \sum_{i\in S_2}\big(\nabla f_i(x_k) - f_i(x_{k-1})\big) + v_{k-1};$$
   \ENDIF
   \STATE{}  $ y^{k+1}_j= \arg\min_{y_j} \big\{ \mathcal {L}_{\rho}(x_k,y^{k+1}_{[j-1]},y_j,y^{k}_{[j+1:m]},z_k) + \frac{1}{2}\|y_j-y^k_j\|_{H_j}^2 \big\} $ for all $j\in [m]$;
   \STATE{}  $ x_{k+1}= \arg\min_x \hat{\mathcal {L}}_{\rho}\big( x,y^{k+1}_{[m]},z_k,v_k \big)$;
   \STATE{}  $z_{k+1} = z_k- \rho(Ax_{k+1} + \sum_{j=1}^mB_jy^{k+1}_j - c)$;
   \ENDFOR
   \STATE {\bfseries Output \ (in theory):} Chosen uniformly random from $\{x_{k},y_{[m]}^{k},z_k\}_{k=1}^{K}$.
   \STATE {\bfseries Output \ (in practice):} $\{x_{K},y_{[m]}^{K},z_K\}$.
\end{algorithmic}
\end{algorithm}

\section{Convergence Analysis}
In the section, we study the convergence properties of both the SPIDER-ADMM and online SPIDER-ADMM.
At the same time, based on our new theoretical analysis framework, we afresh analyze the convergence properties
of existing ADMM-based nonconvex optimization algorithms,
\emph{i.e.,} SVRG-ADMM and SAGA-ADMM, and derive their optimal
IFO complexity for finding an $\epsilon$-approximate stationary point.
\subsection{Convergence Analysis of SPIDER-ADMM}
In the subsection, we study convergence properties of the SPIDER-ADMM algorithm.
The detailed proofs are provided in the Appendix \ref{Appendix:A1}.
Throughout the paper, let $n_k = \lceil k/q \rceil$ such that $(n_k-1)q \leq k \leq n_kq-1$.

\begin{lemma}
 Suppose the sequence $\{x_k,y_{[m]}^k,z_k\}_{k=1}^K$ is generated from Algorithm \ref{alg:1},
 and define a \emph{Lyapunov} function $R_k$ as follows:
 \begin{align}
 R_k \!=\! & \mathcal{L}_{\rho} (x_k,y_{[m]}^k,z_k) \!+\! (\frac{9L^2}{\sigma^A_{\min}\rho}\!+\!\frac{3\sigma^2_{\max}(G)}{\sigma^A_{\min}\eta^2\rho})\|x_k\!-\!x_{k-1}\|^2  \nonumber \\
 & \!+\! \frac{2L^2}{\sigma^A_{\min}\rho b} \sum_{i=(n_k-1)q}^{k-1}\mathbb{E}\|x_{i+1}\!-\!x_i\|^2. \nonumber
\end{align}
Let $b=q$, $\eta=\frac{2\alpha\sigma_{\min}(G)}{3L} \ (0<\alpha \leq 1)$ and $\rho = \frac{\sqrt{170}\kappa_GL}{\sigma^A_{\min}\alpha}$,
then we have
\begin{align}
 \frac{1}{K}\sum_{k=0}^{K-1} \big(\|x_{k+1} - x_{k}\|^2 + \sum_{j=1}^m \|y_j^k-y_j^{k+1}\|^2\big) \leq \frac{R_{0}-R^*}{K\gamma}, \nonumber
\end{align}
where $\gamma = \min(\chi,\sigma_{\min}^H)$ with $\chi\geq \frac{\sqrt{170}\kappa_GL}{4\alpha}$ and $R^*$ is a lower bound of the function $R_k$.
\end{lemma}

Let $\theta_k = \mathbb{E} [\|x_{k+1}-x_{k}\|^2 + \|x_{k}-x_{k-1}\|^2 + \frac{1}{q}\sum_{i=(n_k-1)q}^k \|x_{i+1}-x_i\|^2 + \sum_{j=1}^m \|y_j^k-y_j^{k+1}\|^2 ]$.
Next, based on the above lemma, we give the convergence properties of SPIDER-ADMM.

\begin{theorem} \label{th:1}
 Suppose the sequence $\{x_k,y_{[m]}^k,z_k)_{k=1}^K$ is generated from Algorithm \ref{alg:1}. Let
  \begin{align}
   &\nu_1 = m\big(\rho^2\sigma^B_{\max}\sigma^A_{\max} + \rho^2(\sigma^B_{\max})^2 + \sigma^2_{\max}(H)\big), \nonumber \\
   &\nu_2 = 3(L^2+ \frac{\sigma^2_{\max}(G)}{\eta^2}),\nu_3 = \frac{18 L^2 }{\sigma^A_{\min} \rho^2} + \frac{3\sigma^2_{\max}(G) }{\sigma^A_{\min}\eta^2\rho^2}, \nonumber
 \end{align}
 and let $b=q$, $\eta = \frac{2\alpha\sigma_{\min}(G)}{3L} \ (0<\alpha \leq 1)$, and
 $\rho = \frac{\sqrt{170}\kappa_GL}{\sigma^A_{\min}\alpha}$,
 then we have
 \begin{align}
  \frac{1}{K} \sum_{k=1}^K \mathbb{E}\big[ \mbox{dist}(0,\partial L(x_k,y_{[m]}^k,z_k))^2\big] &\leq\frac{\nu_{\max}}{K}\sum_{k=1}^{K-1} \theta_k \nonumber\\
  &\leq  \frac{3\nu_{\max}(R_{0}-R^*)}{K\gamma}, \nonumber
 \end{align}
 where $\gamma = \min(\chi,\sigma_{\min}^H)$ with $\chi\geq \frac{\sqrt{170}\kappa_GL}{4\alpha}$, $\nu_{\max}=\max\{\nu_1,\nu_2,\nu_3\}$ and
 $R^*$ is a lower bound of the function $R_k$.
 It implies that the iteration number $K$ satisfies
 \begin{align}
  K = \frac{3\nu_{\max}(R_{0}-R^*)}{\epsilon \gamma}, \nonumber
 \end{align}
 then $(x_{k^*},y_{[m]}^{k^*},z_{k^*})$ is an $\epsilon$-approximate stationary point of \eqref{eq:2}, where $k^* = \mathop{\arg\min}_{k}\theta_{k}$.
\end{theorem}

\begin{remark}
 Theorem \ref{th:1} shows that the SPIDER-ADMM has $O(1/K)$ convergence rate. Moreover, given $b=q=\sqrt{n}$, $\eta = \frac{2\alpha\sigma_{\min}(G)}{3L} \ (0<\alpha \leq 1)$ and
 $\rho = \frac{\sqrt{170}\kappa_GL}{\sigma^A_{\min}\alpha}$, the SPIDER-ADMM has the optimal
 IFO of $\mathcal{O}(n+n^{\frac{1}{2}}\epsilon^{-1})$ for finding an $\epsilon$-approximate stationary point. In particular,
 we can choose $\alpha \in (0,1]$ according to different problems to obtain appropriate step-size $\eta$ and penalty parameter $\rho$, \emph{e.g.}, set $\alpha=1$,
 we have $\eta=\frac{2\sigma_{\min}(G)}{3L}$ and $\rho=\frac{\sqrt{170}\kappa_GL}{\sigma^A_{\min}}$.
\end{remark}
\subsection{Convergence Analysis of Online SPIDER-ADMM}
In the subsection, we study convergence properties of the online SPIDER-ADMM algorithm.
The detailed proofs are provided in the Appendix \ref{Appendix:A2}.

\begin{lemma}
 Suppose the sequence $\{x_k,y_{[m]}^k,z_k\}_{k=1}^K$ is generated from Algorithm \ref{alg:2},
 and define a \emph{Lyapunov} function $\Phi_k$ as follows:
 \begin{align}
 \Phi_k = & \mathcal{L}_{\rho} (x_k,y_{[m]}^k,z_k) + (\frac{9L^2}{\sigma^A_{\min}\rho}+\frac{3\sigma^2_{\max}(G)}{\sigma^A_{\min}\eta^2\rho})\|x_k-x_{k-1}\|^2 \nonumber \\
 & + \frac{2L^2}{\sigma^A_{\min}\rho b_2} \sum_{i=(n_k-1)q}^{k-1}\mathbb{E}\|x_{i+1}-x_i\|^2. \nonumber
\end{align}
Let $b_2=q$, $\eta=\frac{2\alpha\sigma_{\min}(G)}{3L} \ (0<\alpha \leq 1)$ and $\rho= \frac{\sqrt{170}\kappa_GL}{\sigma^A_{\min}\alpha}$,
then we have
\begin{align}
& \frac{1}{K}\sum_{k=0}^{K-1} \big(\|x_{k+1} - x_{k}\|^2 + \sum_{j=1}^m \|y_j^k-y_j^{k+1}\|^2\big) \nonumber \\
 &\leq \frac{\Phi_{0}-\Phi^*}{K\gamma} + \frac{2\delta^2}{b_1L\gamma} + \frac{72\delta^2}{\sigma^A_{\min}b_1\rho\gamma}, \nonumber
\end{align}
where $\gamma = \min(\chi,\sigma_{\min}^H)$ with $\chi\geq \frac{\sqrt{170}\kappa_GL}{4\alpha}$ and $\Phi^*$ is a lower bound of the function $\Phi_k$.
\end{lemma}

Let $\theta_k = \mathbb{E} [ \|x_{k+1}-x_{k}\|^2+\|x_{k}-x_{k-1}\|^2+\frac{1}{q}\sum_{i=(n_k-1)q}^k \|x_{i+1}-x_i\|^2+ \sum_{j=1}^m \|y_j^k-y_j^{k+1}\|^2 ]$.
\begin{theorem} \label{th:2}
 Suppose the sequence $\{x_k,y_{[m]}^k,z_k)_{k=1}^K$ is generated from Algorithm \ref{alg:2}. Let
  \begin{align}
   &\nu_1 = m\big(\rho^2\sigma^B_{\max}\sigma^A_{\max} + \rho^2(\sigma^B_{\max})^2 + \sigma^2_{\max}(H)\big), \nonumber \\
   &\nu_2 = 3(L^2+ \frac{\sigma^2_{\max}(G)}{\eta^2}),\nu_3 = \frac{18 L^2 }{\sigma^A_{\min} \rho^2} + \frac{3\sigma^2_{\max}(G) }{\sigma^A_{\min}\eta^2\rho^2}, \nonumber
 \end{align}
 and let $b_2=q=\sqrt{b_1}$, $\eta = \frac{2\alpha\sigma_{\min}(G)}{3L} \ (0<\alpha \leq 1)$ and
 $\rho = \frac{\sqrt{170}\kappa_GL}{\sigma^A_{\min}\alpha}$,
 then we have
  \begin{align}
  &\frac{1}{K} \sum_{k=1}^K \mathbb{E} \big[ \mbox{dist}(0,\partial L(x_k,y_{[m]}^k,z_k))^2\big]\leq \frac{\nu_{\max}}{K}\sum_{k=1}^{K-1} \theta_k
   + \frac{w}{b_1} \nonumber \\
  & \leq  \frac{3\nu_{\max}(\Phi_{0}-\Phi^*)}{K\gamma} + \frac{6\nu_{\max}\delta^2}{b_1\gamma}(\frac{1}{L} + \frac{36}{\sigma^A_{\min}\rho})+ \frac{w}{b_1}, \nonumber
\end{align}
 where $w= 12\delta^2\max\{1, \frac{6}{\sigma^A_{\min}\rho^2}\}$, $\gamma = \min(\chi,\sigma_{\min}^H)$ with $\chi\geq \frac{\sqrt{170}\kappa_GL}{4\alpha}$,
 $\nu_{\max}=\max\{\nu_1,\nu_2,\nu_3\}$ and
 $\Phi^*$ is a lower bound of the function $\Phi_k$.
 It implies that $K$ and $b_1$ satisfy
 \begin{align}
  K \!=\! \frac{6\nu_{\max}(\Phi_{0}-\Phi^*)}{\epsilon \gamma}, \ b_1 \!=\! \frac{12\nu_{\max}\delta^2}{\epsilon\gamma}(\frac{1}{L}+\frac{36}{\sigma^A_{\min}\rho}) \!+\! \frac{2w}{\epsilon}, \nonumber
 \end{align}
 then $(x_{k^*},y_{[m]}^{k^*},z_{k^*})$ is an $\epsilon$-approximate stationary point of \eqref{eq:2}, where $k^* = \mathop{\arg\min}_{k}\theta_{k}$.
\end{theorem}
\begin{remark}
 Theorem \ref{th:2} shows that given $b_2=q=\sqrt{b_1}$, $\eta = \frac{2\alpha\sigma_{\min}(G)}{3L} \ (0<\alpha \leq 1)$,
 $\rho = \frac{\sqrt{170}\kappa_GL}{\sigma^A_{\min}\alpha}$ and $b_1=\mathcal{O}(\epsilon^{-1})$, the online SPIDER-ADMM has the optimal
 IFO of $\mathcal{O}(\epsilon^{-\frac{3}{2}})$ for finding an $\epsilon$-approximate stationary point.
\end{remark}

\subsection{Convergence Analysis of Non-convex SVRG-ADMM}
In the subsection, we extend the existing nonconvex SVRG-ADMM method \cite{huang2016stochastic,zheng2016stochastic}
to the multiple variables setting for solving the problem \eqref{eq:2}.
The SVRG-ADMM algorithm is described in Algorithm 3 given in the Appendix \ref{Appendix:A3}.
Next, we analyze convergence properties of the SVRG-ADMM algorithm,
and derive its optimal IFO complexity.

\begin{lemma}
Suppose the sequence $\{(x^{s}_t,y_{[m]}^{s,t},z^{s}_t)_{t=1}^M\}_{s=1}^S$ is generated from Algorithm \ref{alg:3},
 and define a \emph{Lyapunov} function:
 \begin{align}
 \Gamma^s_t\!=& \mathbb{E}\big[\mathcal{L}_{\rho} (x^s_t,y_{[m]}^{s,t},z^s_t) \!+\! (\frac{3\sigma^2_{\max}(G)}{\sigma^A_{\min}\eta^2\rho} \!+\! \frac{9L^2}{\sigma^A_{\min}\rho})\|x^s_{t}-x^s_{t-1}\|^2 \nonumber \\
 &+ \frac{9L^2 }{\sigma^A_{\min}\rho b}\|x^s_{t-1}-\tilde{x}^s\|^2 + c_t\|x^s_{t}-\tilde{x}^s\|^2\big], \nonumber
 \end{align}
 where the positive sequence $\{c_t\}$ satisfies, for $s =1,2,\cdots,S$
 \begin{equation*}
  c_t= \left\{
  \begin{aligned}
  & \frac{18 L^2 }{\sigma^A_{\min}\rho b} +
     \frac{L}{b} + (1+\beta)c_{t+1}, \ 1 \leq t \leq M, \\
  & 0, \ t \geq M+1.
  \end{aligned}
  \right.\end{equation*}
Let $M=n^{\frac{1}{3}}$, $b=n^{\frac{2}{3}}$, $\eta = \frac{\alpha\sigma_{\min}(G)}{5L} \ (0< \alpha \leq 1)$ and
$\rho = \frac{2\sqrt{231}\kappa_G L}{\sigma^A_{\min}\alpha}$, we have
\begin{align}
\frac{1}{T}\sum_{s=1}^S \sum_{t=0}^{M-1} & \big(\sigma_{\min}^H\sum_{j=1}^m \|y_j^{s,t}-y_j^{s,t+1}\|^2 + \chi_t \|x^s_{t+1}-x^s_t\|^2 \nonumber \\
& + \frac{L}{2b}\|x^s_t-\tilde{x}^s\|^2 \big) \leq \frac{\Gamma^1_0 - \Gamma^*}{T}  .
\end{align}
where $T=MS$, $\chi_t \geq \frac{\sqrt{231}\kappa_G L}{2\alpha} > 0$ and $\Gamma^*$ denotes a lower bound of function $\Gamma^s_t$.
\end{lemma}

Let $\theta^s_{t} = \mathbb{E} [ \|x^s_{t+1}-x^s_{t}\|^2 + \|x^s_{t}-x^s_{t-1}\|^2 + \frac{1}{b}(\|x^s_{t}-\tilde{x}^s\|^2 + \|x^s_{t-1}-\tilde{x}^s\|^2 ) + \sum_{j=1}^m\|y_j^{s,t} - y_j^{s,t+1}\|^2 ]$.
\begin{theorem} \label{th:3}
 Suppose the sequence $\{(x^{s}_t,y_{[m]}^{s,t},z^{s}_t)_{t=1}^M\}_{s=1}^S$ is generated from Algorithm \ref{alg:3} . Let
 \begin{align}
 & \nu_1 = m\big(\rho^2\sigma^B_{\max}\sigma^A_{\max} + \rho^2(\sigma^B_{\max})^2 + \sigma^2_{\max}(H)\big), \nonumber \\
 & \nu_2 = 3L^2 + \frac{3\sigma^2_{\max}(G)}{\eta^2},
 \ \nu_3 = \frac{9L^2 }{\sigma^A_{\min}\rho^2} + \frac{3\sigma^2_{\max}(G)}{\sigma^A_{\min}\eta^2\rho^2}, \nonumber
\end{align}
 and given $M=n^{\frac{1}{3}}$, $b=n^{\frac{2}{3}}$, $\eta = \frac{\alpha\sigma_{\min}(G)}{5L} \ (0 < \alpha \leq 1)$
 and $\rho = \frac{2\sqrt{231}\kappa_G L}{\sigma^A_{\min}\alpha}$,
 then we have
 \begin{align}
\frac{1}{T}\sum_{s=1}^S \sum_{t=0}^{M-1}\mathbb{E}\big[ \mbox{dist}(0,\partial L(x^s_t,y_{[m]}^{s,t},z^s_t))^2\big] \leq \frac{2\nu_{\max}(\Gamma^1_0 - \Gamma^*)}{\gamma T}, \nonumber
\end{align}
where $\gamma = \min(\sigma_{\min}^H,\frac{L}{2},\chi_t)$, $\nu_{\max}= \max(\nu_1,\nu_2,\nu_3)$ and
$\Gamma^*$ is a lower bound of function $\Gamma^s_t$.
It implies that the whole iteration number $T=MS$ satisfies
 \begin{align}
 T = \frac{2\nu_{\max}(\Gamma^1_0 - \Gamma^*)}{\epsilon \gamma}, \nonumber
 \end{align}
then $(x^{s^*}_{t^*},y_{[m]}^{s^*,t^*},z^{s^*}_{t^*})$ is an $\epsilon$-stationary point of \eqref{eq:2}, where $(t^*,s^*) = \mathop{\arg\min}_{t,s}\theta^s_{t}$.
\end{theorem}
\begin{remark}
 Theorem \ref{th:3} shows that given $M=n^{\frac{1}{3}}$, $b=n^{\frac{2}{3}}$, $\eta = \frac{\alpha\sigma_{\min}(G)}{5L} \ (0<\alpha \leq 1)$ and
 $\rho = \frac{2\sqrt{231}\kappa_GL}{\sigma^A_{\min}\alpha}$, the non-convex SVRG-ADMM has the optimal
 IFO complexity of $\mathcal{O}(n+n^{\frac{2}{3}}\epsilon^{-1})$ for finding an $\epsilon$-approximate stationary point.
\end{remark}
\subsection{Convergence Analysis of Non-convex SAGA-ADMM}
In the subsection, we extend the existing nonconvex SAGA-ADMM method \cite{huang2016stochastic}
to the multiple variables setting for solving the problem \eqref{eq:2}.
The SAGA-ADMM algorithm is described in Algorithm 4 given in the Appendix \ref{Appendix:A4}.
Next, we analyze convergence properties of non-convex SAGA-ADMM,
and derive its the optimal IFO complexity.

\begin{lemma}
 Suppose the sequence $\{x_t,y_{[m]}^t,z_t\}_{t=1}^T$ is generated from Algorithm \ref{alg:4},
 and define a \emph{Lyapunov} function
 \begin{align}
 & \Omega_t \!= \!\mathbb{E}\big[ \mathcal{L}_{\rho} (x_t,y_{[m]}^t,z_t) \!+ \! (\frac{3\sigma^2_{\max}(G)}{\sigma^A_{\min}\eta^2\rho} \!+ \! \frac{9L^2}{\sigma^A_{\min}\rho}) \|x_{t}\!-\!x_{t-1}\|^2 \nonumber \\
 & + \frac{9L^2}{\sigma^A_{\min}\rho b}\frac{1}{n} \sum_{i=1}^n\|x_{t-1}-u^{t-1}_i\|^2 + c_t\frac{1}{n} \sum_{i=1}^n\|x_{t}-u^t_i\|^2 \big], \nonumber
\end{align}
 where the positive sequence $\{c_t\}$ satisfies
 \begin{equation*}
  c_t \!=\! \left\{
  \begin{aligned}
  & \frac{18L^2 }{\sigma^A_{\min}\rho b} \!+\! \frac{L}{b} \!+\! (1-p)(1+\beta)c_{t+1}, \ 0 \leq t \leq T-1, \\
  & 0, \ t \geq T,
  \end{aligned}
  \right.\end{equation*}
where $p$ denotes probability of an index $i$ being in $\mathcal{I}_t$.
Further, let $b= n^{\frac{2}{3}}$, $\eta = \frac{\alpha\sigma_{\min}(G)}{17L} \ (0 < \alpha \leq 1)$ and $\rho = \frac{2\sqrt{2031}\kappa_G}{\sigma^A_{\min}\alpha}$
we have
\begin{align}
 \frac{1}{T} \sum_{t=1}^T & \big( \sigma_{\min}^H\sum_{j=1}^m \|y_j^t-y_j^{t+1}\|^2 + \chi_t \|x_t-x_{t+1}\|^2 \nonumber \\
 &+  \frac{L}{2b}\frac{1}{n}\sum_{i=1}^n\|x_t-u^t_i\|^2 \big) \leq \frac{\Omega_0 - \Omega^*}{T}, \nonumber
\end{align}
where $\chi_t \geq \frac{\sqrt{2031}\kappa_GL}{2\alpha} >0$ and $\Omega^*$ denotes a lower bound of function $\Omega_t$.
\end{lemma}

Let $\theta_t =  \mathbb{E} [ \|x_{t+1}-x_t\|^2 + \|x_t-x_{t-1}\|^2 + \frac{1}{b n}\sum^n_{i=1} (\|x_t-u^t_i\|^2 + \|x_{t-1}-u^{t-1}_i\|^2)+\sum_{j=1}^m \|y_j^t-y_j^{t+1}\|^2 ]$.

\begin{theorem} \label{th:4}
Suppose the sequence $\{x_t,y_{[m]}^t,z_t\}_{t=1}^T$ is generated from Algorithm \ref{alg:4}. Let
\begin{align}
  &\nu_1 =m\big(\rho^2\sigma^B_{\max}\sigma^A_{\max} + \rho^2(\sigma^B_{\max})^2 + \sigma^2_{\max}(H)\big), \nonumber \\
  & \nu_2 = 3L^2 + \frac{3\sigma^2_{\max}(G)}{\eta^2}, \ \nu_3 = \frac{9L^2 }{\sigma^A_{\min}\rho^2} + \frac{3\sigma^2_{\max}(G)}{\sigma^A_{\min}\eta^2\rho^2}, \nonumber
 \end{align}
 and given $b= n^{\frac{2}{3}}$, $\eta = \frac{\alpha\sigma_{\min}(G)}{17L} \ (0 < \alpha \leq 1)$ and $\rho = \frac{2\sqrt{2031}\kappa_G}{\sigma^A_{\min}\alpha}$,
 then we have
 \begin{align}
 \frac{1}{T}\sum_{t=1}^T \mathbb{E}\big[ \mbox{dist}(0,\partial L(x_t,y_{[m]}^t,z_t))^2\big] \leq \frac{2\nu_{\max}(\Omega_0 - \Omega^*)}{\gamma T},  \nonumber
 \end{align}
 where $\gamma = \min(\sigma_{\min}^H, \frac{L}{2}, \chi_{t})$ with $\chi_t \geq \frac{\sqrt{2031}\kappa_GL}{2\alpha} >0$, $\nu_{\max}= \max(\nu_1,\nu_2,\nu_3)$ and
 $\Omega^*$ is a lower bound of function $\Omega_t$. It implies that the iteration number $T$ satisfies
 \begin{align}
  T = \frac{2\nu_{\max}}{\epsilon \gamma}(\Omega_0- \Omega^*),  \nonumber
 \end{align}
 then $(x_{t^*},y_{[m]}^{t^*},z_{t^*})$ is an $\epsilon$-approximate stationary point of \eqref{eq:2}, where $t^* = \mathop{\arg\min}_{ 1\leq t\leq T}\theta_{t}$.
\end{theorem}
\begin{remark}
 Theorem \ref{th:4} shows that given $b=n^{\frac{2}{3}}$, $\eta = \frac{\alpha\sigma_{\min}(G)}{17L} \ (0<\alpha \leq 1)$ and
 $\rho = \frac{2\sqrt{2031}\kappa_GL}{\sigma^A_{\min}\alpha}$, the non-convex SAGA-ADMM has the optimal
 IFO of $\mathcal{O}(n+n^{\frac{2}{3}}\epsilon^{-1})$ for finding an $\epsilon$-approximate stationary point.
\end{remark}

\begin{remark}
\textbf{Our contributions} on convergence analysis of both the non-convex SVRG-ADMM and SAGA-ADMM are given as follows:
\begin{itemize}
\item We extend both the existing non-convex SVRG-ADMM and SAGA-ADMM to the multi-block setting
      for solving the problem \eqref{eq:2};
\item We not only give its optimal IFO complexity of $\mathcal{O}(n+n^{\frac{2}{3}}\epsilon^{-1})$,
      but also provide the \emph{specific and simple} choice on the step-size $\eta$ and penalty parameter $\rho$.
\end{itemize}
\end{remark}
% \emph{All related proofs are provided in the supplementary document.}
\begin{table}
  \centering
  \caption{Real datasets } \label{tab:2}
  \begin{tabular}{c|c|c|c}
  \hline
  datasets & $\#samples$ & $\#features$ & $\#classes$ \\ \hline
  \emph{a9a}   & 32,561  & 123 & 2\\
  \emph{w8a} & 64,700  & 300 & 2\\
  \emph{ijcnn1} & 126,702 & 22 & 2 \\
  \emph{covtype.binary} & 581,012 & 54 & 2\\ \hline
  \emph{letter} & 15,000 & 16 & 26 \\
  \emph{sensorless} & 58,509 &48 & 11 \\
  \emph{mnist} & 60,000 & 780 & 10 \\
  \emph{covtype} & 581,012  & 54 & 7 \\
  \hline
  \end{tabular}
\end{table}

\section{ Experiments }
In this section, we will compare the proposed algorithm (SPIDER-ADMM) with the existing non-convex algorithms \big( nc-ADMM \cite{jiang2019structured},
nc-SVRG-ADMM \cite{huang2016stochastic,zheng2016stochastic}, nc-SAGA-ADMM \cite{huang2016stochastic} and nc-SADMM \cite{huang2018mini} \big) on two applications:
1) Graph-guided binary classification; 2) Multi-task learning.
In the experiment, we use some publicly available datasets\footnote{
These data are from the LIBSVM website (www.csie.ntu.edu.tw/~cjlin/libsvmtools/datasets/).},
which are summarized in Table \ref{tab:2}. All algorithms are implemented in MATLAB,
and all experiments are performed on a PC with an Intel i7-4790 CPU and 16GB memory.
\begin{figure}[htbp]
\centering
\subfigure[\emph{a9a}]{\includegraphics[width=0.23\textwidth]{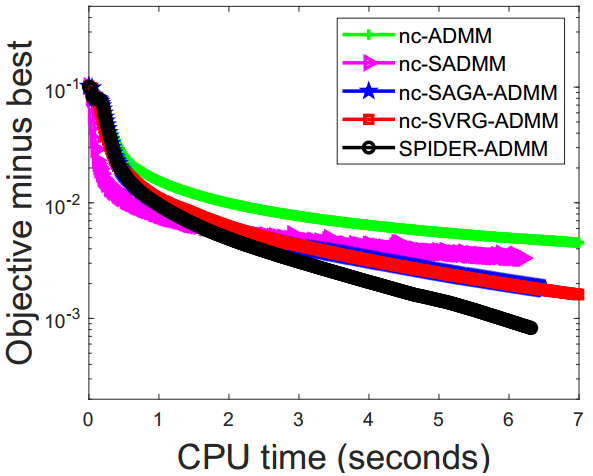}}
\subfigure[\emph{w8a}]{\includegraphics[width=0.23\textwidth]{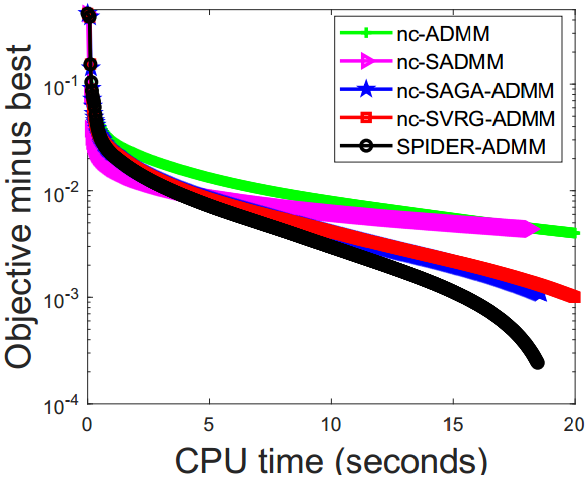}}
\subfigure[\emph{ijcnn1}]{\includegraphics[width=0.23\textwidth]{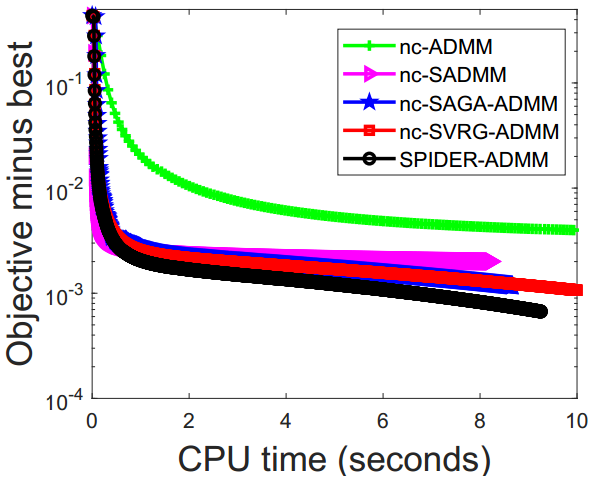}}
\subfigure[\emph{covtype.binary}]{\includegraphics[width=0.23\textwidth]{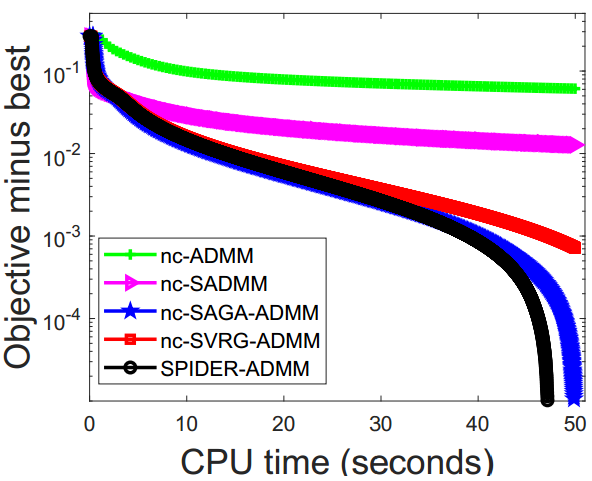}}
\caption{Objective value \emph{versus} CPU time of the \emph{nonconvex} graph-guided binary classification model on some real datasets.} \label{fig:1}
\vspace{-1em}
\end{figure}
\subsection{ Graph-Guided Binary Classification }
In the subsection, we focus on the binary classification task.
Specifically, given a set of training samples $(a_i,b_i)_{i=1}^n$,
where $a_i\in \mathbb{R}^d$, $b_i \in \{-1,1\}$,
then we solve the following nonconvex empirical loss minimization problem:
\begin{align} \label{eq:80}
 \min_{x\in \mathbb{R}^d} \ \frac{1}{n}\sum_{i=1}^n f_i(x) + \lambda\|Ax\|_1,
\end{align}
where $f_i(x)=\frac{1}{1+\exp(b_i a_i^Tx)}$ is the nonconvex sigmoid loss function.
We use the nonsmooth regularizer \emph{i.e.,}
graph-guided fused lasso \cite{kim2009multivariate}, and $A$ decodes the sparsity pattern of graph,
which is obtained by sparse precision matrix estimation \cite{friedman2008sparse}.
To solve the problem \eqref{eq:80},
we give an auxiliary variable $y$ with the constraint $y = Ax$.
In the experiment, we fix the parameter $\lambda=10^{-5}$, and use the same initial solution $x_0$
from the standard normal distribution for all algorithms.

Figure \ref{fig:1} shows that the objective values of our SPIDER-ADMM method faster decrease than those of other methods,
as CPU time consumed increases.
Thus, these results demonstrate that our method has a relatively faster convergence rate than other methods.

\begin{figure}[htbp]
\centering
\subfigure[\emph{letter}]{\includegraphics[width=0.23\textwidth]{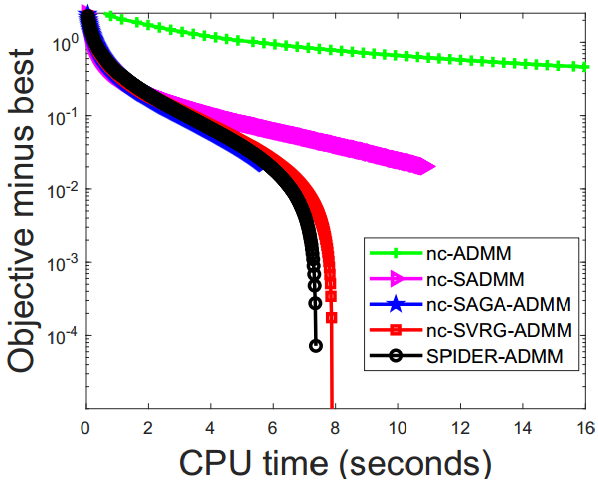}}
\subfigure[\emph{sensorless}]{\includegraphics[width=0.23\textwidth]{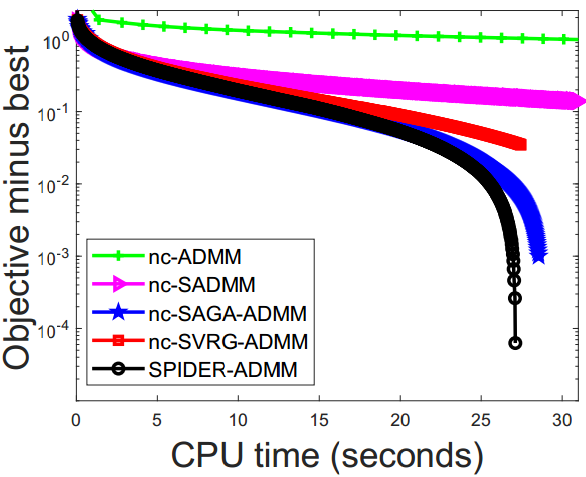}}
\subfigure[\emph{mnist}]{\includegraphics[width=0.23\textwidth]{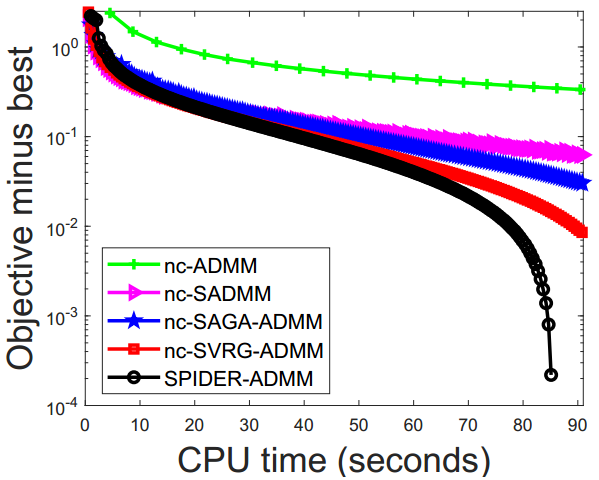}}
\subfigure[\emph{covtype}]{\includegraphics[width=0.23\textwidth]{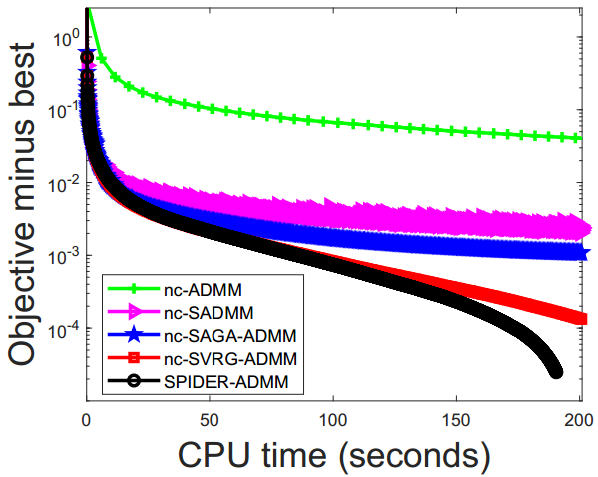}}
\caption{Objective value \emph{versus} CPU time of the \emph{nonconvex} multi-task learning on some real datasets.} \label{fig:2}
\vspace{-1em}
\end{figure}
\subsection{ Multi-Task Learning }
In this subsection, we focus on the multi-task learning task with sparse and low-rank structures.
Specifically, given a set of training samples $(a_i,b_i)_{i=1}^n$,
where $a_i\in \mathbb{R}^d$ and $b_i \in \{1,2,\cdots,c\}$, then let $D\in \mathbb{R}^{n\times c}$ with $D_{ij}=1$ if $j=b_i$,
and $D_{ij}=0$ otherwise.
This multi-task learning is equivalent to solving the following nonconvex problem:
\begin{equation}
 \min_{X \in \mathbb{R}^{c\times d}} \ \frac{1}{n}\sum_{i=1}^n f_i(X) + \lambda_1\sum_{ij}\kappa(|X_{ij}|) + \lambda_2\|X\|_* ,
\end{equation}
where $f_i(X) = \log(\sum_{j=1}^c\exp(X_{j,.}a_i)) - \sum_{j=1}^c D_{ij}X_{j,.}a_i$ is a multinomial logistic loss
function, $\kappa(|X_{ij}|)=\beta\log(1+\frac{|X_{ij}|}{\alpha})$ is the
nonconvex log-sum penalty function \cite{candes2008enhancing}.
Next, we change the above problem into the following form:
\begin{align} \label{eq:82}
 \min & \quad \frac{1}{n}\sum_{i=1}^n \bar{f}_i(X) + \lambda_1\kappa_0\|Y_1\|_1 + \lambda_2\|Y_2\|_* \\
 \mbox{s.t.} & \quad AX + B_1Y_1 + B_2Y_2 = 0, \nonumber
\end{align}
where $\bar{f}_i(X) = f_i(X) + \lambda_1\big( \sum_{ij}\kappa(|X_{ij}|)-\kappa_0\|X\|_1 \big)$, and $\kappa_0=\kappa'(0)$.
Here $A = [I_c;I_c] \in \mathbb{R}^{2c\times c}$, $B_1=[-I_c;0] \in \mathbb{R}^{2c\times c}$ and $B_2=[0;-I]$.
By the Proposition 2.3 in \citet{yao2016efficient},  $\bar{f}_i(X)$ is nonconvex and smooth.
In the experiment, we fix the parameters $\lambda_1 = 10^{-5}$ and $\lambda_2 = 10^{-4}$,
and use the same initial solution $x_0$ from the standard normal distribution for all algorithms.

Figure \ref{fig:2} shows that objective values of our SPIDER-ADMM faster decrease than those of the other methods,
as CPU time consumed increases.
Similarly, these results also demonstrate that
our method has a relatively faster convergence rate than other methods.

\section{Conclusion}
In the paper, we proposed a faster stochastic ADMM method (\emph{i.e.,} SPIDER-ADMM)
for nonconvex optimization. Moreover, we proved that the SPIDER-ADMM achieves a lower IFO complexity of $\mathcal{O}(n+n^{1/2}\epsilon^{-1})$.
Further, we extended the SPIDER-ADMM to the online setting, and
proposed a faster online ADMM method (\emph{i.e.,} online SPIDER-ADMM).
As one of major contribution of this paper, we provided a new theoretical analysis framework for the nonconvex stochastic ADMM methods
with providing an optimal IFO complexity.
Based on our new theoretical analysis framework, we studied the unsolved optimal IFO complexity of
the existing non-convex SVRG-ADMM and SAGA-ADMM methods, and also proved that they reach an IFO complexity of $\mathcal{O}(n+n^{2/3}\epsilon^{-1})$.
In the future work, we can apply the stage-wise stochastic momentum technique \cite{chen2018universal} to accelerate our algorithms.

\section*{Acknowledgments}
We thank the anonymous reviewers for their helpful comments.
F.H. and H.H. were partially supported by U.S. NSF IIS 1836945,  IIS 1836938,  DBI 1836866,  IIS 1845666,  IIS 1852606,  IIS 1838627, IIS 1837956.
S.C. was partially supported by the NSFC under Grant No. 61806093 and No. 61682281,
and the Key Program of NSFC under Grant No. 61732006.

%% Acknowledgements should only appear in the accepted version.
%\section*{Acknowledgements}
%
%\textbf{Do not} include acknowledgements in the initial version of
%the paper submitted for blind review.
%
%If a paper is accepted, the final camera-ready version can (and
%probably should) include acknowledgements. In this case, please
%place such acknowledgements in an unnumbered section at the
%end of the paper. Typically, this will include thanks to reviewers
%who gave useful comments, to colleagues who contributed to the ideas,
%and to funding agencies and corporate sponsors that provided financial
%support.

\nocite{langley00}

\bibliography{reference}
\bibliographystyle{icml2019}

%\appendix
%\section{Do \emph{not} have an appendix here}
%
%\textbf{\emph{Do not put content after the references.}}
%%
%Put anything that you might normally include after the references in a separate
%supplementary file.
%
%We recommend that you build supplementary material in a separate document.
%If you must create one PDF and cut it up, please be careful to use a tool that
%doesn't alter the margins, and that doesn't aggressively rewrite the PDF file.
%pdftk usually works fine.
%
%\textbf{Please do not use Apple's preview to cut off supplementary material.} In
%previous years it has altered margins, and created headaches at the camera-ready
%stage.
%%%%%%%%%%%%%%%%%%%%%%%%%%%%%%%%%%%%%%%%%%%%%%%%%%%%%%%%%%%%%%%%%%%%%%%%%%%%%%%
%%%%%%%%%%%%%%%%%%%%%%%%%%%%%%%%%%%%%%%%%%%%%%%%%%%%%%%%%%%%%%%%%%%%%%%%%%%%%%%

\begin{onecolumn}
%\appendices

\begin{appendices}

\section{ Supplementary Materials }

In this section, we at detail provide the proof of the above lemmas and theorems.
Throughout the paper, let $n_k = [k/q]$ such that $(n_k-1)q \leq k \leq n_kq-1$.
First, we introduce a useful lemma from \citet{fang2018spider}.

\begin{lemma}
 \cite{fang2018spider} Under Assumption 1, the SPIDER generates stochastic gradient $v_k$ satisfies for all $(n_k-1)q+1\leq k \leq n_kq-1$,
 \begin{align} \label{eq:A10}
  \mathbb{E}\|v_k - \nabla f(x_k)\|^2 \leq \frac{L^2}{|S_2|}\mathbb{E}\|x_k-x_{k-1}\|^2 + \mathbb{E}\|v_{k-1}-\nabla f(x_{k-1})\|^2.
 \end{align}
\end{lemma}

From Lemma 1, telescoping \eqref{eq:A10} over $k$ from $(n_k-1)q+1$ to $k$, we have
\begin{align}
 \mathbb{E}\|v_k - \nabla f(x_k)\|^2 \leq \sum_{i=(n_k-1)q}^{k-1}\frac{L^2}{|S_2|}\mathbb{E}\|x_{i+1}-x_i\|^2 + \mathbb{E}\|v_{(n_k-1)q}-\nabla f(x_{(n_k-1)q})\|^2.
\end{align}

In Algorithm \ref{alg:3}, due to $v_{(n_k-1)q}=\nabla f(x_{(n_k-1)q})$ and $|S_2|=b$, we have
\begin{align} \label{eq:A12}
 \mathbb{E}\|v_k - \nabla f(x_k)\|^2 \leq \sum_{i=(n_k-1)q}^{k-1}\frac{L^2}{b}\mathbb{E}\|x_{i+1}-x_i\|^2.
\end{align}

In Algorithm \ref{alg:4}, by Assumption 2 and $|S_2|=b_2$, we have
\begin{align} \label{eq:A13}
 \mathbb{E}\|v_k - \nabla f(x_k)\|^2 \leq \sum_{i=(n_k-1)q}^{k-1}\frac{L^2}{b_2}\mathbb{E}\|x_{i+1}-x_i\|^2 + \frac{\delta^2}{b_1}.
\end{align}

\textbf{Notations:} To make the paper easier to follow, we give
the following notations:
\begin{itemize}
\item $\|\cdot\|$ denotes the vector $\ell_2$ norm and the matrix spectral norm, respectively.
\item $\|x\|_G=\sqrt{x^TGx}$, where $G$ is a positive definite matrix.
\item $\sigma^A_{\min}$ and $\sigma^A_{\max}$ denotes the minimum and maximum eigenvalues of $A^TA$, respectively;
       the conditional number $\kappa_A = \frac{\sigma_{\max}^A}{\sigma_{\min}^A}$.
\item $\sigma^{B_j}_{\max}$ denotes the maximum eigenvalues of $B_j^TB_j$ for all $j\in[k]$, and $\sigma^{B}_{\max} = \max_{j=1}^k \sigma^{B_j}_{\max}$.
\item $\sigma_{\min}(G)$ and $\sigma_{\max}(G)$ denotes the minimum and maximum eigenvalues of matrix $G$, respectively;
      the conditional number $\kappa_G = \frac{\sigma_{\max}(G)}{\sigma_{\min}(G)}$.
\item $\eta$ denotes the step size of updating variable $x$.
\item $L$ denotes the Lipschitz constant of $\nabla f(x)$.
\item $b$ denotes the mini-batch size of stochastic gradient.
\item In both SPIDER-ADMM and online SPIDER-ADMM, $K$ denotes the total number of iteration. In both SVRG-ADMM and SAGA-ADMM,
$T$, $M$ and $S$ are the total number of iterations, the number of iterations in the inner loop,
and the number of iterations in the outer loop, respectively.
\item In SVRG-ADMM algorithm, $y_j^{s,t}$ denotes output of the variable $y_j$ in $t$-th inner loop and $s$-th outer loop.
\end{itemize}

\subsection{ Convergence Analysis of the SPIDER-ADMM }
\label{Appendix:A1}
In this subsection, we conduct convergence analysis of the SPIDER-ADMM. We begin with giving some useful lemmas.

\begin{lemma} \label{lem:A3}
Under Assumption 1 and given the sequence $\{x_k,y_{[m]}^k,z_k\}_{k=1}^K$ from Algorithm \ref{alg:3}, it holds that
\begin{align}
 \mathbb{E}\|z_{k+1} - z_k\|^2 \leq & \frac{18L^2}{\sigma^A_{\min}b} \sum_{i=(n_k-1)q}^{k-1}\mathbb{E}\|x_{i+1}-x_i\|^2
 + (\frac{9L^2}{\sigma^A_{\min}}+\frac{3\sigma^2_{\max}(G)}{\sigma^A_{\min}\eta^2})\|x_k-x_{k-1}\|^2 \nonumber \\
 & + \frac{3\sigma^2_{\max}(G)}{\sigma^A_{\min}\eta^2}\|x_{k+1}-x_{k}\|^2.
\end{align}
\end{lemma}

\begin{proof}
Using the optimal condition of the step 10 in Algorithm \ref{alg:3},
we have
 \begin{align}
   v_k + \frac{G}{\eta}(x_{k+1}-x_k) - A^Tz_k + \rho A^T(Ax_{k+1}+\sum_{j=1}^mB_jy_j^{k+1}-c) = 0.
 \end{align}
Then using the step 11 of Algorithm \ref{alg:3}, we have
 \begin{align} \label{eq:A16}
  A^Tz_{k+1} = v_k + \frac{G}{\eta}(x_{k+1}-x_k).
 \end{align}
It follows that
 \begin{align} \label{eq:A17}
  A^T(z_{k+1}-z_k) = & v_k - v_{k-1} + \frac{G}{\eta}(x_{k+1}-x_k) - \frac{G}{\eta}(x_{k}-x_{k-1}).
 \end{align}
By \eqref{eq:A17}, we have
 \begin{align} \label{eq:A18}
 \mathbb{E}\|z_{k+1}-z_k\|^2 \leq \frac{1}{\sigma^A_{\min}}\big[3 \mathbb{E}\|v_k - v_{k-1}\|^2 +\frac{3\sigma^2_{\max}(G)}{\eta^2}\mathbb{E}\|x_{k+1}-x_k\|^2
  + \frac{3\sigma^2_{\max}(G)}{\eta^2}\mathbb{E}\|x_{k}-x_{k-1}\|^2 \big],
 \end{align}
where the inequality holds by the Jensen's inequality yielding $\|\frac{1}{n}\sum_{i=1}^n z_i\|^2 \leq \frac{1}{n}\sum_{i=1}^n \|z_i\|^2$.

Next, considering the upper bound of $\|v_k - v_{k-1}\|^2$, we have
 \begin{align} \label{eq:A19}
  \mathbb{E}\|v_k - v_{k-1}\|^2 & = \mathbb{E}\|v_k - \nabla f(x_k) + \nabla f(x_k) -  \nabla f(x_{k-1}) + \nabla f(x_{k-1}) - v_{k-1}\|^2 \nonumber \\
  & \leq 3\mathbb{E}\|v_k - \nabla f(x_k)\|^2 + 3\mathbb{E}\|\nabla f(x_k) -  \nabla f(x_{k-1})\|^2 + 3\mathbb{E}\|\nabla f(x_{k-1}) - v_{k-1}\|^2 \nonumber \\
  & \leq \frac{3L^2}{b}\sum_{i=(n_k-1)q}^{k-1}\mathbb{E}\|x_{i+1}-x_i\|^2 + 3L^2\mathbb{E}\|x_{k-1} - x_k\|^2 + \frac{3L^2}{b}\sum_{i=(n_k-1)q}^{k-2}\mathbb{E}\|x_{i+1}-x_i\|^2 \nonumber \\
  & \leq  \frac{6L^2}{b}\sum_{i=(n_k-1)q}^{k-1}\mathbb{E}\|x_{i+1}-x_i\|^2 + 3L^2\mathbb{E}\|x_{k-1} - x_k\|^2,
 \end{align}
where the second inequality holds by Assumption 1 and the inequality \eqref{eq:A12}.

Finally, combining the inequalities \eqref{eq:A18} and \eqref{eq:A19}, we obtain the above result.

\end{proof}

\begin{lemma} \label{lem:A4}
 Suppose the sequence $\{x_k,y_{[m]}^k,z_k\}_{k=1}^K$ is generated from Algorithm \ref{alg:3},
 and define a \emph{Lyapunov} function $R_k$ as follows:
 \begin{align}
 R_k = \mathcal{L}_{\rho} (x_k,y_{[m]}^k,z_k) + (\frac{9L^2}{\sigma^A_{\min}\rho}+\frac{3\kappa_A\sigma^2_{\max}(G)}{\sigma^A_{\min}\eta^2\rho})\|x_k-x_{k-1}\|^2 + \frac{2\kappa_AL^2}{\sigma^A_{\min}\rho b} \sum_{i=(n_k-1)q}^{k-1}\mathbb{E}\|x_{i+1}-x_i\|^2.
\end{align}
Let $b=q$, $\eta=\frac{2\alpha\sigma_{\min}(G)}{3L} \ (0<\alpha \leq 1)$ and $\rho = \frac{\sqrt{170\kappa_A}\kappa_GL}{\sigma^A_{\min}\alpha}$,
then we have
\begin{align}
 \frac{1}{K}\sum_{i=0}^{K-1} (\|x_{i+1} - x_{i}\|^2 + \sum_{j=1}^m \|y_j^i-y_j^{i+1}\|^2) \leq \frac{\mathbb{E} [R_{0}]-R^*}{K\gamma},
\end{align}
where $\gamma = \min(\chi,\sigma_{\min}^H)$ with $\chi\geq \frac{\sqrt{170\kappa_A}\kappa_GL}{4\alpha}$ and $R^*$ is a lower bound of the function $R_k$.
\end{lemma}

\begin{proof}

By the optimal condition of step 9 in Algorithm \ref{alg:3},
we have, for $j\in [m]$
\begin{align}
0 & =(y_j^k-y_j^{k+1})^T\big(\partial g_j(y_j^{k+1}) - B^Tz_k + \rho B^T(Ax_k + \sum_{i=1}^jB_iy_i^{k+1} + \sum_{i=j+1}^mB_iy_i^{k}-c) + H_j(y_j^{k+1}-y_j^k)\big) \nonumber \\
& \leq g_j(y_j^k)- g_j(y_j^{k+1}) - (z_k)^T(B_jy_j^k-B_jy_j^{k+1}) + \rho(By_j^k-By_j^{k+1})^T(Ax_k + \sum_{i=1}^jB_iy_i^{k+1} + \sum_{i=j+1}^mB_iy_i^{k}-c) - \|y_j^{k+1}-y_j^k\|^2_{H_j} \nonumber \\
& = g_j(y_j^k)- g_j(y_j^{k+1}) - (z_k)^T(Ax_k+\sum_{i=1}^{j-1}B_iy_i^{k+1} + \sum_{i=j}^mB_iy_i^{k}-c) + (z_k)^T(Ax_k+\sum_{i=1}^jB_iy_i^{k+1}+ \sum_{i=j+1}^mB_iy_i^{k}-c) \nonumber \\
& \quad  + \frac{\rho}{2}\|Ax_k +\sum_{i=1}^{j-1}B_iy_i^{k+1} + \sum_{i=j}^mB_iy_i^{k}-c\|^2 - \frac{\rho}{2}\|Ax_k+\sum_{i=1}^jB_iy_i^{k+1}+ \sum_{i=j+1}^mB_iy_i^{k}-c\|^2 -\frac{\rho}{2}\|B_jy_j^k-B_jy_j^{k+1}\|^2 \nonumber \\
& \quad - \|y_j^{k+1}-y_j^k\|^2_{H_j} \nonumber \\
& \leq \underbrace{ f(x_k) + g_j(y_j^k) - (z_k)^T(Ax_k+\sum_{i=1}^{j-1}B_iy_i^{k+1} + \sum_{i=j}^mB_iy_i^{k}-c) + \frac{\rho}{2}\|Ax_k +\sum_{i=1}^{j-1}B_iy_i^{k+1} + \sum_{i=j}^mB_iy_i^{k}-c\|^2}_{\mathcal{L}_{\rho} (x_k,y_{[j-1]}^{k+1},y_{[j:m]}^k,z_k)}  - \|y_j^{k+1}-y_j^k\|^2_{H_j}  \nonumber \\
& \quad -\underbrace{(f(x_k) + g_j(y_j^{k+1}) - (z_k)^T(Ax_k+\sum_{i=1}^jB_iy_i^{k+1}+ \sum_{i=j+1}^mB_iy_i^{k}-c) + \frac{\rho}{2}\|Ax_k+\sum_{i=1}^jB_iy_i^{k+1}+ \sum_{i=j+1}^mB_iy_i^{k}-c\|^2}_{\mathcal{L}_{\rho} (x_k,y_{[j]}^{k+1},y_{[j+1:m]}^k,z_k)} \nonumber \\
& \leq \mathcal{L}_{\rho} (x_k,y_{[j-1]}^{k+1},y_{[j:m]}^k,z_k) - \mathcal{L}_{\rho} (x_k,y_{[j]}^{k+1},y_{[j+1:m]}^k,z_k)
- \sigma_{\min}(H_j)\|y_j^k-y_j^{k+1}\|^2,
\end{align}
where the first inequality holds by the convexity of function $g_j(y)$,
and the second equality follows by applying the equality
$(a-b)^Tb = \frac{1}{2}(\|a\|^2-\|b\|^2-\|a-b\|^2)$ on the term $(By_j^k-By_j^{k+1})^T(Ax_k + \sum_{i=1}^jB_iy_i^{k+1} + \sum_{i=j+1}^mB_iy_i^{k}-c)$.
Thus, we have, for all $j\in[m]$
\begin{align} \label{eq:A4-1}
 \mathcal{L}_{\rho} (x_k,y_{[j-1]}^{k+1},y_{[j:m]}^k,z_k) \leq \mathcal{L}_{\rho} (x_k,y_{[j]}^{k+1},y_{[j+1:m]}^k,z_k)
 - \sigma_{\min}(H_j)\|y_j^k-y_j^{k+1}\|^2.
\end{align}
Telescoping inequality \eqref{eq:A4-1} over $j$ from $1$ to $m$, we obtain
\begin{align} \label{eq:A23}
 \mathcal{L}_{\rho} (x_k,y^{k+1}_{[m]},z_k) \leq \mathcal{L}_{\rho} (x_k,y^k_{[m]},z_k)
 - \sigma_{\min}^H\sum_{j=1}^m \|y_j^k-y_j^{k+1}\|^2.
\end{align}
where $\sigma_{\min}^H=\min_{j\in[m]}\sigma_{\min}(H_j)$.

By Assumption 1, we have
\begin{align} \label{eq:A24}
0 \leq f(x_k) - f(x_{k+1}) + \nabla f(x_k)^T(x_{k+1}-x_k) + \frac{L}{2}\|x_{k+1}-x_k\|^2.
\end{align}
Using the optimal condition of step 10 in Algorithm \ref{alg:3},
we have
\begin{align} \label{eq:A25}
 0 = (x_k-x_{k+1})^T \big( v_k - A^Tz_k + \rho A^T(Ax_{k+1} + \sum_{j=1}^mB_jy_j^{k+1}-c) + \frac{G}{\eta}(x_{k+1}-x_k) \big).
\end{align}
Combining \eqref{eq:A24} and \eqref{eq:A25}, we have
\begin{align}
 0 & \leq f(x_k) - f(x_{k+1}) + \nabla f(x_k)^T(x_{k+1}-x_k) + \frac{L}{2}\|x_{k+1}-x_k\|^2 \nonumber \\
 & \quad + (x_k-x_{k+1})^T \big( v_k - A^Tz_k + \rho A^T(Ax_{k+1} + \sum_{j=1}^mB_jy_j^{k+1}-c) + \frac{G}{\eta}(x_{k+1}-x_k) \big)  \nonumber \\
 & = f(x_k) - f(x_{k+1}) + \frac{L}{2}\|x_k-x_{k+1}\|^2 - \frac{1}{\eta}\|x_k - x_{k+1}\|^2_G + (x_k-x_{k+1})^T(v_k-\nabla f(x_k)) \nonumber \\
 & \quad -(z_k)^T(Ax_k-Ax_{k+1}) + \rho(Ax_k - Ax_{k+1})^T(Ax_{k+1} + \sum_{j=1}^mB_jy_j^{k+1}-c) \nonumber \\
 & = f(x_k) - f(x_{k+1}) + \frac{L}{2}\|x_k-x_{k+1}\|^2 - \frac{1}{\eta}\|x_k - x_{k+1}\|^2_G
 + (x_k-x_{k+1})^T\big(v_k-\nabla f(x_k)\big) -(z_k)^T(Ax_k + \sum_{j=1}^mB_jy_j^{k+1}-c)\nonumber \\
 & \quad  + (z_k)^T(Ax_{k+1}+ \sum_{j=1}^mB_jy_j^{k+1}-c) + \frac{\rho}{2}\big(\|Ax_k + \sum_{j=1}^mB_jy_j^{k+1}-c\|^2
 - \|Ax_{k+1} + \sum_{j=1}^mB_jy_j^{k+1}-c\|^2 - \|Ax_k - Ax_{k+1}\|^2 \big) \nonumber \\
 & = \underbrace{f(x_k) -z_k^T(Ax_k + \sum_{j=1}^mB_jy_j^{k+1}-c) + \frac{\rho}{2}\|Ax_{k} + \sum_{j=1}^mB_jy_j^{k+1}-c\|^2}_{L_{\rho}(x_k,y_{[m]}^{k+1},z_k)} + \frac{L}{2}\|x_k-x_{k+1}\|^2
 + (x_k-x_{k+1})^T\big(v_k-\nabla f(x_k)\big)  \nonumber \\
 & \quad - \underbrace{\big(f(x_{k+1})-  z_k^T(Ax_{k+1}+ \sum_{j=1}^mB_jy_j^{k+1}-c) + \frac{\rho}{2} \|Ax_{k+1} + \sum_{j=1}^mB_jy_j^{k+1}-c\|^2\big)}_{L_{\rho}(x_{k+1},y^{k+1}_{[m]},z_{k})} - \frac{1}{\eta}\|x_k - x_{k+1}\|^2_G - \frac{\rho}{2}\|Ax_k - Ax_{k+1}\|^2 \nonumber \\
 & \leq \mathcal{L}_{\rho} (x_k,y_{[m]}^{k+1},z_k) -  \mathcal{L}_{\rho} (x_{k+1},y_{[m]}^{k+1},z_k)
 - (\frac{\sigma_{\min}(G)}{\eta}+\frac{\rho \sigma^A_{\min}}{2}-\frac{L}{2}) \|x_{k+1} - x_{k}\|^2 + (x_k-x_{k+1})^T(v_k-\nabla f(x_k)) \nonumber \\
 & \leq \mathcal{L}_{\rho} (x_k,y_{[m]}^{k+1},z_k) -  \mathcal{L}_{\rho} (x_{k+1},y_{[m]}^{k+1},z_k)
 - (\frac{\sigma_{\min}(G)}{\eta}+\frac{\rho \sigma^A_{\min}}{2}-L ) \|x_{k+1} - x_{k}\|^2 + \frac{1}{2L}\|v_k - \nabla f(x_k)\|^2 \nonumber \\
 & \leq \mathcal{L}_{\rho} (x_k,y_{[m]}^{k+1},z_k) -  \mathcal{L}_{\rho} (x_{k+1},y_{[m]}^{k+1},z_k)
 - (\frac{\sigma_{\min}(G)}{\eta}+\frac{\rho \sigma^A_{\min}}{2}-L ) \|x_{k+1} - x_{k}\|^2 + \frac{L}{2b}\sum_{i=(n_k-1)q}^{k-1}\mathbb{E}\|x_{i+1}-x_i\|^2, \nonumber
\end{align}
where the second equality follows by applying the equality
$(a-b)^Tb = \frac{1}{2}(\|a\|^2-\|b\|^2-\|a-b\|^2)$ over the
term $(Ax_k - Ax_{k+1})^T(Ax_{k+1}+\sum_{j=1}^mB_jy_j^{k+1}-c)$; the third inequality
follows by the inequality $a^Tb \leq \frac{1}{2L}\|a\|^2 + \frac{L}{2}\|b\|^2$, and the forth inequality holds by the inequality \eqref{eq:A12}.
It follows that
\begin{align} \label{eq:A26}
\mathcal{L}_{\rho} (x_{k+1},y_{[m]}^{k+1},z_k) \leq \mathcal{L}_{\rho} (x_k,y_{[m]}^{k+1},z_k)
 - (\frac{\sigma_{\min}(G)}{\eta}+\frac{\rho \sigma^A_{\min}}{2}-L ) \|x_{k+1} - x_{k}\|^2 + \frac{L}{2b}\sum_{i=(n_k-1)q}^{k-1}\mathbb{E}\|x_{i+1}-x_i\|^2.
\end{align}

Using the step 10 in Algorithm \ref{alg:3}, we have
\begin{align} \label{eq:A27}
\mathcal{L}_{\rho} (x_{k+1},y_{[m]}^{k+1},z_{k+1}) -\mathcal{L}_{\rho} (x_{k+1},y_{[m]}^{k+1},z_k)
& = \frac{1}{\rho}\|z_{k+1}-z_k\|^2 \nonumber \\
& \leq \frac{18L^2}{\sigma^A_{\min}b\rho} \sum_{i=(n_k-1)q}^{k-1}\mathbb{E}\|x_{i+1}-x_i\|^2
 + (\frac{9L^2}{\sigma^A_{\min}\rho}+\frac{3\sigma^2_{\max}(G)}{\sigma^A_{\min}\eta^2\rho})\|x_k-x_{k-1}\|^2 \nonumber \\
& + \frac{3\sigma^2_{\max}(G)}{\sigma^A_{\min}\eta^2\rho}\|x_{k+1}-x_{k}\|^2,
\end{align}
where the above inequality holds by Lemma \ref{lem:A3}.

Combining \eqref{eq:A23}, \eqref{eq:A26} and \eqref{eq:A27}, we have
\begin{align} \label{eq:A28}
\mathcal{L}_{\rho} (x_{k+1},y_{[m]}^{k+1},z_{k+1}) & \leq \mathcal{L}_{\rho} (x_k,y_{[m]}^k,z_k)
  - \sigma_{\min}^H\sum_{j=1}^m \|y_j^k-y_j^{k+1}\|^2- (\frac{\sigma_{\min}(G)}{\eta}+\frac{\rho \sigma^A_{\min}}{2}-L - \frac{3\sigma^2_{\max}(G)}{\sigma^A_{\min}\eta^2\rho}) \|x_{k+1} - x_{k}\|^2  \nonumber \\
& \quad + (\frac{L}{2b} + \frac{18L^2}{\sigma^A_{\min}b\rho})\sum_{i=(n_k-1)q}^{k-1}\mathbb{E}\|x_{i+1}-x_i\|^2
 + (\frac{9L^2}{\sigma^A_{\min}\rho}+\frac{3\sigma^2_{\max}(G)}{\sigma^A_{\min}\eta^2\rho})\|x_k-x_{k-1}\|^2.
\end{align}

Next, we define a \emph{Lyapunov} function $R_k$:
\begin{align}
 R_k = \mathcal{L}_{\rho} (x_k,y_{[m]}^k,z_k) + (\frac{9L^2}{\sigma^A_{\min}\rho}+\frac{3\kappa_A\sigma^2_{\max}(G)}{\sigma^A_{\min}\eta^2\rho})\|x_k-x_{k-1}\|^2 + \frac{2\kappa_AL^2}{\sigma^A_{\min}\rho b} \sum_{i=(n_k-1)q}^{k-1}\mathbb{E}\|x_{i+1}-x_i\|^2.
\end{align}
Since
$$\sum_{i=(n_k-1)q}^{k}\mathbb{E}\|x_{i+1}-x_i\|^2 = \sum_{i=(n_k-1)q}^{k-1}\mathbb{E}\|x_{i+1}-x_i\|^2 + \|x_{k+1}-x_k\|^2,$$
and $\kappa_A \geq 1$, the inequality \eqref{eq:A28} can be rewrite as follows:
\begin{align} \label{eq:A30}
R_{k+1} & \leq R_k
  - (\frac{\sigma_{\min}(G)}{\eta}+\frac{\rho \sigma^A_{\min}}{2}-L - \frac{6\kappa_A\sigma^2_{\max}(G)}{\sigma^A_{\min}\eta^2\rho}-\frac{9L^2}{\sigma^A_{\min}\rho}
  - \frac{2\kappa_AL^2}{\sigma^A_{\min}\rho b}) \|x_{k+1} - x_{k}\|^2  \nonumber \\
& \quad- \sigma_{\min}^H\sum_{j=1}^m \|y_j^k-y_j^{k+1}\|^2 + (\frac{L}{2b} + \frac{18L^2}{\sigma^A_{\min}b\rho})\sum_{i=(n_k-1)q}^{k-1}\mathbb{E}\|x_{i+1}-x_i\|^2.
\end{align}

Then telescoping equality \eqref{eq:A30} over $k$ from $(n_k-1)q$ to $k$ where $k \leq n_kq-1$ and
let $n_j=n_k$ for $(n_k-1)q \leq j \leq n_kq-1$, we have
\begin{align} \label{eq:A31}
\mathbb{E} [R_{k+1}] &\leq \mathbb{E} [R_{(n_k-1)q}]
  - (\frac{\sigma_{\min}(G)}{\eta}+\frac{\rho \sigma^A_{\min}}{2}-L - \frac{6\kappa_A\sigma^2_{\max}(G)}{\sigma^A_{\min}\eta^2\rho}-\frac{9L^2}{\sigma^A_{\min}\rho}
   -\frac{2\kappa_AL^2}{\sigma^A_{\min}\rho b} ) \sum_{j=(n_k-1)q}^k \|x_{j+1} - x_{j}\|^2  \nonumber \\
& - \sigma_{\min}^H\sum_{i=(n_k-1)q}^{k}\sum_{j=1}^m \|y_j^i-y_j^{i+1}\|^2 + (\frac{L}{2b} + \frac{18L^2}{\sigma^A_{\min}b\rho})\sum_{j=(n_k-1)q}^{k}\sum_{i=(n_k-1)q}^{k-1}\mathbb{E}\|x_{i+1}-x_i\|^2  \nonumber \\
& \leq \mathbb{E} [R_{(n_k-1)q}]
  - (\frac{\sigma_{\min}(G)}{\eta}+\frac{\rho \sigma^A_{\min}}{2}-L - \frac{6\kappa_A\sigma^2_{\max}(G)}{\sigma^A_{\min}\eta^2\rho}-\frac{9L^2}{\sigma^A_{\min}\rho}
  -\frac{2\kappa_AL^2}{\sigma^A_{\min}\rho b}) \sum_{i=(n_k-1)q}^k \|x_{i+1} - x_{i}\|^2  \nonumber \\
& - \sigma_{\min}^H\sum_{i=(n_k-1)q}^{k-1}\sum_{j=1}^m \|y_j^i-y_j^{i+1}\|^2+ (\frac{Lq}{2b} + \frac{18L^2q}{\sigma^A_{\min}b\rho})\sum_{i=(n_k-1)q}^{k}\mathbb{E}\|x_{i+1}-x_i\|^2  \nonumber \\
& = \mathbb{E} [R_{(n_k-1)q}] - (\frac{\sigma_{\min}(G)}{\eta}+\frac{\rho \sigma^A_{\min}}{2}-L - \frac{6\kappa_A\sigma^2_{\max}(G)}{\sigma^A_{\min}\eta^2\rho}-\frac{9L^2}{\sigma^A_{\min}\rho} - \frac{2\kappa_AL^2}{\sigma^A_{\min}\rho b}-\frac{Lq}{2b} - \frac{18L^2q}{\sigma^A_{\min}b\rho} ) \sum_{i=(n_k-1)q}^k \|x_{i+1} - x_{i}\|^2 \nonumber \\
& \quad - \sigma_{\min}^H\sum_{i=(n_k-1)q}^{k-1}\sum_{j=1}^m \|y_j^i-y_j^{i+1}\|^2,
\end{align}
where the second inequality holds by the fact that
\begin{align}
\sum_{j=(n_k-1)q}^{k}\sum_{i=(n_k-1)q}^{k-1}\mathbb{E}\|x_{i+1}-x_i\|^2 \leq \sum_{j=(n_k-1)q}^{k}\sum_{i=(n_k-1)q}^{k}\mathbb{E}\|x_{i+1}-x_i\|^2 \leq q\sum_{i=(n_k-1)q}^{k}\mathbb{E}\|x_{i+1}-x_i\|^2. \nonumber
\end{align}

Since $b=q$, we have
\begin{align}
 \chi & = \frac{\sigma_{\min}(G)}{\eta}+\frac{\rho \sigma^A_{\min}}{2}-L - \frac{6\sigma^2_{\max}(G)}{\sigma^A_{\min}\eta^2\rho}
 -\frac{9L^2}{\sigma^A_{\min}\rho} -  \frac{2\kappa_AL^2}{\sigma^A_{\min}\rho b} -\frac{Lq}{2b} - \frac{18L^2q}{\sigma^A_{\min}b\rho} \nonumber \\
 & = \underbrace{\frac{\sigma_{\min}(G)}{\eta} - \frac{3L}{2}}_{L_1} + \underbrace{\frac{\rho \sigma^A_{\min}}{2} - \frac{6\sigma^2_{\max}(G)}{\sigma^A_{\min}\eta^2\rho}
 -\frac{27L^2}{\sigma^A_{\min}\rho}-  \frac{2\kappa_AL^2}{\sigma^A_{\min}\rho b}}_{L_2}.
\end{align}
Given $0< \eta \leq \frac{2\sigma_{\min}(G)}{3L}$, we have $L_1\geq 0$. Further, let $\eta = \frac{2\alpha\sigma_{\min}(G)}{3L} \ (0<\alpha \leq 1)$ and
$\rho = \frac{\sqrt{170\kappa_A}\kappa_GL}{\sigma^A_{\min}\alpha}$,
we have
\begin{align}
L_2 & = \frac{\rho \sigma^A_{\min}}{2} - \frac{6\sigma^2_{\max}(G)}{\sigma^A_{\min}\eta^2\rho}
 - \frac{27L^2}{\sigma^A_{\min}\rho} - \frac{2\kappa_AL^2}{\sigma^A_{\min}\rho b}\nonumber \\
 & =  \frac{\rho \sigma^A_{\min}}{2} - \frac{27L^2\kappa^2_G}{2\sigma^A_{\min}\rho\alpha^2}
 -\frac{27L^2}{\sigma^A_{\min}\rho} - \frac{2\kappa_AL^2}{\sigma^A_{\min}\rho b}\nonumber \\
 & \geq \frac{\rho \sigma^A_{\min}}{2} - \frac{27L^2\kappa^2_G\kappa_A}{2\sigma^A_{\min}\rho\alpha^2}
 -\frac{27L^2\kappa^2_G\kappa_A}{\sigma^A_{\min}\rho\alpha^2} -  \frac{2\kappa^2_G\kappa_AL^2}{\sigma^A_{\min}\rho \alpha^2} \nonumber \\
 & \geq \frac{\rho \sigma^A_{\min}}{4} + \underbrace{\frac{\rho \sigma^A_{\min}}{4} - \frac{85L^2\kappa^2_G\kappa_A }{2\sigma^A_{\min}\rho\alpha^2}}_{\geq 0} \nonumber \\
 & \geq \frac{\sqrt{170\kappa_A}\kappa_GL}{4\alpha},
\end{align}
where the first inequality holds by $\kappa_G=\geq 1$ and
$b \geq 1 \geq \alpha^2$ and the third equality holds by $\rho = \frac{\sqrt{170\kappa_A}\kappa_GL}{\sigma^A_{\min}\alpha}$.
Thus, we obtain $\chi \geq \frac{\sqrt{170\kappa_A}\kappa_GL}{4\alpha}$.

Next, using \eqref{eq:A16}, we have
\begin{align}
 z_{k+1} = (A^T)^+(v_{k} + \frac{G}{\eta}(x_{k+1}-x_k)),
\end{align}
where $(A^T)^+$ is the pseudoinverse of $A^T$. Due to that $A$ is full row rank,
we have $(A^T)^+ = (AA^T)^{-1}A$.
It follows that $\sigma_{\max}((A^T)^+)^T(A^T)^+)\leq \frac{\sigma_{\max}^A}{(\sigma^A_{\min})^2}=\frac{\kappa_A}{\sigma^A_{\min}}$.

Then we have
\begin{align}
&\mathcal{L}_{\rho} (x_{k+1},y_{[m]}^{k+1},z_{k+1})
= f(x_{k+1}) + \sum_{j=1}^m g_j(y_j^{k+1}) - z_{k+1}^T(Ax_{k+1} + \sum_{j=1}^mB_jy_j^{k+1} -c) + \frac{\rho}{2}\|Ax_{k+1} + \sum_{j=1}^mB_jy_j^{k+1} -c\|^2 \nonumber \\
& = f(x_{k+1}) + \sum_{j=1}^m g_j(y_j^{k+1}) -  \langle(A^T)^+(v_k + \frac{G}{\eta}(x_{k+1}-x_k)), Ax_{k+1} + \sum_{j=1}^mB_jy_j^{k+1} -c\rangle + \frac{\rho}{2}\|Ax_{k+1} + \sum_{j=1}^mB_jy_j^{k+1} -c\|^2 \nonumber \\
& = f(x_{k+1}) + \sum_{j=1}^m g_j(y_j^{k+1}) - \langle(A^T)^+(v_{k} - \nabla f(x_{k}) + \nabla f(x_{k})+ \frac{G}{\eta}(x_{k+1}-x_k)), Ax_{k+1} + \sum_{j=1}^mB_jy_j^{k+1} -c\rangle  \nonumber \\
& \quad +  \frac{\rho}{2}\|Ax_{k+1} + \sum_{j=1}^mB_jy_j^{k+1} -c\|^2 \nonumber \\
& \geq f(x_{k+1}) + \sum_{j=1}^m g_j(y_j^{k+1}) - \frac{2\kappa_A}{\sigma^A_{\min}\rho}\|v_k - \nabla f(x_{k})\|^2 - \frac{2\kappa_A}{\sigma^A_{\min}\rho}\|\nabla f(x_{k})\|^2
- \frac{2\kappa_A\sigma^2_{\max}(G)}{\sigma^A_{\min}\eta^2\rho}\|x_{k+1}-x_k\|^2 \nonumber \\
& \quad + \frac{\rho}{8}\|Ax_{k+1} + \sum_{j=1}^mB_jy_j^{k+1} -c\|^2 \nonumber\\
& \geq f(x_{k+1}) + \sum_{j=1}^m g_j(y_j^{k+1})  - \frac{2\kappa_AL^2}{\sigma^A_{\min}b\rho} \sum_{i=(n_k-1)q}^{k-1}\mathbb{E}\|x_{i+1}-x_i\|^2 - \frac{2\kappa_A\delta^2}{\sigma^A_{\min}\rho}
- \frac{2\kappa_A\sigma^2_{\max}(G)}{\sigma^A_{\min}\eta^2\rho}\|x_{k+1}-x_k\|^2
\end{align}
where the first inequality is obtained by applying $ \langle a, b\rangle \leq \frac{1}{2\beta}\|a\|^2 + \frac{\beta}{2}\|b\|^2$ to the terms
$\langle(A^T)^+(v_{k} - \nabla f(x_{k})), Ax_{k+1} + By_{[m]}^{k+1} -c\rangle$, $\langle(A^T)^+v_{k}, Ax_{k+1} + By_{[m]}^{k+1} -c\rangle $ and
$\langle(A^T)^+\frac{G}{\eta}(x_{k+1}-x_k), Ax_{k+1} + By_{[m]}^{k+1} -c\rangle$ with $\beta = \frac{\rho}{4}$, respectively. The second inequality follows by the inequality \eqref{eq:A12}
and Assumption 3.
Therefore, we have, for $k=0,1,2,\cdots$
\begin{align}
 R_{k+1}\geq f^* + \sum_{j=1}^m g_j^* - \frac{2\kappa_A\delta^2}{\sigma^A_{\min}\rho}.
\end{align}
It follows that the function $R_k$ is bounded from below. Let $R^*$ denotes a low bound of function $R_k$.

Further, telescoping equality \eqref{eq:A31} over $k$ from $0$ to $K$,
we have
\begin{align}
\mathbb{E} [R_{K}] - \mathbb{E} [R_{0}] & = (\mathbb{E} [R_{q}] - \mathbb{E} [R_{0}]) + (\mathbb{E} [R_{2q}] - \mathbb{E} [R_{q}]) + \cdots + (\mathbb{E} [R_{K}] - \mathbb{E} [R_{(n_k-1)q}]) \nonumber \\
& \leq - \sum_{i=0}^{q-1} (\chi\|x_{i+1} - x_{i}\|^2 + \sigma_{\min}^H\sum_{j=1}^m \|y_j^k-y_j^{k+1}\|^2) - \sum_{i=q}^{2q-1} ( \chi\|x_{i+1} - x_{i}\|^2 + \sigma_{\min}^H\sum_{j=1}^m \|y_j^k-y_j^{k+1}\|^2) \nonumber \\
& \quad - \cdots - \sum_{i=(n_k-1)q}^{K-1} ( \chi \|x_{i+1} - x_{i}\|^2  + \sigma_{\min}^H\sum_{j=1}^m \|y_j^k-y_j^{k+1}\|^2) \nonumber \\
& = - \sum_{i=0}^{K-1} (\chi \|x_{i+1} - x_{i}\|^2+ \sigma_{\min}^H\sum_{j=1}^m \|y_j^k-y_j^{k+1}\|^2) .
\end{align}

Finally, we obtain
\begin{align} \label{eq:A38}
 \frac{1}{K}\sum_{i=0}^{K-1} (\|x_{i+1} - x_{i}\|^2 + \sum_{j=1}^m \|y_j^i-y_j^{i+1}\|^2) \leq \frac{\mathbb{E} [R_{0}]-R^*}{K\gamma},
\end{align}
where $\gamma = \min(\chi,\sigma_{\min}^H)$ with $\chi\geq \frac{\sqrt{170\kappa_A}\kappa_GL}{4\alpha}$.

\end{proof}

\begin{theorem}
 Suppose the sequence $\{x_k,y_{[m]}^k,z_k)_{k=1}^K$ is generated from Algorithm \ref{alg:3}, and let $b=q$, $\eta = \frac{2\alpha\sigma_{\min}(G)}{3L} \ (0<\alpha \leq 1)$,
 $\rho = \frac{\sqrt{170 \kappa_A} \kappa_G L}{\sigma^A_{\min}\alpha}$, and
  \begin{align}
   \nu_1 = m\big(\rho^2\sigma^B_{\max}\sigma^A_{\max} + \rho^2(\sigma^B_{\max})^2 + \sigma^2_{\max}(H)\big), \ \nu_2 = 3(L^2+ \frac{\sigma^2_{\max}(G)}{\eta^2}),
   \ \nu_3 = \frac{18 L^2 }{\sigma^A_{\min} \rho^2} + \frac{3\sigma^2_{\max}(G) }{\sigma^A_{\min}\eta^2\rho^2},
 \end{align}
 then we have
 \begin{align}
  \frac{1}{K}\sum_{k=1}^K \mathbb{E}\big[ \mbox{dist}(0,\partial L(x_k,y_{[m]}^k,z_k))^2\big] \leq \frac{\nu_{\max}}{K}\sum_{k=1}^{K-1} \theta_k \leq  \frac{3\nu_{\max}(R_{0}-R^*)}{K\gamma},
 \end{align}
 where $\gamma \geq \frac{\sqrt{170\kappa_A}\kappa_GL}{4\alpha}$, $\nu_{\max}=\max\{\nu_1,\nu_2,\nu_3\}$ and
 $R^*$ is a lower bound of the function $R_k$.
 It implies that the number of iteration $K$ satisfies
 \begin{align}
 K = \frac{3\nu_{\max}(R_{0}-R^*)}{\epsilon \gamma} \nonumber
 \end{align}
 then $(x_{k^*},y_{k^*},z_{k^*})$ is an $\epsilon$-approximate stationary point of \eqref{eq:2}, where $k^* = \mathop{\arg\min}_{k}\theta_{k}$.
\end{theorem}

\begin{proof}
First, we define a useful variable $\theta_k=\|x_{k+1}-x_{k}\|^2+\|x_{k}-x_{k-1}\|^2+\frac{1}{q}\sum_{i=(n_k-1)q}^k \|x_{i+1}-x_i\|^2 + \sum_{j=1}^m \|y_j^k-y_j^{k+1}\|^2$.
Next, by the optimal condition of the step 9 in Algorithm \ref{alg:3}, we
have, for all $i\in [m]$
\begin{align}
  \mathbb{E}\big[\mbox{dist}(0,\partial_{y_j} L(x,y_{[m]},z))^2\big]_{k+1} & = \mathbb{E}\big[\mbox{dist} (0, \partial g_j(y_j^{k+1})-B_j^Tz_{k+1})^2\big] \nonumber \\
 & = \|B_j^Tz_k -\rho B_j^T(Ax_k + \sum_{i=1}^jB_iy_i^{k+1} + \sum_{i=j+1}^mB_iy_i^{k} -c) - H_j(y_j^{k+1}-y_j^k) -B_j^Tz_{k+1}\|^2 \nonumber \\
 & = \|\rho B_j^TA(x_{k+1}-x_{k}) + \rho B_j^T \sum_{i=j+1}^m B_i (y_i^{k+1}-y_i^{k})- H_j(y_j^{k+1}-y_j^k) \|^2 \nonumber \\
 & \leq m\rho^2\sigma^{B_j}_{\max}\sigma^A_{\max}\|x_{k+1}-x_k\|^2 + m\rho^2\sigma^{B_j}_{\max}\sum_{i=j+1}^m \sigma^{B_i}_{\max}\|y_i^{k+1}-y_i^{k}\|^2 + m\sigma^2_{\max}(H_j)\|y_j^{k+1}-y_j^k\|^2\nonumber \\
 & \leq m\big(\rho^2\sigma^B_{\max}\sigma^A_{\max} + \rho^2(\sigma^B_{\max})^2 + \sigma^2_{\max}(H)\big) \theta_{k},
\end{align}
where the first inequality follows by the inequality $\|\frac{1}{n}\sum_{i=1}^n z_i\|^2 \leq \frac{1}{n}\sum_{i=1}^n \|z_i\|^2$.

By the step 10 of Algorithm \ref{alg:3}, we have
\begin{align}
 \mathbb{E}\big[\mbox{dist}(0,\nabla_x L(x,y_{[m]},z))^2\big]_{k+1} & = \mathbb{E}\|A^Tz_{k+1}-\nabla f(x_{k+1})\|^2  \nonumber \\
 & = \mathbb{E}\|v_k - \nabla f(x_{k+1}) - \frac{G}{\eta} (x_k - x_{k+1})\|^2 \nonumber \\
 & = \mathbb{E}\|v_k - \nabla f(x_{k}) + \nabla f(x_{k})- \nabla f(x_{k+1}) - \frac{G}{\eta}(x_k-x_{k+1})\|^2  \nonumber \\
 & \leq \sum_{i=(n_k-1)q}^{k-1}\frac{3L^2}{b}\mathbb{E}\|x_{i+1}-x_i\|^2 + 3(L^2+ \frac{\sigma^2_{\max}(G)}{\eta^2})\|x_k-x_{k+1}\|^2  \nonumber \\
 & \leq 2(L^2+ \frac{\sigma^2_{\max}(G)}{\eta^2})\theta_{k},
\end{align}
where the second inequality holds by $b=q$.

By the step 11 of Algorithm \ref{alg:3}, we have
\begin{align}
 \mathbb{E}\big[\mbox{dist}(0,\nabla_z L(x,y_{[m]},z))^2\big]_{k+1} & = \mathbb{E}\|Ax_{k+1}+ \sum_{j=1}^m B_jy_j^{k+1}-c\|^2 \nonumber \\
 &= \frac{1}{\rho^2} \mathbb{E} \|z_{k+1}-z_k\|^2  \nonumber \\
 & \leq \frac{18L^2}{\sigma^A_{\min}b\rho^2} \sum_{i=(n_k-1)q}^{k-1}\mathbb{E}\|x_{i+1}-x_i\|^2
 + (\frac{9L^2}{\sigma^A_{\min}\rho^2}+\frac{3\sigma^2_{\max}(G)}{\sigma^A_{\min}\eta^2\rho^2})\|x_k-x_{k-1}\|^2 \nonumber \\
 & + \frac{3\sigma^2_{\max}(G)}{\sigma^A_{\min}\eta^2\rho^2}\|x_{k+1}-x_{k}\|^2 \nonumber \\
 & \leq \big( \frac{18 L^2 }{\sigma^A_{\min} \rho^2} + \frac{3\sigma^2_{\max}(G) }{\sigma^A_{\min}\eta^2\rho^2} \big) \theta_{k} ,
\end{align}
where the second inequality holds by $b=q$.

By \eqref{eq:A38}, we have
\begin{align}
 \frac{1}{K}\sum_{i=0}^{K-1} (\|x_{i+1} - x_{i}\|^2 + \sum_{j=1}^m \|y_j^i-y_j^{i+1}\|^2) \leq \frac{\mathbb{E} [R_{0}]-R^*}{K\gamma},
\end{align}
where $\gamma = \min(\chi,\sigma_{\min}^H)$ with $\chi\geq \frac{\sqrt{170\kappa_A}\kappa_GL}{4\alpha}$. Since
\begin{align}
\sum_{k=0}^{K-1}\sum_{i=(n_k-1)q}^k \|x_{i+1}-x_i\|^2 \leq q\sum_{k=0}^{K-1} \|x_{k+1}-x_k\|^2
\end{align}
we have
\begin{align}
 \frac{1}{K}\sum_{k=1}^K \mathbb{E}\big[ \mbox{dist}(0,\partial L(x_k,y_{[m]}^k,z_k))^2\big] \leq \frac{\nu_{\max}}{K}\sum_{k=1}^{K-1} \theta_k \leq  \frac{3\nu_{\max}(R_{0}-R^*)}{K \gamma},
\end{align}
where $\gamma \geq \frac{\sqrt{170\kappa_A}\kappa_GL}{4\alpha}$, $\nu_{\max}=\max\{\nu_1,\nu_2,\nu_3\}$ with
\begin{align}
   \nu_1 = m\big(\rho^2\sigma^B_{\max}\sigma^A_{\max} + \rho^2(\sigma^B_{\max})^2 + \sigma^2_{\max}(H)\big), \ \nu_2 = 3(L^2+ \frac{\sigma^2_{\max}(G)}{\eta^2}),
   \ \nu_3 = \frac{18 L^2 }{\sigma^A_{\min} \rho^2} + \frac{3\sigma^2_{\max}(G) }{\sigma^A_{\min}\eta^2\rho^2}.
 \end{align}
Given $\eta = \frac{2\alpha\sigma_{\min}(G)}{3L} \ (0<\alpha \leq 1)$ and $\rho = \frac{\sqrt{170 \kappa_A} \kappa_G L}{\sigma^A_{\min}\alpha}$, since $m$ is relatively small,
it easy verifies that
$\nu_{\max} = O(1)$ and $\gamma=O(1)$, which are independent on $n$ and $K$. Thus, we obtain
\begin{align}
 \frac{1}{K}\sum_{k=1}^K \mathbb{E}\big[ \mbox{dist}(0,\partial L(x_k,y_{[m]}^k,z_k))^2\big] \leq  O(\frac{1}{K}).
\end{align}

\end{proof}

\subsection{ Convergence Analysis of the Online SPIDER-ADMM }
\label{Appendix:A2}
In this subsection, we conduct convergence analysis of the online SPIDER-ADMM.
First, we give some useful lemmas.

\begin{lemma} \label{lem:A5}
Under Assumption 1 and given the sequence $\{x_k,y_{[m]}^k,,z_k\}_{k=1}^K$ from Algorithm \ref{alg:4}, it holds that
\begin{align}
 \mathbb{E}\|z_{k+1} - z_k\|^2 \leq & \frac{18L^2}{\sigma^A_{\min}b_2} \sum_{i=(n_k-1)q}^{k-1}\mathbb{E}\|x_{i+1}-x_i\|^2 + \frac{18\sigma^2}{\sigma^A_{\min}b_1}
 + (\frac{9L^2}{\sigma^A_{\min}}+\frac{3\sigma^2_{\max}(G)}{\sigma^A_{\min}\eta^2})\|x_k-x_{k-1}\|^2 \nonumber \\
 & + \frac{3\sigma^2_{\max}(G)}{\sigma^A_{\min}\eta^2}\|x_{k+1}-x_{k}\|^2.
\end{align}
\end{lemma}
Because the proof of the above lemma is the same to the proof of Lemma \ref{lem:A3}, so we omit this proof.

\begin{lemma} \label{lem:A6}
 Suppose the sequence $\{x_k,y_{[m]}^k,,z_k\}_{k=1}^K$ is generated from Algorithm \ref{alg:4},
 and define a \emph{Lyapunov} function $\Phi_k$ as follows:
 \begin{align}
 \Phi_k = \mathcal{L}_{\rho} (x_k,y_{[m]}^k,z_k) + (\frac{9L^2}{\sigma^A_{\min}\rho}+\frac{3\kappa_A\sigma^2_{\max}(G)}{\sigma^A_{\min}\eta^2\rho})\|x_k-x_{k-1}\|^2
 + \frac{2\kappa_AL^2}{\sigma^A_{\min}\rho b} \sum_{i=(n_k-1)q}^{k-1}\mathbb{E}\|x_{i+1}-x_i\|^2.
\end{align}
Let $b_2=q$, $\eta=\frac{2\alpha\sigma_{\min}(G)}{3L} \ (0<\alpha \leq 1)$ and $\rho= \frac{\sqrt{170\kappa_A}\kappa_GL}{\sigma^A_{\min}\alpha}$,
then we have
\begin{align} \label{eq:A63}
 \frac{1}{K}\sum_{i=0}^{K-1} ( \|x_{i+1} - x_{i}\|^2 + \sum_{j=1}^m \|y_j^i-y_j^{i+1}\|^2) \leq \frac{\mathbb{E} [\Phi_{0}]-\Phi^*}{K\gamma} + \frac{\delta^2}{2b_1L\gamma} + \frac{18\delta^2}{\sigma^A_{\min}b_1\rho\gamma},
\end{align}
where $\gamma = \min(\chi, \sigma_{\min}^H)$, $\chi \geq \frac{\sqrt{170\kappa_A}\kappa_GL}{4\alpha}$ and $\Phi^*$ is a lower bound of the function $\Phi_k$.
\end{lemma}

\begin{proof}
This proof is the same as the proof of Lemma \ref{lem:A4}.

By the optimal condition of step 9 in Algorithm \ref{alg:4},
we have, for $j\in [m]$
\begin{align}
0 & =(y_j^k-y_j^{k+1})^T\big(\partial g_j(y_j^{k+1}) - B^Tz_k + \rho B^T(Ax_k + \sum_{i=1}^jB_iy_i^{k+1} + \sum_{i=j+1}^mB_iy_i^{k}-c) + H_j(y_j^{k+1}-y_j^k)\big) \nonumber \\
& \leq g_j(y_j^k)- g_j(y_j^{k+1}) - (z_k)^T(B_jy_j^k-B_jy_j^{k+1}) + \rho(By_j^k-By_j^{k+1})^T(Ax_k + \sum_{i=1}^jB_iy_i^{k+1} + \sum_{i=j+1}^mB_iy_i^{k}-c) - \|y_j^{k+1}-y_j^k\|^2_{H_j} \nonumber \\
& = g_j(y_j^k)- g_j(y_j^{k+1}) - (z_k)^T(Ax_k+\sum_{i=1}^{j-1}B_iy_i^{k+1} + \sum_{i=j}^mB_iy_i^{k}-c) + (z_k)^T(Ax_k+\sum_{i=1}^jB_iy_i^{k+1}+ \sum_{i=j+1}^mB_iy_i^{k}-c) \nonumber \\
& \quad  + \frac{\rho}{2}\|Ax_k +\sum_{i=1}^{j-1}B_iy_i^{k+1} + \sum_{i=j}^mB_iy_i^{k}-c\|^2 - \frac{\rho}{2}\|Ax_k+\sum_{i=1}^jB_iy_i^{k+1}+ \sum_{i=j+1}^mB_iy_i^{k}-c\|^2 -\frac{\rho}{2}\|B_jy_j^k-B_jy_j^{k+1}\|^2 \nonumber \\
& \quad - \|y_j^{k+1}-y_j^k\|^2_{H_j} \nonumber \\
& \leq \underbrace{ f(x_k) + g_j(y_j^k) - (z_k)^T(Ax_k+\sum_{i=1}^{j-1}B_iy_i^{k+1} + \sum_{i=j}^mB_iy_i^{k}-c) + \frac{\rho}{2}\|Ax_k +\sum_{i=1}^{j-1}B_iy_i^{k+1} + \sum_{i=j}^mB_iy_i^{k}-c\|^2}_{\mathcal{L}_{\rho} (x_k,y_{[j-1]}^{k+1},y_{[j:m]}^k,z_k)}  - \|y_j^{k+1}-y_j^k\|^2_{H_j}  \nonumber \\
& \quad -\underbrace{(f(x_k) + g_j(y_j^{k+1}) - (z_k)^T(Ax_k+\sum_{i=1}^jB_iy_i^{k+1}+ \sum_{i=j+1}^mB_iy_i^{k}-c) + \frac{\rho}{2}\|Ax_k+\sum_{i=1}^jB_iy_i^{k+1}+ \sum_{i=j+1}^mB_iy_i^{k}-c\|^2}_{\mathcal{L}_{\rho} (x_k,y_{[j]}^{k+1},y_{[j+1:m]}^k,z_k)} \nonumber \\
& \leq \mathcal{L}_{\rho} (x_k,y_{[j-1]}^{k+1},y_{[j:m]}^k,z_k) - \mathcal{L}_{\rho} (x_k,y_{[j]}^{k+1},y_{[j+1:m]}^k,z_k)
- \sigma_{\min}(H_j)\|y_j^k-y_j^{k+1}\|^2,
\end{align}
where the first inequality holds by the convexity of function $g_j(y)$,
and the second equality follows by applying the equality
$(a-b)^Tb = \frac{1}{2}(\|a\|^2-\|b\|^2-\|a-b\|^2)$ on the term $(By_j^k-By_j^{k+1})^T(Ax_k + \sum_{i=1}^jB_iy_i^{k+1} + \sum_{i=j+1}^mB_iy_i^{k}-c)$.
Thus, we have, for all $j\in[m]$
\begin{align} \label{eq:A6-1}
 \mathcal{L}_{\rho} (x_k,y_{[j-1]}^{k+1},y_{[j:m]}^k,z_k) \leq \mathcal{L}_{\rho} (x_k,y_{[j]}^{k+1},y_{[j+1:m]}^k,z_k)
 - \sigma_{\min}(H_j)\|y_j^k-y_j^{k+1}\|^2.
\end{align}
Telescoping inequality \eqref{eq:A6-1} over $j$ from $1$ to $m$, we obtain
\begin{align} \label{eq:A50}
 \mathcal{L}_{\rho} (x_k,y^{k+1}_{[m]},z_k) \leq \mathcal{L}_{\rho} (x_k,y^k_{[m]},z_k)
 - \sigma_{\min}^H\sum_{j=1}^m \|y_j^k-y_j^{k+1}\|^2.
\end{align}

Using Assumption 1, we have
\begin{align} \label{eq:A51}
0 \leq f(x_k) - f(x_{k+1}) + \nabla f(x_k)^T(x_{k+1}-x_k) + \frac{L}{2}\|x_{k+1}-x_k\|^2.
\end{align}
Using the optimal condition of step 10 in Algorithm \ref{alg:4},
we have
\begin{align} \label{eq:A52}
 0 = (x_k-x_{k+1})^T \big( v_k - A^Tz_k + \rho A^T(Ax_{k+1} + \sum_{j=1}^mB_jy_j^{k+1}-c) + \frac{G}{\eta}(x_{k+1}-x_k) \big).
\end{align}
Combining \eqref{eq:A51} and \eqref{eq:A52}, we have
\begin{align}
 0 & \leq f(x_k) - f(x_{k+1}) + \nabla f(x_k)^T(x_{k+1}-x_k) + \frac{L}{2}\|x_{k+1}-x_k\|^2 \nonumber \\
 & \quad + (x_k-x_{k+1})^T \big( v_k - A^Tz_k + \rho A^T(Ax_{k+1} + \sum_{j=1}^mB_jy_j^{k+1}-c) + \frac{G}{\eta}(x_{k+1}-x_k) \big)  \nonumber \\
 & = f(x_k) - f(x_{k+1}) + \frac{L}{2}\|x_k-x_{k+1}\|^2 - \frac{1}{\eta}\|x_k - x_{k+1}\|^2_G + (x_k-x_{k+1})^T(v_k-\nabla f(x_k)) \nonumber \\
 & \quad -(z_k)^T(Ax_k-Ax_{k+1}) + \rho(Ax_k - Ax_{k+1})^T(Ax_{k+1} + \sum_{j=1}^mB_jy_j^{k+1}-c) \nonumber \\
 & = f(x_k) - f(x_{k+1}) + \frac{L}{2}\|x_k-x_{k+1}\|^2 - \frac{1}{\eta}\|x_k - x_{k+1}\|^2_G
 + (x_k-x_{k+1})^T\big(v_k-\nabla f(x_k)\big) -(z_k)^T(Ax_k + \sum_{j=1}^mB_jy_j^{k+1}-c)\nonumber \\
 & \quad  + (z_k)^T(Ax_{k+1} + \sum_{j=1}^mB_jy_j^{k+1}-c) + \frac{\rho}{2}\big(\|Ax_k + \sum_{j=1}^mB_jy_j^{k+1}-c\|^2
 - \|Ax_{k+1} + \sum_{j=1}^mB_jy_j^{k+1}-c\|^2 - \|Ax_k - Ax_{k+1}\|^2 \big) \nonumber \\
 & = \underbrace{f(x_k) -z_k^T(Ax_k + \sum_{j=1}^mB_jy_j^{k+1}-c) + \frac{\rho}{2}\|Ax_k + \sum_{j=1}^mB_jy_j^{k+1}-c\|^2}_{L_{\rho}(x_k,y_{[m]}^{k+1},z_k)} + \frac{L}{2}\|x_k-x_{k+1}\|^2
 + (x_k-x_{k+1})^T\big(v_k-\nabla f(x_k)\big)  \nonumber \\
 & \quad- \underbrace{\big(f(x_{k+1})
 - z_k^T(Ax_{k+1}+ \sum_{j=1}^mB_jy_j^{k+1}-c) + \frac{\rho}{2} \|Ax_{k+1} + \sum_{j=1}^mB_jy_j^{k+1}-c\|^2\big)}_{L_{\rho}(x_{k+1},y_{[m]}^{k+1},z_{k})} - \frac{1}{\eta}\|x_k - x_{k+1}\|^2_G
 - \frac{\rho}{2}\|Ax_k - Ax_{k+1}\|^2 \nonumber \\
 & \leq \mathcal{L}_{\rho} (x_k,y_{[m]}^{k+1},z_k) -  \mathcal{L}_{\rho} (x_{k+1},y_{[m]}^{k+1},z_k)
 - (\frac{\sigma_{\min}(G)}{\eta}+\frac{\rho \sigma^A_{\min}}{2}-\frac{L}{2}) \|x_{k+1} - x_{k}\|^2 + (x_k-x_{k+1})^T(v_k-\nabla f(x_k)) \nonumber \\
 & \leq \mathcal{L}_{\rho} (x_k,y_{[m]}^{k+1},z_k) -  \mathcal{L}_{\rho} (x_{k+1},y_{[m]}^{k+1},z_k)
 - (\frac{\sigma_{\min}(G)}{\eta}+\frac{\rho \sigma^A_{\min}}{2}-L ) \|x_{k+1} - x_{k}\|^2 + \frac{1}{2L}\|v_k - \nabla f(x_k)\|^2 \nonumber \\
 & \leq \mathcal{L}_{\rho} (x_k,y_{[m]}^{k+1},z_k) -  \mathcal{L}_{\rho} (x_{k+1},y_{[m]}^{k+1},z_k)
 - (\frac{\sigma_{\min}(G)}{\eta}+\frac{\rho \sigma^A_{\min}}{2}-L ) \|x_{k+1} - x_{k}\|^2 + \frac{L}{2b_2}\sum_{i=(n_k-1)q}^{k-1}\mathbb{E}\|x_{i+1}-x_i\|^2+\frac{\sigma^2}{2b_1L}, \nonumber
\end{align}
where the second equality follows by applying the equality
$(a-b)^Tb = \frac{1}{2}(\|a\|^2-\|b\|^2-\|a-b\|^2)$ over the
term $(Ax_k - Ax_{k+1})^T(Ax_{k+1}+\sum_{j=1}^mB_jy_j^{k+1}-c)$; the third inequality
follows by the inequality $a^Tb \leq \frac{1}{2L}\|a\|^2 + \frac{L}{2}\|b\|^2$, and the forth inequality holds by the inequality \eqref{eq:A13}.
It follows that
\begin{align} \label{eq:A53}
\mathcal{L}_{\rho} (x_{k+1},y_{[m]}^{k+1},z_k) \leq &\mathcal{L}_{\rho} (x_k,y_{[m]}^{k+1},z_k)
  - (\frac{\sigma_{\min}(G)}{\eta}+\frac{\rho \sigma^A_{\min}}{2}-L ) \|x_{k+1} - x_{k}\|^2 \nonumber \\
& + \frac{L}{2b_2}\sum_{i=(n_k-1)q}^{k-1}\mathbb{E}\|x_{i+1}-x_i\|^2 + \frac{\sigma^2}{2b_1L}.
\end{align}

Using the step 11 in Algorithm \ref{alg:4}, we have
\begin{align} \label{eq:A54}
\mathcal{L}_{\rho} (x_{k+1},y_{[m]}^{k+1},z_{k+1}) -\mathcal{L}_{\rho} (x_{k+1},y_{[m]}^{k+1},z_k)
& = \frac{1}{\rho}\|z_{k+1}-z_k\|^2 \nonumber \\
& \leq \frac{18L^2}{\sigma^A_{\min}b_2\rho} \sum_{i=(n_k-1)q}^{k-1}\mathbb{E}\|x_{i+1}-x_i\|^2
 + (\frac{9L^2}{\sigma^A_{\min}\rho}+\frac{3\sigma^2_{\max}(G)}{\sigma^A_{\min}\eta^2\rho})\|x_k-x_{k-1}\|^2 \nonumber \\
& \quad + \frac{3\sigma^2_{\max}(G)}{\sigma^A_{\min}\eta^2\rho}\|x_{k+1}-x_{k}\|^2 + \frac{18\sigma^2}{\sigma^A_{\min}b_1\rho},
\end{align}
where the above inequality holds by Lemma \ref{lem:A3}.

Combining \eqref{eq:A50}, \eqref{eq:A53} and \eqref{eq:A54}, we have
\begin{align} \label{eq:A55}
\mathcal{L}_{\rho} (x_{k+1},y_{[m]}^{k+1},z_{k+1}) & \leq \mathcal{L}_{\rho} (x_k,y_{[m]}^k,z_k)
  - \sigma_{\min}^H\sum_{j=1}^m \|y_j^k-y_j^{k+1}\|^2 - (\frac{\sigma_{\min}(G)}{\eta}+\frac{\rho \sigma^A_{\min}}{2}-L - \frac{3\sigma^2_{\max}(G)}{\sigma^A_{\min}\eta^2\rho}) \|x_{k+1} - x_{k}\|^2  \nonumber \\
& \quad + (\frac{L}{2b_2} + \frac{18L^2}{\sigma^A_{\min}b_2\rho})\sum_{i=(n_k-1)q}^{k-1}\mathbb{E}\|x_{i+1}-x_i\|^2
 + (\frac{9L^2}{\sigma^A_{\min}\rho}+\frac{3\sigma^2_{\max}(G)}{\sigma^A_{\min}\eta^2\rho})\|x_k-x_{k-1}\|^2 \nonumber \\
& \quad + \frac{\sigma^2}{2b_1L} + \frac{18\sigma^2}{\sigma^A_{\min}b_1\rho}.
\end{align}

Next, we define a \emph{Lyapunov} function $R_k$:
\begin{align}
 \Phi_k = \mathcal{L}_{\rho} (x_k,y_{[m]}^k,z_k) + (\frac{9L^2}{\sigma^A_{\min}\rho}+\frac{3\kappa_A\sigma^2_{\max}(G)}{\sigma^A_{\min}\eta^2\rho})\|x_k-x_{k-1}\|^2 + \frac{2\kappa_AL^2}{\sigma^A_{\min}\rho b_2} \sum_{i=(n_k-1)q}^{k-1}\mathbb{E}\|x_{i+1}-x_i\|^2,
\end{align}
where $\kappa_A = \frac{\sigma_{\max}^A}{\sigma_{\min}^A} \geq 1$. Since
$$\sum_{i=(n_k-1)q}^{k}\mathbb{E}\|x_{i+1}-x_i\|^2 = \sum_{i=(n_k-1)q}^{k-1}\mathbb{E}\|x_{i+1}-x_i\|^2 + \|x_{k+1}-x_k\|^2,$$
the inequality \eqref{eq:A55} can be rewrite as follows:
\begin{align} \label{eq:A56}
\Phi_{k+1} & \leq \Phi_k - \sigma_{\min}^H\sum_{j=1}^m \|y_j^k-y_j^{k+1}\|^2 - (\frac{\sigma_{\min}(G)}{\eta}+\frac{\rho \sigma^A_{\min}}{2}-L
- \frac{6\kappa_A\sigma^2_{\max}(G)}{\sigma^A_{\min}\eta^2\rho}-\frac{9L^2}{\sigma^A_{\min}\rho}-\frac{2\kappa_AL^2}{\sigma^A_{\min}\rho b_2}) \|x_{k+1} - x_{k}\|^2  \nonumber \\
& \quad + (\frac{L}{2b_2} + \frac{18L^2}{\sigma^A_{\min}b_2\rho})\sum_{i=(n_k-1)q}^{k-1}\mathbb{E}\|x_{i+1}-x_i\|^2 + \frac{\delta^2}{2b_1L} + \frac{18\delta^2}{\sigma^A_{\min}b_1\rho}.
\end{align}

Then telescoping equality \eqref{eq:A56} over $k$ from $(n_k-1)q$ to $k$ where $k \leq n_kq-1$ and
let $n_j=n_k$ for $(n_k-1)q \leq j \leq n_kq-1$, we have
\begin{align} \label{eq:A57}
\mathbb{E} [\Phi_{k+1}] &\leq \mathbb{E} [\Phi_{(n_k-1)q}]
 - (\frac{\sigma_{\min}(G)}{\eta}+\frac{\rho \sigma^A_{\min}}{2}-L - \frac{6\kappa_A\sigma^2_{\max}(G)}{\sigma^A_{\min}\eta^2\rho}-\frac{9L^2}{\sigma^A_{\min}\rho}-\frac{2\kappa_AL^2}{\sigma^A_{\min}\rho b_2})
 \sum_{j=(n_k-1)q}^k \|x_{j+1} - x_{j}\|^2  \nonumber \\
& \quad - \sigma_{\min}^H \sum_{i=(n_k-1)q}^{k}\sum_{j=1}^m \|y_j^i-y_j^{i+1}\|^2 + (\frac{L}{2b_2} + \frac{18L^2}{\sigma^A_{\min}b_2\rho})\sum_{j=(n_k-1)q}^{k}\sum_{i=(n_k-1)q}^{k-1}\mathbb{E}\|x_{i+1}-x_i\|^2 + \frac{\delta^2}{2b_1L} + \frac{18\delta^2}{\sigma^A_{\min}b_1\rho} \nonumber \\
& \leq \mathbb{E} [\Phi_{(n_k-1)q}]
  - (\frac{\sigma_{\min}(G)}{\eta}+\frac{\rho \sigma^A_{\min}}{2}-L - \frac{6\kappa_A\sigma^2_{\max}(G)}{\sigma^A_{\min}\eta^2\rho}-\frac{9L^2}{\sigma^A_{\min}\rho}-\frac{2\kappa_AL^2}{\sigma^A_{\min}\rho b_2})
  \sum_{i=(n_k-1)q}^k \|x_{i+1} - x_{i}\|^2  \nonumber \\
& \quad - \sigma_{\min}^H \sum_{i=(n_k-1)q}^{k}\sum_{j=1}^m \|y_j^i-y_j^{i+1}\|^2+ (\frac{Lq}{2b_2} + \frac{18L^2q}{\sigma^A_{\min}b_2\rho})\sum_{i=(n_k-1)q}^{k}\mathbb{E}\|x_{i+1}-x_i\|^2 + \frac{\delta^2}{2b_1L} + \frac{18\delta^2}{\sigma^A_{\min}b_1\rho} \nonumber \\
& = \mathbb{E} [\Phi_{(n_k-1)q}] - (\frac{\sigma_{\min}(G)}{\eta}+\frac{\rho \sigma^A_{\min}}{2}-L - \frac{6\kappa_A\sigma^2_{\max}(G)}{\sigma^A_{\min}\eta^2\rho} -\frac{9L^2}{\sigma^A_{\min}\rho} -\frac{2\kappa_AL^2}{\sigma^A_{\min}\rho b_2} -\frac{Lq}{2b_2} - \frac{18L^2q}{\sigma^A_{\min}b_2\rho} ) \sum_{i=(n_k-1)q}^k \|x_{i+1} - x_{i}\|^2\nonumber \\
& \quad - \sigma_{\min}^H \sum_{i=(n_k-1)q}^{k}\sum_{j=1}^m \|y_j^i-y_j^{i+1}\|^2+ \frac{\delta^2}{2b_1L} + \frac{18\delta^2}{\sigma^A_{\min}b_1\rho},
\end{align}
where the second inequality holds by the fact that
\begin{align}
\sum_{j=(n_k-1)q}^{k}\sum_{i=(n_k-1)q}^{k-1}\mathbb{E}\|x_{i+1}-x_i\|^2 \leq \sum_{j=(n_k-1)q}^{k}\sum_{i=(n_k-1)q}^{k}\mathbb{E}\|x_{i+1}-x_i\|^2 \leq q\sum_{i=(n_k-1)q}^{k}\mathbb{E}\|x_{i+1}-x_i\|^2. \nonumber
\end{align}

Since $b_2=q$, we have
\begin{align}
 \chi & = \frac{\sigma_{\min}(G)}{\eta}+\frac{\rho \sigma^A_{\min}}{2}-L - \frac{6\kappa_A\sigma^2_{\max}(G)}{\sigma^A_{\min}\eta^2\rho}
 -\frac{9L^2}{\sigma^A_{\min}\rho} - \frac{2\kappa_AL^2}{\sigma^A_{\min}\rho b_2} -\frac{Lq}{2b_2}  - \frac{18L^2q}{\sigma^A_{\min}\rho b_2} \nonumber \\
 & = \underbrace{\frac{\sigma_{\min}(G)}{\eta} - \frac{3L}{2}}_{L_1} + \underbrace{\frac{\rho \sigma^A_{\min}}{2} - \frac{6\kappa_A\sigma^2_{\max}(G)}{\sigma^A_{\min}\eta^2\rho}
 -\frac{27L^2}{\sigma^A_{\min}\rho}- \frac{2\kappa_AL^2}{\sigma^A_{\min}\rho}}_{L_2}.
\end{align}
Given $0< \eta \leq \frac{2\sigma_{\min}(G)}{3L}$, we have $L_1\geq 0$. Further, let $\eta = \frac{2\alpha\sigma_{\min}(G)}{3L} \ (0<\alpha \leq 1)$ and $\rho = \frac{\sqrt{170\kappa_A}\kappa_GL}{\sigma^A_{\min}\alpha}$,
we have
\begin{align}
L_2 & = \frac{\rho \sigma^A_{\min}}{2} - \frac{6\kappa_A\sigma^2_{\max}(G)}{\sigma^A_{\min}\eta^2\rho}
 -\frac{27L^2}{\sigma^A_{\min}\rho} - \frac{2\kappa_AL^2}{\sigma^A_{\min}\rho }\nonumber \\
 & = \frac{\rho \sigma^A_{\min}}{2} - \frac{27\kappa_A\kappa^2_GL^2}{2\sigma^A_{\min}\rho\alpha^2}
 -\frac{27L^2}{\sigma^A_{\min}\rho} - \frac{2\kappa_AL^2}{\sigma^A_{\min}\rho }\nonumber \\
 & \geq \frac{\rho \sigma^A_{\min}}{2} - \frac{27\kappa_A\kappa^2_GL^2}{2\sigma^A_{\min}\rho\alpha^2}
 -\frac{27\kappa_A\kappa^2_GL^2}{\sigma^A_{\min}\rho\alpha^2} - \frac{2\kappa_A\kappa^2_GL^2}{\sigma^A_{\min}\rho \alpha^2}\nonumber \\
 & = \frac{\rho \sigma^A_{\min}}{4} + \underbrace{\frac{\rho \sigma^A_{\min}}{4} - \frac{85\kappa_A\kappa^2_GL^2}{2\sigma^A_{\min}\rho\alpha^2}}_{\geq 0} \nonumber \\
 & \geq \frac{\sqrt{170\kappa_A}\kappa_GL}{4\alpha},
\end{align}
where the first inequality follows by $\kappa_G \geq 1$, $\kappa_A \geq 1$ and $0 < \alpha \leq 1$; and the third equality holds by $\rho = \frac{\sqrt{170\kappa_A}\kappa_GL}{\sigma^A_{\min}\alpha}$.
It follows that $\chi \geq \frac{\sqrt{170\kappa_A}\kappa_GL}{4\alpha}$.

Using \eqref{eq:A16}, we have
\begin{align}
 z_{k+1} = (A^T)^+(v_{k} + \frac{G}{\eta}(x_{k+1}-x_k)),
\end{align}
where $(A^T)^+$ is the pseudoinverse of $A^T$. Due to that $A$ is full row rank,
we have $(A^T)^+ = (AA^T)^{-1}A$.
It follows that $\sigma_{\max}((A^T)^+)^T(A^T)^+)\leq \frac{\sigma_{\max}^A}{(\sigma^A_{\min})^2}=\frac{\kappa_A}{\sigma^A_{\min}}$.

Then we have
\begin{align}
\mathcal{L}_{\rho} (x_{k+1},y_{[m]}^{k+1},z_{k+1})
& = f(x_{k+1}) + \sum_{j=1}^mg_j(y_j^{k+1}) - z_{k+1}^T(Ax_{k+1} + \sum_{j=1}^mB_jy_j^{k+1} -c) + \frac{\rho}{2}\|Ax_{k+1} + \sum_{j=1}^mB_jy_j^{k+1} -c\|^2 \nonumber \\
& = f(x_{k+1}) + \sum_{j=1}^mg_j(y_j^{k+1}) -  \langle(A^T)^+(v_k + \frac{G}{\eta}(x_{k+1}-x_k)), Ax_{k+1} + \sum_{j=1}^mB_jy_j^{k+1} -c\rangle \nonumber \\
& \quad + \frac{\rho}{2}\|Ax_{k+1} + \sum_{j=1}^mB_jy_j^{k+1} -c\|^2 \nonumber \\
& = f(x_{k+1}) + \sum_{j=1}^mg_j(y_j^{k+1}) - \langle(A^T)^+(v_{k} - \nabla f(x_{k}) + \nabla f(x_{k})+ \frac{G}{\eta}(x_{k+1}-x_k)), Ax_{k+1} + \sum_{j=1}^mB_jy_j^{k+1} -c\rangle  \nonumber \\
& \quad +  \frac{\rho}{2}\|Ax_{k+1} + \sum_{j=1}^mB_jy_j^{k+1} -c\|^2 \nonumber \\
& \geq f(x_{k+1}) + \sum_{j=1}^mg_j(y_j^{k+1}) - \frac{2\kappa_A}{\sigma^A_{\min}\rho}\|v_k - \nabla f(x_{k})\|^2 - \frac{2\kappa_A}{\sigma^A_{\min}\rho}\|\nabla f(x_{k})\|^2
- \frac{2\kappa_A\sigma^2_{\max}(G)}{\sigma^A_{\min}\eta^2\rho}\|x_{k+1}-x_k\|^2 \nonumber \\
& \quad + \frac{\rho}{8}\|Ax_{k+1} + \sum_{j=1}^mB_jy_j^{k+1} -c\|^2 \nonumber\\
& \geq f(x_{k+1}) + \sum_{j=1}^mg_j(y_j^{k+1}) - \frac{2\kappa_AL^2}{\sigma^A_{\min}b_2\rho} \sum_{i=(n_k-1)q}^{k-1}\mathbb{E}\|x_{i+1}-x_i\|^2
- \frac{4\kappa_A\delta^2}{\sigma^A_{\min}b_1\rho}- \frac{2\kappa_A\delta^2}{\sigma^A_{\min}\rho} \nonumber \\
& \quad - \frac{2\kappa_A\sigma^2_{\max}(G)}{\sigma^A_{\min}\eta^2\rho}\|x_{k+1}-x_k\|^2
\end{align}
where the first inequality is obtained by applying $ \langle a, b\rangle \leq \frac{1}{2\beta}\|a\|^2 + \frac{\beta}{2}\|b\|^2$ to the terms
$\langle(A^T)^+(v_{k} - \nabla f(x_{k})), Ax_{k+1} + \sum_{j=1}^mB_jy_j^{k+1} -c\rangle$, $\langle(A^T)^+v_{k}, Ax_{k+1} + \sum_{j=1}^mB_jy_j^{k+1} -c\rangle $ and
$\langle(A^T)^+\frac{G}{\eta}(x_{k+1}-x_k), Ax_{k+1} + \sum_{j=1}^mB_jy_j^{k+1} -c\rangle$ with $\beta = \frac{\rho}{4}$, respectively.
The second inequality follows by the inequality \eqref{eq:A13} and Assumption 3.
Therefore, we have, for $k=0,1,2,\cdots$
\begin{align}
 \Phi_{k+1}\geq f^* + \sum_{j=1}^mg_j^* - \frac{2(2+b_1)\kappa_A\delta^2}{\sigma^A_{\min}\rho b_1}.
\end{align}
It follows that the function $\Phi_k$ is bounded from below. Let $\Phi^*$ denotes a low bound of function $\Phi_k$.

Further, telescoping equality \eqref{eq:A57} over $k$ from $0$ to $K$,
we have
\begin{align}
\mathbb{E} [\Phi_{K}] - \mathbb{E} [\Phi_{0}] & = (\mathbb{E} [\Phi_{q}] - \mathbb{E} [\Phi_{0}]) + (\mathbb{E} [\Phi_{2q}] - \mathbb{E} [\Phi_{q}]) + \cdots + (\mathbb{E} [\Phi_{K}] - \mathbb{E} [\Phi_{(n_k-1)q}]) \nonumber \\
& \leq -\sum_{i=0}^{q-1} ( \chi \|x_{i+1} - x_{i}\|^2 + \sigma_{\min}^H\sum_{j=1}^m \|y_j^i-y_j^{i+1}\|^2)-  \sum_{i=q}^{2q-1} (\chi \|x_{i+1} - x_{i}\|^2 + \sigma_{\min}^H\sum_{j=1}^m \|y_j^i-y_j^{i+1}\|^2) \nonumber \\
& - \cdots - \sum_{i=(n_k-1)q}^{K-1} ( \chi\|x_{i+1} - x_{i}\|^2 + \sigma_{\min}^H \sum_{j=1}^m \|y_j^i-y_j^{i+1}\|^2)
+ \frac{K\delta^2}{2b_1L} + \frac{K18\delta^2}{\sigma^A_{\min}b_1\rho} \nonumber \\
& = -\sum_{i=0}^{K-1} ( \chi \|x_{i+1} - x_{i}\|^2 + \sigma_{\min}^H\sum_{j=1}^m \|y_j^i-y_j^{i+1}\|^2) + \frac{K\delta^2}{2b_1L} + \frac{K18\delta^2}{\sigma^A_{\min}b_1\rho}.
\end{align}
Thus, the above inequality implies that
\begin{align} \label{eq:A63}
 \frac{1}{K}\sum_{i=0}^{K-1} ( \|x_{i+1} - x_{i}\|^2 + \sum_{j=1}^m \|y_j^i-y_j^{i+1}\|^2) \leq \frac{\mathbb{E} [\Phi_{0}]-\Phi^*}{K\gamma} + \frac{\delta^2}{2b_1L\gamma} + \frac{18\delta^2}{\sigma^A_{\min}b_1\rho\gamma},
\end{align}
where $\gamma = \min(\chi, \sigma_{\min}^H)$ and $\chi \geq \frac{\sqrt{170\kappa_A}\kappa_GL}{4\alpha}$.

\end{proof}

\begin{theorem}
 Suppose the sequence $\{x_k,y_{[m]}k,z_k)_{k=1}^K$ is generated from Algorithm \ref{alg:4}, and let $b_2=q=\sqrt{b_1}$, $\eta = \frac{2\alpha\sigma_{\min}(G)}{3L} \ (0<\alpha \leq 1)$,
 $\rho = \frac{\sqrt{170\kappa_A}\kappa_GL}{\sigma^A_{\min}\alpha}$, and
  \begin{align}
   \nu_1 = m\big(\rho^2\sigma^B_{\max}\sigma^A_{\max} + \rho^2(\sigma^B_{\max})^2 + \sigma^2_{\max}(H)\big), \ \nu_2 = 3(L^2+ \frac{\sigma^2_{\max}(G)}{\eta^2}),
   \ \nu_3 = \frac{18 L^2 }{\sigma^A_{\min} \rho^2} + \frac{3\sigma^2_{\max}(G) }{\sigma^A_{\min}\eta^2\rho^2},
 \end{align}
 then we have
 \begin{align}
  \frac{1}{K}\sum_{k=1}^K \mathbb{E}\big[ \mbox{dist}(0,\partial L(x_k,y_{[m]}^k,z_k))^2\big] \leq \frac{\nu_{\max}}{K}\sum_{k=1}^{K-1} \theta_k \leq  \frac{3\nu_{\max}(\Phi_{0}-\Phi^*)}{K\gamma}
  + \frac{3\nu_{\max}\delta^2}{b_1\gamma}(\frac{1}{2L}+\frac{18}{\sigma^A_{\min}\rho}),
 \end{align}
 where $\gamma \geq \frac{\sqrt{170\kappa_A}\kappa_GL}{4\alpha}$, $\nu_{\max}=\max\{\nu_1,\nu_2,\nu_3\}$ and
 $\Phi^*$ is a lower bound of the function $\Phi_k$.
 It implies that $K$ and $b_1$ satisfy
 \begin{align}
 K = \frac{6\nu_{\max}(\Phi_{0}-\Phi^*)}{\epsilon \gamma}, \quad    b_1 = \frac{6\nu_{\max}\delta^2}{\epsilon\gamma}(\frac{1}{2L}+\frac{18\alpha}{\sqrt{170\kappa_A}\kappa_G L})\nonumber
 \end{align}
 then $(x_{k^*},y_{[m]}^{k^*},z_{k^*})$ is an $\epsilon$-approximate stationary point of \eqref{eq:2}, where $k^* = \mathop{\arg\min}_{k}\theta_{k}$.
\end{theorem}

\begin{proof}
We begin with defining a useful variable $\theta_k=\|x_{k+1}-x_{k}\|^2+\|x_{k}-x_{k-1}\|^2+\frac{1}{q}\sum_{i=(n_k-1)q}^k \|x_{i+1}-x_i\|^2 + \sum_{j=1}^m\|y_j^k - y_j^{k+1}\|^2$.
Next, by the optimal condition of the step 9 in Algorithm \ref{alg:4}, we have, for all $i\in [m]$
\begin{align}
  \mathbb{E}\big[\mbox{dist}(0,\partial_{y_j} L(x,y_{[m]},z))^2\big]_{k+1} & = \mathbb{E}\big[\mbox{dist} (0, \partial g_j(y_j^{k+1})-B_j^Tz_{k+1})^2\big] \nonumber \\
 & = \|B_j^Tz_k -\rho B_j^T(Ax_k + \sum_{i=1}^jB_iy_i^{k+1} + \sum_{i=j+1}^mB_iy_i^{k} -c) - H_j(y_j^{k+1}-y_j^k) -B_j^Tz_{k+1}\|^2 \nonumber \\
 & = \|\rho B_j^TA(x_{k+1}-x_{k}) + \rho B_j^T \sum_{i=j+1}^m B_i (y_i^{k+1}-y_i^{k})- H_j(y_j^{k+1}-y_j^k) \|^2 \nonumber \\
 & \leq m\rho^2\sigma^{B_j}_{\max}\sigma^A_{\max}\|x_{k+1}-x_k\|^2 + m\rho^2\sigma^{B_j}_{\max}\sum_{i=j+1}^m \sigma^{B_i}_{\max}\|y_i^{k+1}-y_i^{k}\|^2 + m\sigma^2_{\max}(H_j)\|y_j^{k+1}-y_j^k\|^2\nonumber \\
 & \leq m\big(\rho^2\sigma^B_{\max}\sigma^A_{\max} + \rho^2(\sigma^B_{\max})^2 + \sigma^2_{\max}(H)\big) \theta_{k},
\end{align}
where the first inequality follows by the inequality $\|\frac{1}{n}\sum_{i=1}^n z_i\|^2 \leq \frac{1}{n}\sum_{i=1}^n \|z_i\|^2$.

By the step 10 of Algorithm \ref{alg:4}, we have
\begin{align}
 \mathbb{E}\big[\mbox{dist}(0,\nabla_x L(x,y_{[m]},z))^2\big]_{k+1} & = \mathbb{E}\|A^Tz_{k+1}-\nabla f(x_{k+1})\|^2  \nonumber \\
 & = \mathbb{E}\|v_k - \nabla f(x_{k+1}) - \frac{G}{\eta} (x_k - x_{k+1})\|^2 \nonumber \\
 & = \mathbb{E}\|v_k - \nabla f(x_{k}) + \nabla f(x_{k})- \nabla f(x_{k+1}) - \frac{G}{\eta}(x_k-x_{k+1})\|^2  \nonumber \\
 & \leq \sum_{i=(n_k-1)q}^{k-1}\frac{3L^2}{b}\mathbb{E}\|x_{i+1}-x_i\|^2 + 3(L^2+ \frac{\sigma^2_{\max}(G)}{\eta^2})\|x_k-x_{k+1}\|^2  \nonumber \\
 & \leq 3(L^2+ \frac{\sigma^2_{\max}(G)}{\eta^2})\theta_{k},
\end{align}
where the second inequality holds by $b_2=q$.

By the step 11 of Algorithm \ref{alg:4}, we have
\begin{align}
 \mathbb{E}\big[\mbox{dist}(0,\nabla_z L(x,y_{[m]},z))^2\big]_{k+1} & = \mathbb{E}\| Ax_{k+1}+\sum_{j=1}^mB_jy_j^{k+1}-c \|^2 \nonumber \\
 &= \frac{1}{\rho^2} \mathbb{E} \|z_{k+1}-z_k\|^2  \nonumber \\
 & \leq \frac{18L^2}{\sigma^A_{\min}b\rho^2} \sum_{i=(n_k-1)q}^{k-1}\mathbb{E}\|x_{i+1}-x_i\|^2
 + (\frac{9L^2}{\sigma^A_{\min}\rho^2}+\frac{3\sigma^2_{\max}(G)}{\sigma^A_{\min}\eta^2\rho^2})\|x_k-x_{k-1}\|^2 \nonumber \\
 & + \frac{3\sigma^2_{\max}(G)}{\sigma^A_{\min}\eta^2\rho^2}\|x_{k+1}-x_{k}\|^2 \nonumber \\
 & \leq \big( \frac{18 L^2 }{\sigma^A_{\min} \rho^2} + \frac{3\sigma^2_{\max}(G) }{\sigma^A_{\min}\eta^2\rho^2} \big) \theta_{k} ,
\end{align}
where the second inequality holds by $b_2=q$.

By \eqref{eq:A63}, we have
\begin{align} \label{eq:A63}
 \frac{1}{K}\sum_{i=0}^{K-1} ( \|x_{i+1} - x_{i}\|^2 + \sum_{j=1}^m \|y_j^i-y_j^{i+1}\|^2) \leq \frac{\mathbb{E} [\Phi_{0}]-\Phi^*}{K\gamma} + \frac{\delta^2}{2b_1L\gamma} + \frac{18\delta^2}{\sigma^A_{\min}b_1\rho\gamma},
\end{align}
where $\gamma = \min(\chi, \sigma_{\min}^H)$ and $\chi \geq \frac{\sqrt{170\kappa_A}\kappa_GL}{4\alpha}$. Since
\begin{align}
\sum_{k=0}^{K-1}\sum_{i=(n_k-1)q}^k \|x_{i+1}-x_i\|^2 \leq q\sum_{k=0}^{K-1} \|x_{k+1}-x_k\|^2
\end{align}
we have
\begin{align}
  \frac{1}{K}\sum_{k=1}^K\mathbb{E}\big[ \mbox{dist}(0,\partial L(x_k,y_{[m]}^k,z_k))^2\big] \leq \frac{\nu_{\max}}{K}\sum_{k=1}^{K-1} \theta_k \leq  \frac{3\nu_{\max}(\Phi_{0}-\Phi^*)}{K\gamma} + \frac{3\nu_{\max}\delta^2}{b_1\gamma}(\frac{1}{2L} + \frac{18 }{\sigma^A_{\min}\rho}),
\end{align}
where $\gamma \geq \frac{\sqrt{170\kappa_A}\kappa_GL}{4\alpha}$ and $\nu_{\max}=\max\{\nu_1,\nu_2,\nu_3\}$ with
\begin{align}
   \nu_1 =  m\big(\rho^2\sigma^B_{\max}\sigma^A_{\max} + \rho^2(\sigma^B_{\max})^2 + \sigma^2_{\max}(H)\big), \ \nu_2 = 3(L^2+ \frac{\sigma^2_{\max}(G)}{\eta^2}),
   \ \nu_3 = \frac{18 L^2 }{\sigma^A_{\min} \rho^2} + \frac{3\sigma^2_{\max}(G) }{\sigma^A_{\min}\eta^2\rho^2}.
 \end{align}
Given  $\eta = \frac{2\alpha\sigma_{\min}(G)}{3L} \ (0<\alpha \leq 1)$ and $\rho = \frac{\sqrt{170\kappa_A}\kappa_GL}{\sigma^A_{\min}\alpha}$, since $m$ is relatively small,
it easy verifies that $\gamma = O(1)$
and $\nu_{\max}=O(1)$, which are independent on $b_1$ and $K$. Thus, we obtain
\begin{align}
  \frac{1}{K}\sum_{k=1}^K \mathbb{E}\big[ \mbox{dist}(0,\partial L(x_k,y_{[m]}^k,z_k))^2\big] \leq O(\frac{1}{K}) + O(\frac{1}{b_1}).
\end{align}

\end{proof}

\newpage

\subsection{ Theoretical Analysis of the non-convex SVRG-ADMM}
\label{Appendix:A3}
In this subsection, we first extend the existing nonconvex SVRG-ADMM \cite{zheng2016stochastic,huang2016stochastic} to the multi-blocks setting for solving the problem \eqref{eq:2},
which is summarized in Algorithm \ref{alg:3}. Then we afresh study the convergence analysis of this non-convex SVRG-ADMM.

\begin{algorithm}[htb]
   \caption{ SVRG-ADMM for Nonconvex Optimization}
   \label{alg:3}
\begin{algorithmic}[1]
   \STATE {\bfseries Input:} $M$, $T$, $S=[T/M]$, $\rho>0$ and $H_j\succ0$ for all $j\in [m]$;
   \STATE {\bfseries Initialize:} $x_0^1$, $\tilde{x}^1=x_0^1$, $z_0^1$ and $y_j^{0,1}$ for all $j\in [m]$;
   \FOR {$s=1,2,\cdots,S$}
   \STATE{} $\nabla f(\tilde{x}^{s})=\frac{1}{n}\sum_{i=1}^n\nabla f_i(\tilde{x}^{s})$;
   \FOR {$t=0,1,\cdots,M-1$}
   \STATE{} Uniformly random pick a mini-batch $\mathcal{I}_t$ (with replacement) from $\{1,2,\cdots,n\}$ with $|\mathcal{I}_t|=b$, and compute
            $$v_{t}^{s} = \nabla f_{\mathcal{I}_t}(x_{t}^{s})-\nabla f_{\mathcal{I}_t}(\tilde{x}^s)+\nabla f(\tilde{x}^s);$$
   \STATE{} $ y^{s,t+1}_j= \arg\min_{y_j} \mathcal {L}_{\rho}(x^s_t,y^{s,t+1}_{[j-1]},y_j,y^{s,t}_{[j+1:m]},z_t) + \frac{1}{2}\|y_j-y^{s,t}_j\|_{H_j}^2$ for all $j\in [m]$;
   \STATE{} $x^{s}_{t+1} = \arg\min_x \hat{\mathcal {L}}_{\rho}\big(x,y^s_{t+1}, z_t^s, v_{t}^{s}\big)$;
   \STATE{} $z_{t+1}^{s} = z_{t}^{s}-\rho(Ax_{t+1}^{s} + \sum_{j=1}^mB_jy_j^{s,t+1}-c)$;
   \ENDFOR
   \STATE{} $\tilde{x}^{s+1}=x_0^{s+1}=x_{M}^{s}$, $y_j^{s+1,0}=y_j^{s,M}$ for all $j\in [m]$, $z_0^{s+1}=z_{M}^{s}$;
   \ENDFOR
   \STATE {\bfseries Output \ (in theory):} Chosen uniformly random from $\{(x_{t}^s,y_{[m]}^{s,t},z_{t}^s)_{t=1}^M\}_{s=1}^S$.
   \STATE {\bfseries Output \ (in practice):} $\{x_{T}^S,y_{[m]}^{S,T},z_{T}^S\}$.
\end{algorithmic}
\end{algorithm}

\begin{lemma} \label{lem:A9}
Suppose the sequence $\big\{(x^s_t,y_{[m]}^{s,t},z^s_t)_{t=1}^M\big\}_{s=1}^S$ is generated by Algorithm \ref{alg:3}. The following inequality holds
 \begin{align}
 \mathbb{E}\|z^s_{t+1}-z^s_{t}\|^2 \leq & \frac{9L^2 }{\sigma^A_{\min} b} \big( \|x^s_t - \tilde{x}^s\|^2
 + \|x^s_{t-1} - \tilde{x}^s\|^2\big)
  + \frac{3\sigma^2_{\max}(G)}{\sigma^A_{\min}\eta^2}\|x^s_{t+1}-x^s_t\|^2 \nonumber \\
  & + (\frac{3\sigma^2_{\max}(G)}{\sigma^A_{\min}\eta^2}+\frac{9L^2}{\sigma^A_{\min}})\|x^s_{t}-x^s_{t-1}\|^2.
 \end{align}
\end{lemma}

\begin{proof}
 Using the optimal condition for the step 8 of Algorithm \ref{alg:3}, we have
 \begin{align}
   v^s_t + \frac{1}{\eta}G(x^s_{t+1}-x^s_t) - A^Tz^s_t + \rho A^T(Ax^s_{t+1} + \sum_{j=1}^mB_jy_j^{s,t+1}-c) = 0,
 \end{align}
 By the step 10 of Algorithm \ref{alg:3}, we have
 \begin{align} \label{eq:A65}
  A^Tz^s_{t+1} = v^s_t + \frac{1}{\eta}G(x^s_{t+1}-x^s_t).
 \end{align}
 Since
 \begin{align}
 A^T(z^s_{t+1}-z^s_t) = v^s_t - v^s_{t-1} + \frac{G}{\eta}(x^s_{t+1}-x^s_t) - \frac{G}{\eta}(x^s_{t}-x^s_{t-1}),
 \end{align}
 then we have
 \begin{align} \label{eq:A67}
 \|z^s_{t+1}-z^s_t\|^2
 \leq \frac{1}{\sigma^A_{\min}}\big[3\|v^s_t - v^s_{t-1}\|^2 +\frac{3\sigma^2_{\max}(G)}{\eta^2}\|x^s_{t+1}-x^s_t\|^2
  + \frac{3\sigma^2_{\max}(G)}{\eta^2}\|x^s_{t}-x^s_{t-1}\|^2 \big] .
 \end{align}

 Next, considering the upper bound of $\|v^s_t - v^s_{t-1}\|^2$, we have
 \begin{align} \label{eq:A68}
  \|v^s_t - v^s_{t-1}\|^2 & = \|v^s_t - \nabla f(x^s_t) + \nabla f(x^s_t) -  \nabla f(x^s_{t-1}) + \nabla f(x^s_{t-1}) - v^s_{t-1}\|^2 \nonumber \\
  & \leq 3\|v^s_t - \nabla f(x^s_t)\|^2 + 3\|\nabla f(x^s_t) -  \nabla f(x^s_{t-1})\|^2 + 3\|\nabla f(x^s_{t-1}) - v^s_{t-1}\|^2 \nonumber \\
  & \leq \frac{3L^2}{b}\|x^s_t - \tilde{x}^s\|^2 + \frac{3L^2}{b}\|x^s_{t-1} - \tilde{x}^s\|^2 + 3L^2\|x^s_t - x^s_{t-1}\|^2
 \end{align}
 where the second inequality holds by Lemma 3 of \cite{Reddi2016Prox} and Assumption 1.
 Finally, combining \eqref{eq:A67} and \eqref{eq:A68}, we obtain the above result.
\end{proof}

\begin{lemma} \label{lem:A10}
 Suppose the sequence $\{(x^{s}_t,y_{[m]}^{s,t},z^{s}_t)_{t=1}^M\}_{s=1}^S$ is generated from Algorithm \ref{alg:3},
 and define a \emph{Lyapunov} function:
 \begin{align}
 \Gamma^s_t = \mathbb{E}\big[\mathcal{L}_{\rho} (x^s_t,y_{[m]}^{s,t},z^s_t) + (\frac{3\sigma^2_{\max}(G)}{\sigma^A_{\min}\eta^2\rho} + \frac{9L^2}{\sigma^A_{\min}\rho})\|x^s_{t}-x^s_{t-1}\|^2
 + \frac{9L^2 }{\sigma^A_{\min}\rho b}\|x^s_{t-1}-\tilde{x}^s\|^2 + c_t\|x^s_{t}-\tilde{x}^s\|^2\big],
 \end{align}
 where the positive sequence $\{c_t\}$ satisfies, for $s =1,2,\cdots,S$
 \begin{equation*}
  c_t= \left\{
  \begin{aligned}
  & \frac{18 L^2 }{\sigma^A_{\min}\rho b} +
     \frac{L}{b} + (1+\beta)c_{t+1}, \ 1 \leq t \leq M, \\
  & 0, \ t \geq M+1.
  \end{aligned}
  \right.\end{equation*}
Let $M=[n^{\frac{1}{3}}]$, $b=[n^{\frac{2}{3}}]$, $\eta = \frac{\alpha\sigma_{\min}(G)}{5L} \ (0< \alpha \leq 1)$ and
$\rho = \frac{2\sqrt{231}\kappa_G L}{\sigma^A_{\min}\alpha}$, we have
\begin{align}
\frac{1}{T}\sum_{s=1}^S \sum_{t=0}^{M-1} (\sigma_{\min}^H\sum_{j=1}^m \|y_j^{s,t}-y_j^{s,t+1}\|^2 + \frac{L}{2b}\|x^s_t-\tilde{x}^s\|^2_2 + \chi_t \|x^s_{t+1}-x^s_t\|^2) \leq \frac{\Gamma^1_0 - \Gamma^*}{T}  .
\end{align}
where $\Gamma^*$ denotes a low bound of $\Gamma^s_t$ and and $\chi_t \geq \frac{\sqrt{231}\kappa_G L}{2\alpha} > 0$.
\end{lemma}

\begin{proof}

By the optimal condition of step 7 in Algorithm \ref{alg:3},
we have, for $j\in [m]$
\begin{align}
0 & =(y_j^{s,t}-y_j^{s,t+1})^T\big(\partial g_j(y_j^{s,t+1}) - B^Tz_t^s + \rho B^T(Ax_t^s + \sum_{i=1}^jB_iy_i^{s,t+1} + \sum_{i=j+1}^mB_iy_i^{s,t}-c) + H_j(y_j^{s,t+1}-y_j^{s,t})\big) \nonumber \\
& \leq g_j(y_j^{s,t})- g_j(y_j^{s,t+1}) - (z_t^s)^T(B_jy_j^{s,t}-B_jy_j^{s,t+1}) + \rho(By_j^{s,t}-By_j^{s,t+1})^T(Ax_t^s + \sum_{i=1}^jB_iy_i^{s,t+1} + \sum_{i=j+1}^mB_iy_i^{s,t}-c) \nonumber \\
& \quad - \|y_j^{s,t+1}-y_j^{s,t}\|^2_{H_j} \nonumber \\
& = g_j(y_j^{s,t})- g_j(y_j^{s,t+1}) - (z_t^s)^T(Ax_t^s+\sum_{i=1}^{j-1}B_iy_i^{s,t+1} + \sum_{i=j}^mB_iy_i^{s,t}-c) + (z_t^s)^T(Ax_t^s+\sum_{i=1}^jB_iy_i^{s,t+1}+ \sum_{i=j+1}^mB_iy_i^{s,t}-c) \nonumber \\
& \quad  + \frac{\rho}{2}\|Ax_t^s +\sum_{i=1}^{j-1}B_iy_i^{s,t+1} + \sum_{i=j}^mB_iy_i^{s,t}-c\|^2 - \frac{\rho}{2}\|Ax_t^s+\sum_{i=1}^jB_iy_i^{s,t+1}+ \sum_{i=j+1}^mB_iy_i^{s,t}-c\|^2
  - \|y_j^{s,t+1}-y_j^{s,t}\|^2_{H_j} \nonumber \\
& \quad  -\frac{\rho}{2}\|B_jy_j^{s,t}-B_jy_j^{s,t+1}\|^2  \nonumber \\
& \leq \underbrace{ f(x_t^s) + g_j(y_j^{s,t}) - (z_t^s)^T(Ax_t^s+\sum_{i=1}^{j-1}B_iy_i^{s,t+1} + \sum_{i=j}^mB_iy_i^{s,t}-c) + \frac{\rho}{2}\|Ax_t^s +\sum_{i=1}^{j-1}B_iy_i^{s,t+1} + \sum_{i=j}^mB_iy_i^{s,t}-c\|^2}_{\mathcal{L}_{\rho} (x_t^s,y_{[j-1]}^{s,t+1},y_{[j:m]}^{s,t},z_t^s)}  - \|y_j^{s,t+1}-y_j^{s,t}\|^2_{H_j}  \nonumber \\
& \quad -\underbrace{(f(x_t^s) + g_j(y_j^{s,t+1}) - (z_t^s)^T(Ax_t^s+\sum_{i=1}^jB_iy_i^{s,t+1}+ \sum_{i=j+1}^mB_iy_i^{s,t}-c) + \frac{\rho}{2}\|Ax_t^s+\sum_{i=1}^jB_iy_i^{s,t+1}+ \sum_{i=j+1}^mB_iy_i^{s,t}-c\|^2}_{\mathcal{L}_{\rho} (x_t^s,y_{[j]}^{s,t+1},y_{[j+1:m]}^{s,t},z_t^s)} \nonumber \\
& \leq \mathcal{L}_{\rho} (x_t^s,y_{[j-1]}^{s,t+1},y_{[j:m]}^{s,t},z_t^s) - \mathcal{L}_{\rho} (x_t^s,y_{[j]}^{s,t+1},y_{[j+1:m]}^{s,t},z_t^s)
- \sigma_{\min}(H_j)\|y_j^{s,t}-y_j^{s,t+1}\|^2,
\end{align}
where the first inequality holds by the convexity of function $g_j(y)$,
and the second equality follows by applying the equality
$(a-b)^Tb = \frac{1}{2}(\|a\|^2-\|b\|^2-\|a-b\|^2)$ on the term $(By_j^{s,t}-By_j^{s,t+1})^T(Ax_t^s + \sum_{i=1}^jB_iy_i^{s,t+1} + \sum_{i=j+1}^mB_iy_i^{s,t}-c)$.
Thus, we have, for all $j\in[m]$
\begin{align} \label{eq:A10-1}
 \mathcal{L}_{\rho} (x_t^s,y_{[j-1]}^{s,t+1},y_{[j:m]}^{s,t},z_t^s) \leq \mathcal{L}_{\rho} (x_t^s,y_{[j]}^{s,t+1},y_{[j+1:m]}^{s,t},z_t^s)
 - \sigma_{\min}(H_j)\|y_j^{s,t}-y_j^{s,t+1}\|^2.
\end{align}
Telescoping inequality \eqref{eq:A10-1} over $j$ from $1$ to $m$, we obtain
\begin{align} \label{eq:A73}
 \mathcal{L}_{\rho} (x_t^s,y^{s,t+1}_{[m]},z_t^s) \leq \mathcal{L}_{\rho} (x_t^s,y^{s,t}_{[m]},z_t^s)
 - \sigma_{\min}^H\sum_{j=1}^m \|y_j^{s,t}-y_j^{s,t+1}\|^2,
\end{align}
where $\sigma_{\min}^H=\min_{j\in[m]}\sigma_{\min}(H_j)$.

By Assumption 1, we have
\begin{align} \label{eq:A74}
0 \leq f(x^s_t) - f(x^s_{t+1}) + \nabla f(x^s_t)^T(x^s_{t+1}-x^s_t) + \frac{L}{2}\|x^s_{t+1}-x^s_t\|^2.
\end{align}
Using optimal condition of the step 8 in Algorithm \ref{alg:3},
we have
\begin{align} \label{eq:A75}
 0 = (x^s_t-x^s_{t+1})^T \big( v^s_t - A^Tz^s_t + \rho A^T(Ax^s_{t+1} + \sum_{j=1}^mB_jy_j^{s,t+1}-c) + \frac{G}{\eta}(x^s_{t+1}-x^s_t) \big).
\end{align}
Combining \eqref{eq:A74} and \eqref{eq:A75}, we have
\begin{align}
 0 & \leq f(x^s_t) - f(x^s_{t+1}) + \nabla f(x^s_t)^T(x^s_{t+1}-x^s_t) + \frac{L}{2}\|x^s_{t+1}-x^s_t\|^2 \nonumber \\
 & \quad + (x^s_t-x^s_{t+1})^T \big( v^s_t - A^Tz^s_t + \rho A^T(Ax^s_{t+1} + \sum_{j=1}^mB_jy_j^{s,t+1}-c) + \frac{G}{\eta}(x^s_{t+1}-x^s_t) \big)  \nonumber \\
 & = f(x^s_t) - f(x^s_{t+1}) + \frac{L}{2}\|x^s_t-x^s_{t+1}\|^2 - \frac{1}{\eta}\|x^s_t - x^s_{t+1}\|^2_G + (x^s_t-x^s_{t+1})^T(v^s_t-\nabla f(x^s_t)) \nonumber \\
 & \quad -(z^s_t)^T(Ax^s_t-Ax^s_{t+1}) + \rho(Ax^s_t - Ax^s_{t+1})^T(Ax^s_{t+1} + \sum_{j=1}^mB_jy_j^{s,t+1}-c) \nonumber \\
 & \mathop{=}^{(i)} f(x^s_t) - f(x^s_{t+1}) + \frac{L}{2}\|x^s_t-x^s_{t+1}\|^2 - \frac{1}{\eta}\|x^s_t - x^s_{t+1}\|^2_G
 + (x^s_t-x^s_{t+1})^T(v^s_t-\nabla f(x^s_t)) -(z^s_t)^T(Ax^s_t + \sum_{j=1}^mB_jy_j^{s,t+1}-c)\nonumber \\
 & \quad  + (z^s_t)^T(Ax^s_{t+1}+ \sum_{j=1}^mB_jy_j^{s,t+1}-c) + \frac{\rho}{2}\big(\|Ax^s_{t} + \sum_{j=1}^mB_jy_j^{s,t+1}-c\|^2
 - \|Ax^s_{t+1} + \sum_{j=1}^mB_jy_j^{s,t+1}-c\|^2 - \|Ax^s_t - Ax^s_{t+1}\|^2 \big) \nonumber \\
 & = \underbrace{f(x^s_t) -(z^s_t)^T(Ax^s_t + \sum_{j=1}^mB_jy_j^{s,t+1}-c) + \frac{\rho}{2}\|Ax^s_{t} + \sum_{j=1}^mB_jy_j^{s,t+1}-c\|^2}_{\mathcal{L}_{\rho} (x^s_t,y_{[m]}^{s,t+1},z^s_t)}
 + \frac{L}{2}\|x^s_t-x^s_{t+1}\|^2 + (x^s_t-x^s_{t+1})^T(v^s_t-\nabla f(x^s_t)) \nonumber \\
 & \quad - \underbrace{f(x^s_{t+1}) -(z^s_t)^T(Ax^s_{t+1} + \sum_{j=1}^mB_jy_j^{s,t+1}-c) + \frac{\rho}{2}\|Ax^s_{t+1} + \sum_{j=1}^mB_jy_j^{s,t+1}-c\|^2}_{\mathcal{L}_{\rho} (x^s_{t+1},y_{[m]}^{s,t+1},z^s_t)} -\frac{1}{\eta}\|x^s_t - x^s_{t+1}\|^2_G - \frac{\rho}{2}\|Ax^s_t - Ax^s_{t+1}\|^2 \nonumber \\
 & \leq \mathcal{L}_{\rho} (x^s_t,y_{[m]}^{s,t+1},z^s_t) -  \mathcal{L}_{\rho} (x^s_{t+1},y_{[m]}^{s,t+1},z^s_t)
 - (\frac{\sigma_{\min}(G)}{\eta}+\frac{\rho \sigma^A_{\min}}{2}-\frac{L}{2})\|x^s_t - x^s_{t+1}\|^2 +(x^s_t-x^s_{t+1})^T(v^s_t-\nabla f(x^s_t)) \nonumber \\
 & \mathop{\leq}^{(ii)}  \mathcal{L}_{\rho} (x^s_t,y_{[m]}^{s,t+1},z^s_t) -  \mathcal{L}_{\rho} (x^s_{t+1},y_{[m]}^{s,t+1},z^s_t)
 - (\frac{\sigma_{\min}(G)}{\eta}+\frac{\rho \sigma^A_{\min}}{2}-L)\|x^s_t - x^s_{t+1}\|^2 + \frac{1}{2L}\|v^s_t-\nabla f(x^s_t)\|^2 \nonumber \\
 & \mathop{\leq}^{(iii)} \mathcal{L}_{\rho} (x^s_t,y_{[m]}^{s,t+1},z^s_t) -  \mathcal{L}_{\rho} (x^s_{t+1},y_{[m]}^{s,t+1},z^s_t)
 - (\frac{\sigma_{\min}(G)}{\eta}+\frac{\rho \sigma^A_{\min}}{2}-L)\|x^s_t - x^s_{t+1}\|^2 +\frac{ L}{2b}\|x^s_t-\tilde{x}^s\|^2 ,
\end{align}
where the equality $(i)$ holds by applying the equality
$(a-b)^Tb = \frac{1}{2}(\|a\|^2-\|b\|^2-\|a-b\|^2)$ on the term $(Ax^s_t - Ax^s_{t+1})^T(Ax^s_{t+1}+\sum_{j=1}^mB_jy_j^{s,t+1}-c)$, the inequality
$(ii)$ holds by the inequality $a^Tb \leq \frac{L}{2}\|a\|^2 + \frac{1}{2L}\|b\|^2$,
and the inequality $(iii)$ holds by Lemma 3 of \cite{Reddi2016Prox}. Thus, we obtain
\begin{align} \label{eq:A77}
\mathcal{L}_{\rho} (x^s_{t+1},y_{[m]}^{s,t+1},z^s_t) \leq & \mathcal{L}_{\rho} (x^s_t,y_{[m]}^{s,t+1},z^s_t) -
(\frac{\sigma_{\min}(G)}{\eta}+\frac{\rho \sigma^A_{\min}}{2}-L)\|x^s_t - x^s_{t+1}\|^2 +\frac{ L}{2b}\|x^s_t-\tilde{x}^s\|^2 .
\end{align}

By the step 9 in Algorithm \ref{alg:3}, we have
\begin{align} \label{eq:A78}
\mathcal{L}_{\rho} (x^s_{t+1},y_{[m]}^{s,t+1},z^s_{t+1}) -
\mathcal{L}_{\rho} (x^s_{t+1},y_{[m]}^{s,t+1},z^s_t)
& = \frac{1}{\rho}\|z^s_{t+1}-z^s_t\|^2 \nonumber \\
& \leq  \frac{9 L^2 }{\sigma^A_{\min} b\rho} \big( \|x^s_t - \tilde{x}^s\|^2 + \|x^s_{t-1} - \tilde{x}^s\|^2\big)
 + \frac{3\sigma^2_{\max}(G)}{\sigma^A_{\min}\eta^2\rho}\|x^s_{t+1}-x^s_t\|^2 \nonumber \\
& \quad + (\frac{3\sigma^2_{\max}(G)}{\sigma^A_{\min}\eta^2\rho}+\frac{9L^2}{\sigma^A_{\min}\rho})\|x^s_{t}-x^s_{t-1}\|^2,
\end{align}
where the first inequality follows by Lemma \ref{lem:A9}.

Combining \eqref{eq:A73}, \eqref{eq:A77} and \eqref{eq:A78}, we have
\begin{align}
\mathcal{L}_{\rho} (x^s_{t+1},y_{[m]}^{s,t+1},z^s_{t+1}) & \leq \mathcal{L}_{\rho} (x^s_t,y_{[m]}^{s,t},z^s_t)
- \sigma_{\min}^H\sum_{j=1}^m \|y_j^{s,t}-y_j^{s,t+1}\|^2 - (\frac{\sigma_{\min}(G)}{\eta}+\frac{\rho \sigma^A_{\min}}{2}-L)\|x^s_t - x^s_{t+1}\|^2 \nonumber \\
& \quad  +\frac{ L}{2b}\|x^s_t-\tilde{x}^s\|^2 + \frac{9 L^2 }{\sigma^A_{\min} b\rho}\big( \|x^s_t - \tilde{x}^s\|^2 + \|x^s_{t-1} - \tilde{x}^s\|^2\big)
+ \frac{3\sigma^2_{\max}(G)}{\sigma^A_{\min}\eta^2\rho}\|x^s_{t+1}-x^s_t\|^2 \nonumber \\
& \quad + (\frac{3\sigma^2_{\max}(G)}{\sigma^A_{\min}\eta^2\rho}+\frac{9L^2}{\sigma^A_{\min}\rho})\|x^s_{t}-x^s_{t-1}\|^2.
\end{align}

Next, we define a \emph{Lyapunov} function $\Gamma^s_t$ as follows:
\begin{align} \label{eq:69}
 \Gamma^s_t = \mathbb{E}\big[\mathcal{L}_{\rho} (x^s_t,y_{[m]}^{s,t},z^s_t) + (\frac{3\sigma^2_{\max}(G)}{\sigma^A_{\min}\eta^2\rho} + \frac{9L^2}{\sigma^A_{\min}\rho})\|x^s_{t}-x^s_{t-1}\|^2
 + \frac{9 L^2 }{\sigma^A_{\min}\rho b}\|x^s_{t-1}-\tilde{x}^s\|^2 + c_t\|x^s_{t}-\tilde{x}^s\|^2\big].
\end{align}
Considering the upper bound of $\|x^s_{t+1}-\tilde{x}^s\|^2$, we have
\begin{align}  \label{eq:70}
\|x^s_{t+1}-x^s_t + x^s_t - \tilde{x}^s\|^2
& = \|x^s_{t+1}-x^s_t\|^2 + 2(x^s_{t+1}-x^s_t)^T(x^s_t-\tilde{x}^s) +\|x^s_t -\tilde{x}^s\|^2 \nonumber \\
& \leq \|x^s_{t+1}-x^s_t\|^2 + 2\big(\frac{1}{2\beta} \|x^s_{t+1}-x^s_t\|^2 + \frac{\beta}{2}\|x^s_t-\tilde{x}^s\|^2\big)
+ \|x^s_t -\tilde{x}^s\|^2 \nonumber \\
& = (1+1/\beta)\|x^s_{t+1}-x^s_t\|^2 +(1+\beta)\|x^s_t -\tilde{x}^s\|^2,
\end{align}
where the above inequality holds by by the
Cauchy-Schwarz inequality with $\beta>0$. Combining \eqref{eq:69} with \eqref{eq:70}, then we obtain
\begin{align} \label{eq:A82}
\Gamma^s_{t+1} & = \mathbb{E}\big[\mathcal{L}_{\rho}(x^s_{t+1},y_{[m]}^{s,t+1},z^s_{t+1}) + (\frac{3\sigma^2_{\max}(G)}{\sigma^A_{\min}\eta^2\rho}+\frac{9L^2}{\sigma^A_{\min}\rho}) \|x^s_{t+1}-x^s_{t}\|^2 + \frac{9 L^2 }{\sigma^A_{\min} b\rho} \|x^s_{t}-\tilde{x}^s\|^2 + c_{t+1}\|x^s_{t+1}-\tilde{x}^s\|^2\big] \nonumber \\
& \leq \mathcal{L}_{\rho} (x^s_t,y_{[m]}^{s,t},z^s_t) + (\frac{3\sigma^2_{\max}(G)}{\sigma^A_{\min}\eta^2\rho}+\frac{9L^2}{\sigma^A_{\min}\rho})\|x^s_{t}-x^s_{t-1}\|^2
+ \frac{9 L^2}{\sigma^A_{\min}\rho b} \|x^s_{t-1}-\tilde{x}^s\|^2 + \big(\frac{18 L^2 }{\sigma^A_{\min}\rho b} + \frac{L}{b}  +(1+\beta)c_{t+1}\big)\|x^s_t-\tilde{x}^s\|^2 \nonumber\\
& \quad - \big( \frac{\sigma_{\min}(G)}{\eta} + \frac{\rho\sigma^A_{\min}}{2} - L - \frac{6\sigma^2_{\max}(G)}{\sigma^A_{\min}\eta^2\rho}-\frac{9L^2}{\sigma^A_{\min}\rho}
- (1+1/\beta)c_{t+1} \big)\|x^s_t - x^s_{t+1}\|^2 \nonumber \\
& \quad - \sigma_{\min}^H\sum_{j=1}^m \|y_j^{s,t}-y_j^{s,t+1}\|^2 - \frac{L }{2b}\|x^s_{t}-\tilde{x}^s\|^2   \nonumber \\
& \leq \Gamma^s_t - \chi_t \|x^s_t - x^s_{t+1}\|^2 - \sigma_{\min}^H\sum_{j=1}^m \|y_j^{s,t}-y_j^{s,t+1}\|^2 - \frac{L}{2b}\|x^s_{t}-\tilde{x}^s\|^2,
\end{align}
where $c_t = \frac{18 L^2 }{\sigma^A_{\min}\rho b} + \frac{L}{b}  + (1+\beta)c_{t+1}$ and $\chi_t = \frac{\sigma_{\min}(G)}{\eta}+\frac{\rho
\sigma^A_{\min}}{2} - L - \frac{6\sigma^2_{\max}(G)}{\sigma^A_{\min}\eta^2\rho}-\frac{9L^2}{\sigma^A_{\min}\rho} - (1+1/\beta)c_{t+1}$.

Next, we will prove the relationship between $\Gamma^{s+1}_1$ and
$\Gamma^s_M$. Since $x^{s+1}_0 = x^s_M = \tilde{x}^{s+1}$, we have
\begin{align}
 v^{s+1}_0 = \nabla f_{\mathcal{I}}(x^{s+1}_0) - \nabla f_{\mathcal{I}}(x^{s+1}_0) + \nabla f(x^{s+1}_0) = \nabla f(x^{s+1}_0) = \nabla f(x^{s}_M).
\end{align}
Thus, we obtain
\begin{align}
 \mathbb{E} \|v^{s+1}_0 - v^{s}_M\|^2 & = \mathbb{E}\|\nabla f(x^{s}_M) - \nabla f_{\mathcal{I}}(x^{s}_M)
 + \nabla f_{\mathcal{I}}(\tilde{x}^{s}) - \nabla f(\tilde{x}^{s})\|^2 \nonumber \\
 & = \| \nabla f_{\mathcal{I}}(x^{s}_M)
 - \nabla f_{\mathcal{I}}(\tilde{x}^{s})-\mathbb{E}_{\mathcal{I}} [\nabla f_{\mathcal{I}}(x^{s}_M)
 - \nabla f_{\mathcal{I}}(\tilde{x}^{s})]\|^2 \nonumber \\
 & \leq \frac{1}{bn} \sum_{i=1}^n \mathbb{E}\| \nabla f_{i}(x^{s}_M) - \nabla f_{i}(\tilde{x}^{s})\|^2 \nonumber \\
 & \leq \frac{L^2}{b}\|x^s_M-\tilde{x}^s\|^2.
\end{align}
By the step 9 of Algorithm \ref{alg:3}, we have
\begin{align} \label{eq:A85}
 \|z^{s+1}_1-z^s_M\|^2 & \leq
 \frac{1}{\sigma^A_{\min}}\|v^{s+1}_0 - v^{s}_M + \frac{G}{\eta}(x^{s+1}_1-x^{s+1}_0) + \frac{G}{\eta}(x^s_M - x^s_{M-1}) \|^2 \nonumber \\
 &= \frac{1}{\sigma^A_{\min}}\|\nabla f(x^{s}_M) - v^{s}_M + \frac{G}{\eta}(x^{s+1}_1-x^{s}_M) + \frac{G}{\eta}(x^s_M - x^s_{M-1})\|^2 \nonumber \\
 & \leq \frac{1}{\sigma^A_{\min}} \big( 3\|\nabla f(x^{s}_M) - v^{s}_M\|^2 +
 \frac{3\sigma^2_{\max}(G)}{\eta^2}\|x^{s+1}_1 - x^s_M\|^2+ \frac{3\sigma^2_{\max}(G)}{\eta^2}\|x^s_M-x^{s}_{M-1}\|^2 \big)\nonumber \\
 & \leq \frac{1}{\sigma^A_{\min}} \big( 3\|\nabla f(x^{s}_M) - v^{s}_M\|^2 + \frac{3\sigma^2_{\max}(G)}{\eta^2}\|x^{s+1}_1 - x^s_M\|^2
 + \frac{3\sigma^2_{\max}(G)}{\eta^2}\|x^s_M-x^{s}_{M-1}\|^2 \big)\nonumber \\
 & \leq \frac{1}{\sigma^A_{\min}} \big( \frac{3L^2}{b}\|x^s_M-\tilde{x}^s\|^2_2 + \frac{3\sigma^2_{\max}(G)}{\eta^2}\|x^{s+1}_1 - x^s_M\|^2
 + \frac{3\sigma^2_{\max}(G)}{\eta^2}\|x^s_M-x^{s}_{M-1}\|^2 \big).
\end{align}

Since $x^{s}_M=x^{s+1}_0$, $y_j^{s,M}=y_j^{s+1,0}$ for all $j\in [m]$ and $z^{s}_M=z^{s+1}_0$,  by \eqref{eq:A73}, we have
\begin{align} \label{eq:A86}
\mathcal{L}_{\rho} (x^{s+1}_0,y_{[m]}^{s+1,1},z^{s+1}_0) \leq \mathcal{L}_{\rho} (x^{s}_M,y^{s,M}_{[m]},z^{s}_M) - \sigma_{\min}^H\sum_{j=1}^m \|y_j^{s,M}-y_j^{s+1,1}\|^2.
\end{align}
By \eqref{eq:A77}, we have
\begin{align} \label{eq:A87}
\mathcal{L}_{\rho} (x^{s+1}_{1},y^{s+1,1}_{[m]},z^{s+1}_0) \leq
\mathcal{L}_{\rho} (x^{s+1}_0,y^{s+1,1}_{[m]},z^{s+1}_0) - (\frac{\sigma_{\min}(G)}{\eta}+\frac{\rho \sigma^A_{\min}}{2}-L)\|x^{s+1}_0 - x^{s+1}_{1}\|^2.
\end{align}
By \eqref{eq:A78}, we have
\begin{align} \label{eq:A88}
\mathcal{L}_{\rho} (x^{s+1}_{1},y^{s+1,1}_{[m]},z^{s+1}_1) & \leq
\mathcal{L}_{\rho} (x^{s+1}_{1},y^{s+1,1}_{[m]},z^{s+1}_0) + \frac{1}{\rho}\|z^{s+1}_1-z^{s+1}_0\|^2 \nonumber \\
& \leq \mathcal{L}_{\rho} (x^{s+1}_{1},y^{s+1,1}_{[m]},z^{s+1}_0) + \frac{1}{\sigma^A_{\min}\rho} \big( \frac{3L^2}{b}\|x^s_M-\tilde{x}^s\|^2_2 + \frac{3\sigma^2_{\max}(G)}{\eta^2}\|x^{s+1}_1 - x^s_M\|^2 \nonumber \\
& \quad + \frac{3\sigma^2_{\max}(G)}{\eta^2}\|x^s_M-x^{s}_{M-1}\|^2 \big).
\end{align}
where the second inequality holds by \eqref{eq:A85}.

Combining \eqref{eq:A86}, \eqref{eq:A87} with \eqref{eq:A88}, we have
\begin{align}
\mathcal{L}_{\rho} (x^{s+1}_{1},y^{s+1,1}_{[m]},z^{s+1}_1) & \leq
\mathcal{L}_{\rho} (x^{s}_M,y^{s,M}_{[m]},z^{s}_M) - \sigma_{\min}^H\sum_{j=1}^m \|y_j^{s,M}-y_j^{s+1,1}\|^2 -(\frac{\sigma_{\min}(G)}{\eta}+\frac{\rho \sigma^A_{\min}}{2}-L)\|x^{s+1}_0 - x^{s+1}_{1}\|^2 + \nonumber \\
& \quad \frac{1}{\sigma^A_{\min}\rho} \big( \frac{3L^2d}{b}\|x^s_M-\tilde{x}^s\|^2_2 +\frac{3\sigma^2_{\max}(G)}{\eta^2}\|x^{s+1}_1 - x^s_M\|^2 + \frac{3\sigma^2_{\max}(G)}{\eta^2}\|x^s_M-x^{s}_{M-1}\|^2 \big).
\end{align}

Therefore, we have
\begin{align} \label{eq:A89}
\Gamma^{s+1}_1
& = \mathbb{E}\big[\mathcal{L}_{\rho}(x^{s+1}_{1},y_{[m]}^{s+1,1},z^{s+1}_1)
+ (\frac{3\sigma^2_{\max}(G)}{\sigma^A_{\min}\eta^2\rho} + \frac{9L^2}{\sigma^A_{\min}\rho})\|x^{s+1}_{1}-x^{s+1}_{0}\|^2
+ \frac{9 L^2 }{\sigma^A_{\min} b\rho}\|x^{s+1}_{0}-\tilde{x}^{s+1}\|^2 + c_1\|x^{s+1}_{1}-\tilde{x}^{s+1}\|^2\big] \nonumber \\
& = \mathcal{L}_{\rho} (x^{s+1}_{1},y_{[m]}^{s+1,1},z^{s+1}_1) +
\big(\frac{3\sigma^2_{\max}(G)}{\sigma^A_{\min}\eta^2\rho} + \frac{9L^2}{\sigma^A_{\min}\rho} + c_1 \big) \|x^{s+1}_{1}-x^{s+1}_{0}\|^2 \nonumber \\
& \leq \mathcal{L}_{\rho} (x^{s}_M,y_{[m]}^{s,M},z^{s}_M) +
(\frac{3\sigma^2_{\max}(G)}{\sigma^A_{\min}\eta^2\rho} + \frac{9L^2}{\sigma^A_{\min}\rho})\|x^{s}_{M}-x^{s}_{M-1}\|^2
+ \frac{9L^2}{\sigma^A_{\min}\rho b}\|x^s_{M-1}-\tilde{x}^s\|^2_2 + (\frac{18L^2 }{\sigma^A_{\min}\rho b} + \frac{L}{b}) \|x^s_M-\tilde{x}^s\|^2_2  \nonumber \\
& \quad - \sigma_{\min}^H\sum_{j=1}^m \|y_j^{s,M}-y_j^{s+1,1}\|^2 - \big(\frac{\sigma_{\min}(G)}{\eta}+ \frac{\rho \sigma^A_{\min}}{2} - L -
\frac{6\sigma^2_{\max}(G)}{\sigma^A_{\min}\eta^2\rho} - \frac{9L^2}{\sigma^A_{\min}\rho}- c_1 \big)\|x^{s+1}_1-x^s_M\|^2_2 \nonumber \\
& \quad - \frac{9L^2}{\sigma^A_{\min}\rho} \|x^s_M - x^s_{M-1}\|^2_2 - \frac{9 L^2}{\sigma^A_{\min}\rho b}\|x^s_{M-1}-\tilde{x}^s\|^2_2 - (\frac{15 L^2 }{\sigma^A_{\min}\rho b}
+ \frac{L}{b})\|x^s_M-\tilde{x}^s\|^2_2 \nonumber \\
& \leq \Gamma^s_M - \sigma_{\min}^H\sum_{j=1}^m \|y_j^{s,M}-y_j^{s+1,1}\|^2 - \frac{L}{2b}\|x^s_M-\tilde{x}^s\|^2_2 - \big(\frac{\sigma_{\min}(G)}{\eta}+ \frac{\rho \sigma^A_{\min}}{2} - L -
\frac{6\sigma^2_{\max}(G)}{\sigma^A_{\min}\eta^2\rho} - \frac{9L^2}{\sigma^A_{\min}\rho}- c_1 \big)\|x^{s+1}_1-x^s_M\|^2 \nonumber \\
& = \Gamma^s_M - \sigma_{\min}^H\sum_{j=1}^m \|y_j^{s,M}-y_j^{s+1,1}\|^2 - \frac{L}{2b}\|x^s_M-\tilde{x}^s\|^2_2 - \chi_M \|x^{s+1}_1-x^s_M\|^2,
\end{align}
where $c_M = \frac{18L^2}{\sigma^A_{\min}\rho b} + \frac{L}{b}$, and $\chi_M = \frac{\sigma_{\min}(G)}{\eta}+ \frac{\rho \sigma^A_{\min}}{2} - L -
\frac{6\sigma^2_{\max}(G)}{\sigma^A_{\min}\eta^2\rho} - \frac{9L^2}{\sigma^A_{\min}\rho}- c_1$.

Let $c_{M+1} = 0$ and $\beta=\frac{1}{M}$, recursing on $t$, we have
\begin{align}
 c_{t+1} = (\frac{18 L^2}{\sigma^A_{\min}\rho b} + \frac{L}{ b})\frac{(1+\beta)^{M-t}-1}{\beta}
 & = \frac{M}{b}(\frac{18L^2}{\sigma^A_{\min}\rho} + L ) \big((1+\frac{1}{M})^{M-t}-1\big) \nonumber \\
 & \leq \frac{M}{b}(\frac{18 L^2}{\sigma^A_{\min}\rho} + L )(e-1) \leq  \frac{2M}{b}(\frac{18 L^2}{\sigma^A_{\min}\rho} + L ).
\end{align}
where the first inequality holds by $(1+\frac{1}{M})^M$ is an increasing function and $\lim_{M\rightarrow \infty}(1+\frac{1}{M})^M=e$.
It follows that, for $t=1,2,\cdots,M$
\begin{align}
 \chi_t & \geq \frac{\sigma_{\min}(G)}{\eta}+\frac{\rho
\sigma^A_{\min}}{2} - L -\frac{6\sigma^2_{\max}(G)}{\sigma^A_{\min}\eta^2\rho}-\frac{9L^2}{\sigma^A_{\min}\rho} - (1+1/\beta)\frac{2M}{b}(\frac{18 L^2}{\sigma^A_{\min}\rho} + L ) \nonumber \\
& = \frac{\sigma_{\min}(G)}{\eta}+\frac{\rho
\sigma^A_{\min}}{2} - L -\frac{6\sigma^2_{\max}(G)}{\sigma^A_{\min}\eta^2\rho}-\frac{9L^2}{\sigma^A_{\min}\rho} - (1+M)\frac{2M}{b}(\frac{18 L^2}{\sigma^A_{\min}\rho} + L ) \nonumber \\
& \geq \frac{\sigma_{\min}(G)}{\eta}+\frac{\rho
\sigma^A_{\min}}{2} - L -\frac{6\sigma^2_{\max}(G)}{\sigma^A_{\min}\eta^2\rho}-\frac{9L^2}{\sigma^A_{\min}\rho} - \frac{4M^2}{b}(\frac{18 L^2}{\sigma^A_{\min}\rho} + L ) \nonumber \\
& = \underbrace{\frac{\sigma_{\min}(G)}{\eta} - L - \frac{4M^2L}{b}}_{Q_1} + \underbrace{\frac{\rho\sigma^A_{\min}}{2} - \frac{6\sigma^2_{\max}(G)}{\sigma^A_{\min}\eta^2\rho}-\frac{9L^2}{\sigma^A_{\min}\rho}
- \frac{72M^2L^2}{b\sigma^A_{\min}\rho}}_{Q_2}.
\end{align}

Let $M = [n^{\frac{1}{3}}]$, $b=[n^{\frac{2}{3}}]$ and $0< \eta \leq \frac{\sigma_{\min}(G)}{5L}$, we have $Q_1 \geq 0$. Further, set $\eta = \frac{\alpha\sigma_{\min}(G)}{5L} \ (0< \alpha \leq 1)$ and
$\rho = \frac{2\sqrt{231}\kappa_G L}{\sigma^A_{\min}\alpha}$, we have
\begin{align}
Q_2 & = \frac{\rho\sigma^A_{\min}}{2} - \frac{6\sigma^2_{\max}(G)}{\sigma^A_{\min}\eta^2\rho}-\frac{9L^2}{\sigma^A_{\min}\rho} - \frac{72M^2L^2}{b\sigma^A_{\min}\rho} \nonumber \\
& = \frac{\rho\sigma^A_{\min}}{2} - \frac{150\kappa^2_G L^2}{\sigma^A_{\min}\rho\alpha^2}-\frac{9L^2}{\sigma^A_{\min}\rho} - \frac{72L^2}{\sigma^A_{\min}\rho} \nonumber \\
& \geq \frac{\rho\sigma^A_{\min}}{2} - \frac{150\kappa^2_G L^2}{\sigma^A_{\min}\rho\alpha^2}-\frac{9\kappa^2_GL^2}{\sigma^A_{\min}\rho\alpha^2} - \frac{72\kappa^2_G L^2}{\sigma^A_{\min}\rho\alpha^2} \nonumber \\
& = \frac{\rho\sigma^A_{\min}}{4} + \underbrace{ \frac{\rho\sigma^A_{\min}}{4} - \frac{231\kappa^2_G L^2}{\sigma^A_{\min}\rho\alpha^2} }_{\geq 0} \nonumber \\
& \geq \frac{\sqrt{231}\kappa_G L}{2\alpha} > 0 \nonumber
\end{align}
where $\kappa_G = \frac{\sigma_{\max}(G)}{\sigma_{\min}(G)} \geq 1$. Thus, we have $\chi_t \geq \frac{\sqrt{231}\kappa_G L}{2\alpha} > 0$ for all $t$.

Since $\frac{L}{2b} > 0$ and $\chi_t > 0$, by \eqref{eq:A82} and \eqref{eq:A89}, the function $\Gamma^s_t$ is monotone decreasing.
Using \eqref{eq:69}, we have
\begin{align} \label{eq:A106}
 \Gamma^s_t &\geq   \mathbb{E}\big[\mathcal{L}_{\rho} (x^s_t,y_{[m]}^{s,t},z^s_t) ] \nonumber \\
 & = f(x^s_t) + \sum_{j=1}^mg(y_j^{s,t}) - (z^s_t)^T(Ax^s_t + \sum_{j=1}^mB_jy_j^{s,t} - c) + \frac{\rho}{2}\|Ax^s_t + \sum_{j=1}^mB_jy_j^{s,t} -c\| \nonumber \\
 & = f(x^s_t) + \sum_{j=1}^mg(y_j^{s,t}) - \frac{1}{\rho}(z^s_t)^T(z^s_{t-1} - z^s_t) + \frac{1}{2\rho}\|z^s_t - z^s_{t-1}\|^2 \nonumber \\
 & = f(x^s_t) + \sum_{j=1}^mg(y_j^{s,t}) - \frac{1}{2\rho}\|z^s_{t-1}\|^2 + \frac{1}{2\rho}\|z^s_t\|^2 + \frac{1}{\rho}\|z^s_t - z^s_{t-1}\|^2 \nonumber \\
 & \geq f^* + \sum_{j=1}^mg_j^* - \frac{1}{2\rho}\|z^s_{t-1}\|^2 + \frac{1}{2\rho}\|z^s_t\|^2.
\end{align}
Summing the inequality \eqref{eq:A106} over $t=0,1\cdots,M$ and $s=1,2,\cdots,S$, we have
\begin{align}
 \frac{1}{T} \sum_{s=1}^S\sum_{t=0}^M  \Gamma^s_t \geq f^* + \sum_{j=1}^mg_j^*  - \frac{1}{2\rho}\|z^1_{0}\|^2.
\end{align}
Thus, the function $\Gamma^s_t$ is bounded from below. Set $\Gamma^*$ denotes a low bound of $\Gamma^s_t$.

Finally, telescoping \eqref{eq:A82} and \eqref{eq:A89} over $t$ from $0$ to $M-1$
and over $s$ from $1$ to $S$, we have
\begin{align} \label{eq:A93}
\frac{1}{T}\sum_{s=1}^S \sum_{t=0}^{M-1} ( \sigma_{\min}^H\sum_{j=1}^m \|y_j^{s,t}-y_j^{s,t+1}\|^2 + \frac{L}{2b}\|x^s_t-\tilde{x}^s\|^2_2 + \chi_t\|x^s_t-x^s_{t+1}\|^2)
\leq \frac{\Gamma^1_0 - \Gamma^*}{T}.
\end{align}
where $T=MS$ and $\chi_t \geq \frac{\sqrt{231}\kappa_G L}{2\alpha} > 0$.

\end{proof}

\begin{theorem}
 Suppose the sequence $\{(x^{s}_t,y_{[m]}^{s,t},z^{s}_t)_{t=1}^M\}_{s=1}^S$ is generated from Algorithm \ref{alg:3}, and let $\eta = \frac{\alpha\sigma_{\min}(G)}{5L} \ (0 < \alpha \leq 1)$,
 $\rho = \frac{2\sqrt{231}\kappa_G L}{\sigma^A_{\min}\alpha}$ and
 \begin{align}
 \nu_1 = m\big(\rho^2\sigma^B_{\max}\sigma^A_{\max} + \rho^2(\sigma^B_{\max})^2 + \sigma^2_{\max}(H)\big), \ \nu_2 = 3L^2 + \frac{3\sigma^2_{\max}(G)}{\eta^2},
 \ \nu_3 = \frac{9L^2 }{\sigma^A_{\min}\rho^2} + \frac{3\sigma^2_{\max}(G)}{\sigma^A_{\min}\eta^2\rho^2}.
\end{align}
 then we have
 \begin{align}
\frac{1}{T}\sum_{s=1}^S \sum_{t=0}^{M-1} \mathbb{E}\big[ \mbox{dist}(0,\partial L(x^s_t,y_{[m]}^{s,t},z^s_t))^2\big] &\leq \frac{\nu_{\max}}{T} \sum_{s=1}^S \sum_{t=0}^{M-1}\theta^s_t \leq \frac{2\nu_{\max}(\Gamma^1_0 - \Gamma^*)}{\gamma T}
\end{align}
where $\min(\sigma_{\min}^H,\frac{L}{2},\chi_t)$, $\nu_{\max}= \max(\nu_1,\nu_2,\nu_3)$ and
$\Gamma^*$ is a lower bound of function $\Gamma^s_t$.
Thus, given $(t^*,s^*) = \mathop{\arg\min}_{t,s}\theta^s_{t}$ and
 \begin{align}
 T = \frac{2\nu_{\max}(\Gamma^1_0 - \Gamma^*)}{\epsilon \gamma}, \nonumber
 \end{align}
then $(x^{s^*}_{t^*},y_{[m]}^{s^*,t^*},z^{s^*}_{t^*})$ is an $\epsilon$-stationary point of \eqref{eq:2}.
\end{theorem}

\begin{proof}
First, we define a variable $\theta^s_{t} = \|x^s_{t+1}-x^s_{t}\|^2 + \|x^s_{t}-x^s_{t-1}\|^2 + \frac{1}{b}(\|x^s_{t}-\tilde{x}^s\|^2 + \|x^s_{t-1}-\tilde{x}^s\|^2 ) + \sum_{j=1}^m\|y_j^{s,t} - y_j^{s,t+1}\|^2$.
By the step 7 of Algorithm \ref{alg:3}, we have, for all $i\in [m]$
\begin{align}
  \mathbb{E}\big[\mbox{dist}(0,\partial_{y_j} L(x,y_{[m]},z))^2\big]_{s,t+1} & = \mathbb{E}\big[\mbox{dist} (0, \partial g_j(y_j^{s,t+1})-B_j^Tz^s_{t+1})^2\big] \nonumber \\
 & = \|B_j^Tz^s_t -\rho B_j^T(Ax^s_t + \sum_{i=1}^jB_iy_i^{s,t+1} + \sum_{i=j+1}^m B_iy_i^{s,t} -c) - H_j(y_j^{s,t+1}-y_j^{s,t}) -B_j^Tz^s_{t+1}\|^2 \nonumber \\
 & = \|\rho B_j^TA(x^s_{t+1}-x^s_{t}) + \rho B_j^T \sum_{i=j+1}^m B_i (y_i^{s,t+1}-y_i^{s,t})- H_j(y_j^{s,t+1}-y_j^{s,t}) \|^2 \nonumber \\
 & \leq m\rho^2\sigma^{B_j}_{\max}\sigma^A_{\max}\|x^s_{t+1}-x^s_t\|^2 + m\rho^2\sigma^{B_j}_{\max}\sum_{i=j+1}^m \sigma^{B_i}_{\max}\|y_i^{s,t+1}-y_i^{s,t}\|^2 \nonumber \\
 & \quad + m\sigma^2_{\max}(H_j)\|y_j^{s,t+1}-y_j^{s,t}\|^2\nonumber \\
 & \leq m\big(\rho^2\sigma^B_{\max}\sigma^A_{\max} + \rho^2(\sigma^B_{\max})^2 + \sigma^2_{\max}(H)\big) \theta^s_{t},
\end{align}
where the first inequality follows by the inequality $\|\frac{1}{n}\sum_{i=1}^n z_i\|^2 \leq \frac{1}{n}\sum_{i=1}^n \|z_i\|^2$.

By the step 8 of Algorithm \ref{alg:3}, we have
\begin{align}\label{eq:A96}
 \mathbb{E}[\mbox{dist}(0,\nabla_x L(x,y,z))]_{s,t+1} & = \mathbb{E}\|A^Tz^s_{t+1}-\nabla f(x^s_{t+1})\|^2  \nonumber \\
 & = \mathbb{E}\|v^s_t - \nabla f(x^s_{t+1}) - \frac{G}{\eta} (x^s_t-x^s_{t+1})\|^2 \nonumber \\
 & = \mathbb{E}\|v^s_t - \nabla f(x^s_{t}) +\nabla f(x^s_{t})- \nabla f(x^s_{t+1})
  - \frac{G}{\eta}(x^s_t-x^s_{t+1})\|^2  \nonumber \\
 & \leq  \frac{3L^2}{b}\|x^s_t-\tilde{x}^s\|^2 + 3(L^2+ \frac{\sigma^2_{\max}(G)}{\eta^2})\|x^s_t-x^s_{t+1}\|^2  \nonumber \\
 & \leq \big( 3L^2 + \frac{3\sigma^2_{\max}(G)}{\eta^2} \big)\theta^s_{t}.
\end{align}
By the step 9 of Algorithm \ref{alg:3}, we have
\begin{align}\label{eq:85}
 \mathbb{E}[\mbox{dist}(0,\nabla_{z} L(x,y,z))]_{s,t+1} & = \mathbb{E}\|Ax^s_{t+1}+By^s_{t+1}-c\|^2 \nonumber \\
 &= \frac{1}{\rho^2} \mathbb{E} \|z^s_{t+1}-z^s_t\|^2  \nonumber \\
 & \leq \frac{9L^2 }{\sigma^A_{\min}\rho^2 b} \big( \|x^s_t - \tilde{x}^s\|^2
  + \|x^s_{t-1} - \tilde{x}^s\|^2\big) + \frac{3\sigma^2_{\max}(G)}{\sigma^A_{\min}\eta^2\rho^2}\|x^s_{t+1}-x^s_t\|^2 \nonumber \\
 & \quad + \frac{3(\sigma^2_{\max}(G) + 3L^2\eta^2)}{\sigma^A_{\min}\eta^2\rho^2}\|x^s_{t}-x^s_{t-1}\|^2 \nonumber \\
 & \leq \big( \frac{9L^2 }{\sigma^A_{\min}\rho^2} + \frac{3\sigma^2_{\max}(G)}{\sigma^A_{\min}\eta^2\rho^2} \big)\theta^s_{t}. \nonumber \\
\end{align}

Using \eqref{eq:A93}, we have
\begin{align}
 \frac{1}{T}\sum_{s=1}^S \sum_{t=0}^{M-1} ( \sigma_{\min}^H\sum_{j=1}^m \|y_j^{s,t}-y_j^{s,t+1}\|^2 + \frac{L}{2b}\|x^s_t-\tilde{x}^s\|^2  + \chi_t \|x^s_{t+1}-x^s_t\|^2) \leq \frac{\Gamma^1_0 - \Gamma^*}{T},
\end{align}
where $\chi_t \geq \frac{\sqrt{231}\kappa_G L}{2\alpha} > 0$.
Thus, we have
\begin{align}
 \frac{1}{T}\sum_{s=1}^S \sum_{t=0}^{M-1} \mathbb{E}\big[ \mbox{dist}(0,\partial L(x^s_t,y_{[m]}^{s,t},z^s_t))^2\big] & \leq \frac{\nu_{\max}}{T} \sum_{s=1}^S \sum_{t=0}^{M-1}\theta^s_t
  \leq \frac{2\nu_{\max}(\Gamma^1_0 - \Gamma^*)}{\gamma T},
\end{align}
where $\gamma = \min(\sigma_{\min}^H,\frac{L}{2},\chi_t)$ and $\nu_{\max} = \max(\nu_1,\nu_2,\nu_3)$ with
\begin{align}
 \nu_1 = m\big(\rho^2\sigma^B_{\max}\sigma^A_{\max} + \rho^2(\sigma^B_{\max})^2 + \sigma^2_{\max}(H)\big), \ \nu_2 = 3L^2 + \frac{3\sigma^2_{\max}(G)}{\eta^2},
 \ \nu_3 = \frac{9L^2 }{\sigma^A_{\min}\rho^2} + \frac{3\sigma^2_{\max}(G)}{\sigma^A_{\min}\eta^2\rho^2}.
\end{align}

Given $\eta = \frac{\alpha\sigma_{\min}(G)}{5L} \ (0 < \alpha \leq 1)$ and $\rho = \frac{2\sqrt{231}\kappa_G L}{\sigma^A_{\min}\alpha}$, since $m$ is relatively small,
it easy verifies that $\nu_{\max} = O(1)$ and $\gamma=O(1)$, which are independent on $n$ and $T$.
Thus, we obtain
\begin{align}
\frac{1}{T}\sum_{s=1}^S \sum_{t=0}^{M-1} \mathbb{E}\big[ \mbox{dist}(0,\partial L(x^s_t,y_{[m]}^{s,t},z^s_t))^2\big]  \leq O(\frac{1}{T}).
\end{align}

\end{proof}

\subsection{ Theoretical Analysis of the non-convex SAGA-ADMM }
\label{Appendix:A4}
In the subsection, we first extend the existing nonconvex SAGA-ADM to to the multi-blocks setting for solving the problem \eqref{eq:2},
which is summarized in Algorithm \ref{alg:4}. Then we afresh study the convergence analysis of this non-convex SVRG-ADMM.

\begin{algorithm}[htb]
   \caption{ SAGA-ADMM for Nonconvex  Optimization }
   \label{alg:4}
\begin{algorithmic}[1]
   \STATE {\bfseries Input:} $T$, $\eta$, $\rho$ and $H_j\succ0$ for all $j\in [m]$;
   \STATE {\bfseries Initialize:} $x_0$, $u_i^0=x_0$ for $i\in \{1,2,\cdots,n\}$, $\phi_0=\frac{1}{n}\sum_{i=1}^n\nabla f_i(u^0_i)$ and $y_j^0$ for all $j\in [m]$;
   \FOR {$t=0,1,\cdots,T-1$}
   \STATE{} Uniformly random pick a mini-batch $\mathcal{I}_t$ (with replacement) from $\{1,2,\cdots,n\}$ with $|\mathcal{I}_t|=b$, and compute
             $$v_{t} = \frac{1}{b}\sum_{i_t\in \mathcal{I}_t}
             \big(\nabla f_{i_t}(x_{t})-\nabla f_{i_t}(u^t_{i_t}) \big)+\hat{\phi}_t$$
            with $\phi_t=\frac{1}{n}\sum_{i=1}^n\nabla f_i(u^t_i)$;
   \STATE{} $ y^{t+1}_j= \arg\min_{y_j} \mathcal {L}_{\rho}(x_t,y^{t+1}_{[j-1]},y_j,y^{t}_{[j+1:m]},z_t) + \frac{1}{2}\|y_j-y^{t}_j\|_{H_j}^2$ for all $j\in [m]$;
   \STATE{} $x_{t+1}=\arg\min_x \hat{\mathcal {L}}_{\rho}\big(x,y_{t+1},z_{t},v_{t}\big)$;
   \STATE{} $z_{t+1} = z_{t}-\rho(Ax_{t+1} + \sum_{j=1}^mB_jy_j^{t+1}-c)$;
   \STATE{} $u^{t+1}_{i_t}= x_{t}$ for $i \in \mathcal{I}_t$ and $u_i^{t+1}=u^t_i$ for $i \not\in \mathcal{I}_t$;
   \STATE{} $\phi_{t+1} = \phi_t-\frac{1}{n}\sum_{i_t\in \mathcal{I}_t} \big(\nabla f_{i_t}(u^t_{i_t})-\nabla f_{i_t}(u^{t+1}_{i_t})\big)$;
   \ENDFOR
   \STATE {\bfseries Output \ (in theory):} Chosen uniformly random from $\{x_{t},y_{[m]}^{t},z_t\}_{t=1}^{T}$.
   \STATE {\bfseries Output \ (in practice):} $\{x_{T},y_{[m]}^{T},z_T\}$.
\end{algorithmic}
\end{algorithm}

\begin{lemma} \label{lem:A11}
 Suppose the sequence $\{x_t,y_{[m]}^t,z_t\}_{t=1}^T$ is generated by Algorithm \ref{alg:4}. The following inequality holds
 \begin{align}
 \mathbb{E}\|z_{t+1}-z_{t}\|^2 \leq & \frac{9 L^2 }{\sigma^A_{\min} b} \frac{1}{n}\sum_{i=1}^n \big( \|x_t - u^t_i\|^2
 + \|x_{t-1} - u^{t-1}_i\|^2\big)
  + \frac{3\sigma^2_{\max}(G)}{\sigma^A_{\min}\eta^2}\|x_{t+1}-x_t\|^2 \nonumber \\
  & + \frac{3(\sigma^2_{\max}(G) + 3L^2\eta^2)}{\sigma^A_{\min}\eta^2}\|x_{t}-x_{t-1}\|^2.
 \end{align}
\end{lemma}

\begin{proof}
 By the optimize condition of the the step 6 in Algorithm \ref{alg:4}, we have
 \begin{align}
   v_t + \frac{1}{\eta}G(x_{t+1}-x_t) - A^Tz_t + \rho A^T(Ax_{t+1}+\sum_{j=1}^mB_jy_j^{t+1}-c) = 0.
 \end{align}
 Using the step 7 of Algorithm \ref{alg:4}, then we have
 \begin{align} \label{eq:A12-1}
  A^Tz_{t+1} = v_t + \frac{G}{\eta}(x_{t+1}-x_t).
 \end{align}
 It follows that
 \begin{align}
 A^T(z_{t+1}-z_t) = & v_t - v_{t-1} + \frac{G}{\eta}(x_{t+1}-x_t) - \frac{1}{\eta}G(x_{t}-x_{t-1}).
 \end{align}
 By Assumption 4, we have
 \begin{align} \label{eq:A12-2}
 \|z_{t+1}-z_t\|^2 \leq \frac{1}{\sigma^A_{\min}}\big[3\|v_t - v_{t-1}\|^2 + \frac{3\sigma^2_{\max}(G)}{\eta^2}\|x_{t+1}-x_t\|^2
  + \frac{3\sigma^2_{\max}(G)}{\eta^2}\|x_{t}-x_{t-1}\|^2 \big] .
 \end{align}

 Next, considering the upper bound of $\|v^s_t - v^s_{t-1}\|^2$, we have
 \begin{align} \label{eq:A12-3}
  \|v_t - v_{t-1}\|^2 & = \|v_t - \nabla f(x_t) + \nabla f(x_t) -  \nabla f(x_{t-1}) + \nabla f(x_{t-1}) - v_{t-1}\|^2 \nonumber \\
  & \leq 3\|v_t - \nabla f(x_t)\|^2 + 3\|\nabla f(x_t) -  \nabla f(x_{t-1})\|^2 + 3\|\nabla f(x_{t-1}) - v_{t-1}\|^2 \nonumber \\
  & \leq \frac{3 L^2}{b} \frac{1}{n} \sum_{i=1}^n \big(\|x_t - u^t_i\|^2 + \|x_{t-1} - u^{t-1}_i\|^2 \big) + 3L^2\|x_t - x_{t-1}\|^2
 \end{align}
 where the second inequality holds by lemma 4 of \cite{Reddi2016Prox} and Assumption 1.
 Finally, combining the inequalities \eqref{eq:A12-2} and \eqref{eq:A12-3}, we can obtain the above result.
\end{proof}

\begin{lemma} \label{lem:A12}
 Suppose the sequence $\{x_t,y_{[m]}^t,z_t\}_{t=1}^T$ is generated from Algorithm \ref{alg:4},
 and define a \emph{Lyapunov} function
 \begin{align}
 \Omega_t = \mathbb{E}\big[ \mathcal{L}_{\rho} (x_t,y_{[m]}^t,z_t)+ (\frac{3\sigma^2_{\max}(G)}{\sigma^A_{\min}\rho\eta^2}+\frac{9L^2}{\sigma^A_{\min}\rho})
 \|x_{t}-x_{t-1}\|^2 + \frac{9L^2}{\sigma^A_{\min}\rho b}\frac{1}{n} \sum_{i=1}^n\|x_{t-1}-u^{t-1}_i\|^2 + c_t\frac{1}{n} \sum_{i=1}^n\|x_{t}-u^t_i\|^2 \big], \nonumber
\end{align}
 where the positive sequence $\{c_t\}$ satisfies
 \begin{equation*}
  c_t= \left\{
  \begin{aligned}
  & \frac{18L^2 }{\sigma^A_{\min}\rho b} + \frac{L}{b} +(1-p)(1+\beta)c_{t+1}, \ 0 \leq t \leq T-1, \\
  & 0, \ t \geq T,
  \end{aligned}
  \right.\end{equation*}
where $p$ denotes probability of an index $i$ being in $\mathcal{I}_t$.
Further, let $b= [n^{\frac{2}{3}}]$, $\eta = \frac{\alpha\sigma_{\min}(G)}{17L} \ (0 < \alpha \leq 1)$ and $\rho = \frac{2\sqrt{2031}\kappa_G}{\sigma^A_{\min}\alpha}$
we have
\begin{align}
 \frac{1}{T} \sum_{t=1}^T ( \sigma_{\min}^H\sum_{j=1}^m \|y_j^t-y_j^{t+1}\|^2 + \chi_t \|x_t-x_{t+1}\|^2 +  \frac{L}{2b}\frac{1}{n}\sum_{i=1}^n\|x_t-u^t_i\|^2 )
 \leq \frac{\Omega_0 - \Omega^*}{T},
\end{align}
where $\chi_t \geq \frac{\sqrt{2031}\kappa_GL}{2\alpha} >0$ and $\Omega^*$ denotes a low bound of $\Omega_t$.
\end{lemma}

\begin{proof}
By the optimal condition of step 5 in Algorithm \ref{alg:4},
we have, for $j\in [m]$
\begin{align}
0 & =(y_j^t-y_j^{t+1})^T\big(\partial g_j(y_j^{t+1}) - B^Tz_t + \rho B^T(Ax_t + \sum_{i=1}^jB_iy_i^{t+1} + \sum_{i=j+1}^mB_iy_i^{k}-c) + H_j(y_j^{t+1}-y_j^t)\big) \nonumber \\
& \leq g_j(y_j^t)- g_j(y_j^{t+1}) - (z_t)^T(B_jy_j^t-B_jy_j^{t+1}) + \rho(By_j^t-By_j^{t+1})^T(Ax_t + \sum_{i=1}^jB_iy_i^{t+1} + \sum_{i=j+1}^mB_iy_i^{k}-c) - \|y_j^{t+1}-y_j^t\|^2_{H_j} \nonumber \\
& = g_j(y_j^t)- g_j(y_j^{t+1}) - (z_t)^T(Ax_t+\sum_{i=1}^{j-1}B_iy_i^{t+1} + \sum_{i=j}^mB_iy_i^{k}-c) + (z_t)^T(Ax_t+\sum_{i=1}^jB_iy_i^{t+1}+ \sum_{i=j+1}^mB_iy_i^{k}-c) \nonumber \\
& \quad  + \frac{\rho}{2}\|Ax_t +\sum_{i=1}^{j-1}B_iy_i^{t+1} + \sum_{i=j}^mB_iy_i^{k}-c\|^2 - \frac{\rho}{2}\|Ax_t+\sum_{i=1}^jB_iy_i^{t+1}+ \sum_{i=j+1}^mB_iy_i^{k}-c\|^2 -\frac{\rho}{2}\|B_jy_j^t-B_jy_j^{t+1}\|^2 \nonumber \\
& \quad - \|y_j^{t+1}-y_j^t\|^2_{H_j} \nonumber \\
& \leq \underbrace{ f(x_t) + g_j(y_j^t) - (z_t)^T(Ax_t+\sum_{i=1}^{j-1}B_iy_i^{t+1} + \sum_{i=j}^mB_iy_i^{k}-c) + \frac{\rho}{2}\|Ax_t +\sum_{i=1}^{j-1}B_iy_i^{t+1} + \sum_{i=j}^mB_iy_i^{k}-c\|^2}_{\mathcal{L}_{\rho} (x_t,y_{[j-1]}^{t+1},y_{[j:m]}^t,z_t)}  - \|y_j^{t+1}-y_j^t\|^2_{H_j}  \nonumber \\
& \quad -\underbrace{(f(x_t) + g_j(y_j^{t+1}) - (z_t)^T(Ax_t+\sum_{i=1}^jB_iy_i^{t+1}+ \sum_{i=j+1}^mB_iy_i^{k}-c) + \frac{\rho}{2}\|Ax_t+\sum_{i=1}^jB_iy_i^{t+1}+ \sum_{i=j+1}^mB_iy_i^{k}-c\|^2}_{\mathcal{L}_{\rho} (x_t,y_{[j]}^{t+1},y_{[j+1:m]}^t,z_t)} \nonumber \\
& \leq \mathcal{L}_{\rho} (x_t,y_{[j-1]}^{t+1},y_{[j:m]}^t,z_t) - \mathcal{L}_{\rho} (x_t,y_{[j]}^{t+1},y_{[j+1:m]}^t,z_t)
- \sigma_{\min}(H_j)\|y_j^t-y_j^{t+1}\|^2,
\end{align}
where the first inequality holds by the convexity of function $g_j(y)$,
and the second equality follows by applying the equality
$(a-b)^Tb = \frac{1}{2}(\|a\|^2-\|b\|^2-\|a-b\|^2)$ on the term $(By_j^t-By_j^{t+1})^T(Ax_t + \sum_{i=1}^jB_iy_i^{t+1} + \sum_{i=j+1}^mB_iy_i^{k}-c)$.
Thus, we have, for all $j\in[m]$
\begin{align} \label{eq:A12-1}
 \mathcal{L}_{\rho} (x_t,y_{[j-1]}^{t+1},y_{[j:m]}^t,z_t) \leq \mathcal{L}_{\rho} (x_t,y_{[j]}^{t+1},y_{[j+1:m]}^t,z_t)
 - \sigma_{\min}(H_j)\|y_j^t-y_j^{t+1}\|^2.
\end{align}
Telescoping inequality \eqref{eq:A12-1} over $j$ from $1$ to $m$, we obtain
\begin{align} \label{eq:A13-1}
 \mathcal{L}_{\rho} (x_t,y^{t+1}_{[m]},z_t) \leq \mathcal{L}_{\rho} (x_t,y^t_{[m]},z_t)
 - \sigma_{\min}^H\sum_{j=1}^m \|y_j^t-y_j^{t+1}\|^2.
\end{align}
where $\sigma_{\min}^H=\min_{j\in[m]}\sigma_{\min}(H_j)$.

Using Assumption 1, we have
\begin{align} \label{eq:A13-2}
0 \leq f(x_t) - f(x_{t+1}) + \nabla f(x_t)^T(x_{t+1}-x_t) + \frac{L}{2}\|x_{t+1}-x_t\|^2.
\end{align}
By the step 6 of Algorithm \ref{alg:4},
we have
\begin{align} \label{eq:A13-3}
 0 = (x_t-x_{t+1})^T \big( v_t - A^Tz_t + \rho A^T(Ax_{t+1} + \sum_{j=1}^mB_jy_j^{t+1}-c) + \frac{G}{\eta}(x_{t+1}-x_t) \big).
\end{align}
Combining \eqref{eq:A13-2} and \eqref{eq:A13-3}, we have
\begin{align}
 0 & \leq f(x_t) - f(x_{t+1}) + \nabla f(x_t)^T(x_{t+1}-x_t) + \frac{L}{2}\|x_{t+1}-x_t\|^2 \nonumber \\
 & \quad + (x_t-x_{t+1})^T \big( v_t - A^Tz_t + \rho A^T(Ax_{t+1} + \sum_{j=1}^mB_jy_j^{t+1}-c) + \frac{G}{\eta}(x_{t+1}-x_t) \big)  \nonumber \\
 & = f(x_t) - f(x_{t+1}) + \frac{L}{2}\|x_t-x_{t+1}\|^2 - \frac{1}{\eta}\|x_t - x_{t+1}\|^2_G + (x_t-x_{t+1})^T(v_t-\nabla f(x_t)) \nonumber \\
 & \quad -(z_t)^T(Ax_t-Ax_{t+1}) + \rho(Ax_t - Ax_{t+1})^T(Ax_t + \sum_{j=1}^mB_jy_j^{t+1}-c) \nonumber \\
 & \mathop{=}^{(i)} f(x_t) - f(x_{t+1}) + \frac{L}{2}\|x_t-x_{t+1}\|^2 - \frac{1}{\eta}\|x_t - x_{t+1}\|^2_G + (x_t-x_{t+1})^T(v_t-\nabla f(x_t)) -(z_t)^T(Ax_t + \sum_{j=1}^mB_jy_j^{t+1}-c)\nonumber \\
 & \quad  + (z_t)^T(Ax_{t+1}+ \sum_{j=1}^mB_jy_j^{t+1}-c) + \frac{\rho}{2}\big(\|Ax_{t} + \sum_{j=1}^mB_jy_j^{t+1}-c\|^2 - \|Ax_{t+1} + \sum_{j=1}^mB_jy_j^{t+1}-c\|^2 - \|Ax_t - Ax_{t+1}\|^2 \big) \nonumber \\
 & = \underbrace{ f(x_t) -(z_t)^T(Ax_t + \sum_{j=1}^mB_jy_j^{t+1}-c) +  \frac{\rho}{2} \|Ax_{t} + \sum_{j=1}^mB_jy_j^{t+1}-c\|^2}_{ \mathcal{L}_{\rho} (x_t,y_{[m]}^{t+1},z_t)}
 + \frac{L}{2}\|x_t-x_{t+1}\|^2 + (x_t-x_{t+1})^T(v_t-\nabla f(x_t)) \nonumber \\
 & \quad -\underbrace{ (f(x_{t+1}) -(z_t)^T(Ax_{t+1} + \sum_{j=1}^mB_jy_j^{t+1}-c) +  \frac{\rho}{2} \|Ax_{t+1} + \sum_{j=1}^mB_jy_j^{t+1}-c\|^2}_{ \mathcal{L}_{\rho} (x_{t+1},y_{[m]}^{t+1},z_t)} - \frac{1}{\eta}\|x_t - x_{t+1}\|^2_G - \frac{\rho}{2}\|Ax_t - Ax_{t+1}\|^2 \nonumber \\
 & \leq \mathcal{L}_{\rho} (x_t,y_{[m]}^{t+1},z_t) -  \mathcal{L}_{\rho} (x_{t+1},y_{[m]}^{t+1},z_t)
 - (\frac{\sigma_{\min}(G)}{\eta} + \frac{\rho \sigma^A_{\min}}{2} - \frac{L}{2}) \|x_t - x_{t+1}\|^2 +(x_t-x_{t+1})^T(v_t-\nabla f(x_t)) \nonumber \\
 & \mathop{\leq}^{(ii)}  \mathcal{L}_{\rho} (x_t,y_{[m]}^{t+1},z_t) -  \mathcal{L}_{\rho} (x_{t+1},y_{[m]}^{t+1},z_t)
 - (\frac{\sigma_{\min}(G)}{\eta} + \frac{\rho \sigma^A_{\min}}{2} - L) \|x_t - x_{t+1}\|^2 + \frac{1}{2L}\|v_t-\nabla f(x_t)\|^2 \nonumber \\
 & \mathop{\leq}^{(iii)} \mathcal{L}_{\rho} (x_t,y_{[m]}^{t+1},z_t) -  \mathcal{L}_{\rho} (x_{t+1},y_{[m]}^{t+1},z_t)
 - (\frac{\sigma_{\min}(G)}{\eta} + \frac{\rho \sigma^A_{\min}}{2} - L) \|x_t - x_{t+1}\|^2 + \frac{L}{2b}\frac{1}{n} \sum_{i=1}^n \|x_t-u^t_i\|^2,
\end{align}
where the equality $(i)$ holds by applying the equality
$(a-b)^Tb = \frac{1}{2}(\|a\|^2-\|b\|^2-\|a-b\|^2)$ on the
term $(Ax_t - Ax_{t+1})^T(Ax_{t+1}+\sum_{j=1}^mB_jy_j^{t+1}-c)$; the inequality
$(ii)$ follows by the inequality $a^Tb \leq \frac{L}{2}\|a\|^2 + \frac{1}{2L}\|a\|^2$,
and the inequality $(iii)$ holds by Lemma 4 of \cite{Reddi2016Prox}.
Thus, we obtain
\begin{align} \label{eq:A13-4}
\mathcal{L}_{\rho} (x_{t+1},y_{[m]}^{t+1},z_t) \leq & \mathcal{L}_{\rho} (x_t,y_{[m]}^{t+1},z_t) -
(\frac{\sigma_{\min}(G)}{\eta} + \frac{\rho \sigma^A_{\min}}{2} - L)\|x_t - x_{t+1}\|^2 \nonumber \\
& +\frac{ L}{2b}\frac{1}{n} \sum_{i=1}^n\|x_t-u^t_i\|^2.
\end{align}

By the step 7 in Algorithm \ref{alg:4}, we have
\begin{align} \label{eq:A13-5}
\mathcal{L}_{\rho} (x_{t+1},y_{[m]}^{t+1},z_{t+1}) - \mathcal{L}_{\rho} (x_{t+1},y_{[m]}^{t+1},z_t)
& = \frac{1}{\rho}\|z_{t+1}-z_t\|^2 \nonumber \\
& \leq  \frac{9L^2 }{\sigma^A_{\min}\rho b} \frac{1}{n} \sum_{i=1}^n\big( \|x_t - u^t_i\|^2 + \|x_{t-1} - u^{t-1}_i\|^2\big)
 + \frac{3\sigma^2_{\max}(G)}{\sigma^A_{\min}\eta^2\rho}\|x_{t+1}-x_t\|^2 \nonumber \\
& \quad + \frac{3(\sigma^2_{\max}(G)+3L^2\eta^2)}{\sigma^A_{\min}\eta^2\rho}\|x_{t}-x_{t-1}\|^2,
\end{align}
where the first inequality follows by Lemma \ref{lem:A11}.

Combining \eqref{eq:A13-1}, \eqref{eq:A13-4} and \eqref{eq:A13-5}, we have
\begin{align} \label{eq:A13-6}
\mathcal{L}_{\rho} (x_{t+1},y_{[m]}^{t+1},z_{t+1}) & \leq \mathcal{L}_{\rho} (x_t,y_{[m]}^t,z_t)
 - (\frac{\sigma_{\min}(G)}{\eta} + \frac{\rho \sigma^A_{\min}}{2} - L)\|x_t - x_{t+1}\|^2 +\frac{ L}{2b}\frac{1}{n} \sum_{i=1}^n\|x_t-u^t_i\|^2  \nonumber \\
& \quad - \sigma_{\min}^H\sum_{j=1}^m \|y_j^t-y_j^{t+1}\|^2 + \frac{9L^2 }{\sigma^A_{\min}\rho b} \frac{1}{n} \sum_{i=1}^n\big( \|x_t - u^t_i\|^2 + \|x_{t-1} - u^{t-1}_i\|^2\big) \nonumber \\
& \quad + \frac{3\sigma^2_{\max}(G)}{\sigma^A_{\min}\eta^2\rho}\|x_{t+1}-x_t\|^2 + \frac{3(\sigma^2_{\max}(G)+3L^2\eta^2)}{\sigma^A_{\min}\eta^2\rho}\|x_{t}-x_{t-1}\|^2.
\end{align}

Next, we define a \emph{Lyapunov} function as follows:
\begin{align} \label{eq:A132}
 \Omega_t = \mathbb{E}\big[ \mathcal{L}_{\rho} (x_t,y_{[m]}^t,z_t)+ (\frac{3\sigma^2_{\max}(G)}{\sigma^A_{\min}\rho\eta^2}+\frac{9L^2}{\sigma^A_{\min}\rho})
 \|x_{t}-x_{t-1}\|^2 + \frac{9 L^2 }{\sigma^A_{\min}\rho b}\frac{1}{n} \sum_{i=1}^n\|x_{t-1}-z^{t-1}_i\|^2 +\frac{c_t}{n} \sum_{i=1}^n\|x_{t}-z^t_i\|^2 \big],
\end{align}
where $\kappa_A = \frac{\sigma^A_{\max}}{\sigma^A_{\min}} \geq 1$.

By the step 9 of Algorithm \ref{alg:4}, we have
 \begin{align} \label{eq:A13-7}
  \frac{1}{n}\sum_{i=1}^n \|x_{t+1}-u^{t+1}_i\|^2 &= \frac{1}{n}\sum_{i=1}^n \big( p\|x_{t+1}-x_{t}\|^2 + (1-p)\|x_{t+1}-u^{t}_i\|^2 \big) \nonumber \\
   & = \frac{p}{n} \sum_{i=1}^n \|x_{t+1}-x_{t}\|^2 + \frac{1-p}{n}\sum_{i=1}^n\|x_{t+1}-u^{t}_i\|^2  \nonumber \\
   & = p \|x_{t+1}-x_{t}\|^2 + \frac{1-p}{n} \sum_{i=1}^n \|x_{t+1}-u^{t}_i\|^2,
 \end{align}
where $p$ denotes probability of an index $i$ being in $\mathcal{I}_t$. Here, we have
 \begin{align}
  p = 1-(1-\frac{1}{n})^b \geq 1- \frac{1}{1+b/n} = \frac{b/n}{1+b/n} \geq \frac{b}{2n},
 \end{align}
where the first inequality follows from $(1-a)^b\leq \frac{1}{1+ab}$, and the second inequality holds by $b\leq n$.
Considering the upper bound of $\|x_{t+1}-z^{t}_i\|^2$, we have
\begin{align} \label{eq:A13-8}
  \|x_{t+1}-u^{t}_i\|^2 & = \|x_{t+1}-x_t+x_t-u^{t}_i\|^2 \nonumber \\
  & =  \|x_{t+1}-x_t\|^2 + 2(x_{t+1}-x_t)^T(x_t-u^{t}_i)+ \|x_t-u^{t}_i\|^2 \nonumber \\
  & \leq \|x_{t+1}-x_t\|^2 + 2\big( \frac{1}{2\beta}\|x_{t+1}-x_t\|^2 + \frac{\beta}{2}\|x_t-u^{t}_i\|^2\big)+ \|x_t-u^{t}_i\|^2 \nonumber \\
  & = (1+\frac{1}{\beta})\|x_{t+1}-x_t\|^2 + (1+\beta)\|x_t-u^{t}_i\|^2,
\end{align}
 where $\beta>0$.
Combining \eqref{eq:A13-7} with \eqref{eq:A13-8}, we have
 \begin{align}
 \frac{1}{n}\sum_{i=1}^n \|x_{t+1}-u^{t+1}_i\|^2 \leq (1+\frac{1-p}{\beta})\|x_{t+1}-x_t\|^2 +\frac{(1-p)(1+\beta)}{n}\sum_{i=1}^n\|x_t-u^{t}_i\|^2.
 \end{align}
It follows that
\begin{align} \label{eq:A13-9}
\Omega_{t+1} & = \mathbb{E}\big[\mathcal{L}_{\rho}(x_{t+1},y_{[m]}^{t+1},z_{t+1}) + (\frac{3\sigma^2_{\max}(G)}{\sigma^A_{\min}\rho\eta^2}+\frac{9L^2}{\sigma^A_{\min}\rho})\|x_{t+1}-x_{t}\|^2
+ \frac{9L^2 }{\sigma^A_{\min}\rho b}\frac{1}{n} \sum_{i=1}^n\|x_{t}-u^{t}_i\|^2 +\frac{c_{t+1}}{n}\sum_{i=1}^n \|x_{t+1}-u^{t+1}_i\|^2\big] \nonumber \\
& \leq \mathcal{L}_{\rho} (x_t,y_{[m]}^t,z_t) + (\frac{3\sigma^2_{\max}(G)}{\sigma^A_{\min}\rho\eta^2}+\frac{9L^2}{\sigma^A_{\min}\rho})\|x_{t}-x_{t-1}\|^2
 + \frac{9L^2}{\sigma^A_{\min}\rho b} \frac{1}{n}\sum_{i=1}^n\|x_{t-1}-u^{t-1}_i\|^2   \nonumber \\
& \quad + \big(\frac{18L^2}{\sigma^A_{\min}\rho b} + \frac{L}{b}+(1-p)(1+\beta)c_{t+1}\big)\frac{1}{n}\sum_{i=1}^n\|x_t-u^t_i\|^2 - \sigma_{\min}^H\sum_{j=1}^m \|y_j^t-y_j^{t+1}\|^2 \nonumber\\
& \quad - \frac{L}{2b}\frac{1}{n}\sum_{i=1}^n\|x_t-u^t_i\|^2 - \big( \frac{\sigma_{\min}(G)}{\eta}+\frac{\rho\sigma^A_{\min}}{2} - L
-\frac{6\sigma^2_{\max}(G)}{\sigma^A_{\min}\eta^2\rho} - \frac{9L^2}{\sigma^A_{\min}\rho}-(1+\frac{1-p}{\beta})c_{t+1} \big)\|x_t - x_{t+1}\|^2 \nonumber \\
& = \Omega_t - \sigma_{\min}^H\sum_{j=1}^m \|y_j^t-y_j^{t+1}\|^2 - \chi_t \|x_t - x_{t+1}\|^2 - \frac{L}{2b}\frac{1}{n}\sum_{i=1}^n\|x_t-u^t_i\|^2,
\end{align}
where $c_t = \frac{18L^2}{\sigma^A_{\min}\rho b} + \frac{L}{b} +(1-p)(1+\beta)c_{t+1}$
and $\chi_t=\frac{\sigma_{\min}(G)}{\eta}+\frac{\rho\sigma^A_{\min}}{2} - L
-\frac{6\sigma^2_{\max}(G)}{\sigma^A_{\min}\eta^2\rho} - \frac{9L^2}{\sigma^A_{\min}\rho}-(1+\frac{1-p}{\beta})c_{t+1}$.

Let $c_T = 0$ and $\beta=\frac{b}{4n}$. Since $(1-p)(1+\beta)=1+\beta-p-p\beta\leq 1+\beta-p$ and $p\geq \frac{b}{2n}$,
it follows that
\begin{align}
 c_t \leq c_{t+1}(1-\theta) + \frac{18L^2 }{\sigma^A_{\min}\rho b} + \frac{L}{b},
\end{align}
where $\theta = p-\beta\geq \frac{b}{4n}$.
Then recursing on $t$, for $0\leq t \leq T-1$, we have
\begin{align}
 c_t \leq \frac{1}{b}(\frac{18L^2}{\sigma^A_{\min}\rho} + L) \frac{1-\theta^{T-t}}{\theta} \leq \frac{1}{b\theta}(\frac{18L^2}{\sigma^A_{\min}\rho} + L)
 \leq \frac{4n}{b^2}(\frac{18L^2}{\sigma^A_{\min}\rho} + L).
\end{align}
It follows that
\begin{align}
\chi_t & = \frac{\sigma_{\min}(G)}{\eta}+\frac{\rho\sigma^A_{\min}}{2} - L
-\frac{6\sigma^2_{\max}(G)}{\sigma^A_{\min}\eta^2\rho} - \frac{9L^2}{\sigma^A_{\min}\rho}-(1+\frac{1-p}{\beta})c_{t+1} \nonumber \\
& \geq \frac{\sigma_{\min}(G)}{\eta}+\frac{\rho\sigma^A_{\min}}{2} - L
-\frac{6\sigma^2_{\max}(G)}{\sigma^A_{\min}\eta^2\rho} - \frac{9L^2}{\sigma^A_{\min}\rho}-( 1 + \frac{4n-2b}{b})\frac{4n}{b^2}(\frac{18\kappa_AL^2}{\sigma^A_{\min}\rho} + L) \nonumber \\
& = \frac{\sigma_{\min}(G)}{\eta}+\frac{\rho\sigma^A_{\min}}{2} - L
-\frac{6\sigma^2_{\max}(G)}{\sigma^A_{\min}\eta^2\rho} - \frac{9L^2}{\sigma^A_{\min}\rho} - (\frac{4n}{b}-1)\frac{4n}{b^2}(\frac{18\kappa_AL^2}{\sigma^A_{\min}\rho} + L) \nonumber \\
& \geq \frac{\sigma_{\min}(G)}{\eta}+\frac{\rho\sigma^A_{\min}}{2} - L
-\frac{6\sigma^2_{\max}(G)}{\sigma^A_{\min}\eta^2\rho} - \frac{9L^2}{\sigma^A_{\min}\rho} - \frac{16n^2}{b^3}(\frac{18\kappa_AL^2}{\sigma^A_{\min}\rho} + L) \nonumber \\
& = \underbrace{\frac{\sigma_{\min}(G)}{\eta}- L - \frac{16n^2L}{b^3}}_{Q_1} + \underbrace{\frac{\rho\sigma^A_{\min}}{2} -\frac{6\sigma^2_{\max}(G)}{\sigma^A_{\min}\eta^2\rho} - \frac{9L^2}{\sigma^A_{\min}\rho} - \frac{288n^2L^2}{\sigma^A_{\min}\rho b^3}}_{Q_2}
\end{align}

Let $b=[n^{\frac{2}{3}}]$ and $0< \eta \leq \frac{\sigma_{\min}(G)}{17 L}$, we have $Q_1 \geq 0$. Further, let $\eta = \frac{\alpha\sigma_{\min}(G)}{17 L} \ (0<\alpha \leq 1)$
and $\rho=\frac{2\sqrt{2031}\kappa_GL}{\sigma^A_{\min}\alpha}$, we have
\begin{align}
Q_2 & = \frac{\rho\sigma^A_{\min}}{2} -\frac{6\sigma^2_{\max}(G)}{\sigma^A_{\min}\eta^2\rho} - \frac{9L^2}{\sigma^A_{\min}\rho} - \frac{288n^2L^2}{\sigma^A_{\min}\rho b^3} \nonumber \\
& = \frac{\rho\sigma^A_{\min}}{2} -\frac{1734\kappa^2_GL^2}{\sigma^A_{\min}\alpha^2\rho} - \frac{9L^2}{\sigma^A_{\min}\rho} - \frac{288L^2}{\sigma^A_{\min}\rho } \nonumber \\
& \geq \frac{\rho\sigma^A_{\min}}{4} + \underbrace{\frac{\rho\sigma^A_{\min}}{4} - \frac{2031\kappa^2_GL^2}{\sigma^A_{\min}\alpha^2\rho}}_{\geq 0} \nonumber \\
& \geq \frac{\sqrt{2031}\kappa_GL}{2\alpha},
\end{align}
where $\kappa_G \geq 1$. Thus, we have $\chi_t \geq \frac{\sqrt{2031}\kappa_GL}{2\alpha}$ for all $t$.

Since $\frac{L}{2b} > 0$ and $\chi_t > 0$, by \eqref{eq:A13-9}, the function $\Omega_t$ is monotone decreasing.
By \eqref{eq:A132}, we have
\begin{align} \label{eq:A142}
 \Omega_t &\geq   \mathbb{E}\big[\mathcal{L}_{\rho} (x_t,y_{[m]}^t,z_t) ] \nonumber \\
 & = f(x_t) + \sum_{j=1}^m g_j(y_j^t) - (z_t)^T(Ax_t + \sum_{j=1}^m B_j y_j^{t} - c) + \frac{\rho}{2}\|Ax_t + \sum_{j=1}^m B_j y_j^{t} -c\| \nonumber \\
 & = f(x_t) + \sum_{j=1}^m g_j(y_j^t) - \frac{1}{\rho}(z_t)^T(z_{t-1} - z_t) + \frac{1}{2\rho}\|z_t - z_{t-1}\|^2 \nonumber \\
 & = f(x_t) + \sum_{j=1}^m g_j(y_j^t) - \frac{1}{2\rho}\|z_{t-1}\|^2 + \frac{1}{2\rho}\|z_t\|^2 + \frac{1}{\rho}\|z_t - z_{t-1}\|^2 \nonumber \\
 & \geq f^* + \sum_{j=1}^m g_j^* - \frac{1}{2\rho}\|z_{t-1}\|^2 + \frac{1}{2\rho}\|z_t\|^2.
\end{align}
Summing the inequality \eqref{eq:A142} over $t=0,1\cdots,T$, we have
\begin{align}
 \frac{1}{T} \sum_{t=0}^T  \Omega_t \geq f^* + \sum_{j=1}^m g_j^*  - \frac{1}{2\rho}\|z_{0}\|^2.
\end{align}
Thus, the function $\Omega_t$ is bounded from below. Set $\Omega^*$ denotes a low bound of $\Omega_t$.

Finally, telescoping inequality \eqref{eq:A13-9} over $t$ from $0$ to $T$,
we have
\begin{align} \label{eq:A144}
 \frac{1}{T} \sum_{t=1}^T ( \sigma_{\min}^H\sum_{j=1}^m \|y_j^t-y_j^{t+1}\|^2 + \chi_t \|x_t-x_{t+1}\|^2 +  \frac{L}{2b}\frac{1}{n}\sum_{i=1}^n\|x_t-u^t_i\|^2 )
 \leq \frac{\Omega_0 - \Omega^*}{T},
\end{align}
where $\chi_t \geq \frac{\sqrt{2031}\kappa_GL}{2\alpha} >0$.

\end{proof}

\begin{theorem}
Suppose the sequence $\{x_t,y_{[m]}^t,z_t\}_{t=1}^T$ is generated from Algorithm \ref{alg:4}, and let $b= [n^{\frac{2}{3}}]$, $\eta = \frac{\alpha\sigma_{\min}(G)}{17L} \ (0 < \alpha \leq 1)$, $\rho = \frac{2\sqrt{2031}\kappa_G}{\sigma^A_{\min}\alpha}$ and
 \begin{align}
  &\nu_1 =m\big(\rho^2\sigma^B_{\max}\sigma^A_{\max} + \rho^2(\sigma^B_{\max})^2 + \sigma^2_{\max}(H)\big), \ \nu_2 = 3L^2 + \frac{3\sigma^2_{\max}(G)}{\eta^2} \nonumber \\
  &\nu_3 = \frac{9L^2 }{\sigma^A_{\min}\rho^2} + \frac{3\sigma^2_{\max}(G)}{\sigma^A_{\min}\eta^2\rho^2},
 \end{align}
 then we have
\begin{align}
\frac{1}{T}\sum_{t=1}^T \mathbb{E}\big[ \mbox{dist}(0,\partial L(x_t,y_{[m]}^t,z_t))^2\big] \leq \frac{\nu_{\max}}{T} \sum_{t=1}^T\theta_t \leq \frac{2\nu_{\max}(\Omega_0 - \Omega^*)}{\gamma T}  \nonumber
\end{align}
 where $\gamma = \min(\sigma_{\min}^H, L/2, \chi_{t})$ with $\chi_t \geq \frac{\sqrt{2031}\kappa_GL}{2\alpha} >0$, $\nu_{\max}= \max(\nu_1,\nu_2,\nu_3)$ and
 $\Omega^*$ is a lower bound of function $\Omega_t$.
 Then, given $t^* = \mathop{\arg\min}_{ 1\leq t\leq T}\theta_{t}$ and
 \begin{align}
  T = \frac{2\kappa_{\max}}{\epsilon \gamma}(\Omega_0- \Omega^*),
 \end{align}
 then $(x_{t^*},y_{[m]}^{t^*},z_{t^*})$ is an $\epsilon$-approximate stationary point of \eqref{eq:2}.
\end{theorem}

\begin{proof}
We begin with defining a useful variable $\theta_t = \|x_{t+1}-x_t\|^2 + \|x_t-x_{t-1}\|^2 + \frac{1}{b n}\sum^n_{i=1} (\|x_t-u^t_i\|^2 + \|x_{t-1}-u^{t-1}_i\|^2)+\sum_{j=1}^m \|y_j^t-y_j^{t+1}\|^2$.
By the optimal condition of the step 5 in Algorithm \ref{alg:4}, we
have, for all $i\in [m]$
\begin{align}
  \mathbb{E}\big[\mbox{dist}(0,\partial_{y_j} L(x,y_{[m]},z))^2\big]_{t+1} & = \mathbb{E}\big[\mbox{dist} (0, \partial g_j(y_j^{t+1})-B_j^Tz_{t+1})^2\big] \nonumber \\
 & = \|B_j^Tz_t -\rho B_j^T(Ax_t + \sum_{i=1}^jB_iy_i^{t+1} + \sum_{i=j+1}^mB_iy_i^{t} -c) - H_j(y_j^{t+1}-y_j^k) -B_j^Tz_{t+1}\|^2 \nonumber \\
 & = \|\rho B_j^TA(x_{t+1}-x_{t}) + \rho B_j^T \sum_{i=j+1}^m B_i (y_i^{t+1}-y_i^{t})- H_j(y_j^{t+1}-y_j^k) \|^2 \nonumber \\
 & \leq m\rho^2\sigma^{B_j}_{\max}\sigma^A_{\max}\|x_{t+1}-x_t\|^2 + m\rho^2\sigma^{B_j}_{\max}\sum_{i=j+1}^m \sigma^{B_i}_{\max}\|y_i^{t+1}-y_i^{t}\|^2
 + m\sigma^2_{\max}(H_j)\|y_j^{t+1}-y_j^k\|^2\nonumber \\
 & \leq m\big(\rho^2\sigma^B_{\max}\sigma^A_{\max} + \rho^2(\sigma^B_{\max})^2 + \sigma^2_{\max}(H)\big) \theta_{t},
\end{align}
where the first inequality follows by the inequality $\|\frac{1}{n}\sum_{i=1}^n z_i\|^2 \leq \frac{1}{n}\sum_{i=1}^n \|z_i\|^2$.

By the step 6 in Algorithm \ref{alg:4}, we have
\begin{align}\label{eq:A14-3}
 \mathbb{E}[\mbox{dist}(0,\nabla_x L(x,y_{[m]},z))]_{t+1} & = \mathbb{E}\|A^Tz_{t+1}-\nabla f(x_{t+1})\|^2  \nonumber \\
 & = \mathbb{E}\|v_t - \nabla f(x_{t+1}) - \frac{G}{\eta} (x_t-x_{t+1})\|^2 \nonumber \\
 & = \mathbb{E}\|v_t - \nabla f(x_{t}) +\nabla f(x_{t})- \nabla f(x_{t+1})
  - \frac{G}{\eta}(x_t-x_{t+1})\|^2  \nonumber \\
 & \leq  \frac{3L^2}{b n}\sum_{i=1}^n\|x_t-u^t_i\|^2 + 3(L^2+ \frac{\sigma^2_{\max}(G)}{\eta^2})\|x_t-x_{t+1}\|^2 \nonumber \\
 & \leq \big( 3L^2+ \frac{3\sigma^2_{\max}(G)}{\eta^2} \big)\theta_{t}.
\end{align}

By the step 7 of Algorithm \ref{alg:4}, we have
\begin{align}\label{eq:A14-4}
 \mathbb{E}[\mbox{dist}(0,\nabla_{z} L(x,y_{[m]},z))]_{t+1} & = \mathbb{E}\|Ax_{t+1}+By_{t+1}-c\|^2 \nonumber \\
 & = \frac{1}{\rho^2} \mathbb{E} \|z_{t+1}-z_t\|^2  \nonumber \\
 & \leq \frac{9L^2 }{\sigma^A_{\min}\rho^2 b} \frac{1}{n}\sum_{i=1}^n\big( \|x_t - u^t_i\|^2
  + \|x_{t-1} - u^{t-1}_i\|^2\big) + \frac{3\sigma^2_{\max}(G)}{\sigma^A_{\min}\eta^2\rho^2}\|x_{t+1}-x_t\|^2 \nonumber \\
 & \quad +(\frac{3\sigma^2_{\max}(G)}{\sigma^A_{\min}\eta^2\rho^2} + \frac{9L^2)}{\sigma^A_{\min}\rho^2})\|x_{t}-x_{t-1}\|^2 \nonumber \\
 & \leq \big(\frac{9L^2}{\sigma^A_{\min}\rho^2} + \frac{3\sigma^2_{\max}(G)}{\sigma^A_{\min}\eta^2\rho^2} \big)\theta_{t}. \nonumber \\
\end{align}

Using \eqref{eq:A144}, we have
\begin{align}
\frac{1}{T}\sum_{t=1}^T \mathbb{E}\big[ \mbox{dist}(0,\partial L(x_t,y_{[m]}^t,z_t))^2\big] & \leq \frac{\nu_{\max}}{T} \sum_{t=1}^T \theta_t
\leq \frac{2\nu_{\max}(\Omega_0 - \Omega^*)}{\gamma T}  ,
\end{align}
where $\gamma = \min(\sigma_{\min}^H, L/2, \chi_{t})$, $\nu_{\max}= \max(\nu_1,\nu_2,\nu_3)$ with
\begin{align}
 \nu_1 = m\big(\rho^2\sigma^B_{\max}\sigma^A_{\max} + \rho^2(\sigma^B_{\max})^2 + \sigma^2_{\max}(H)\big),
 \ \nu_2 = 6 L^2 + \frac{3\sigma^2_{\max}(G)}{\eta^2}, \ \nu_3 = \frac{9L^2 }{\sigma^A_{\min}\rho^2} + \frac{3\sigma^2_{\max}(G)}{\sigma^A_{\min}\eta^2\rho^2}. \nonumber
\end{align}
Given $\eta = \frac{\alpha\sigma_{\min}(G)}{17L} \ (0 < \alpha \leq 1)$ and $\rho = \frac{2\sqrt{2031}\kappa_G}{\sigma^A_{\min}\alpha}$,
since $m$ is relatively small, it easy verifies that $\gamma = O(1) $ and $\nu_{\max} = O(1)$,
which are independent on $n$ and $T$.
Thus, we obtain
\begin{align}
\frac{1}{T}\sum_{t=1}^T \mathbb{E}\big[ \mbox{dist}(0,\partial L(x_t,y_{[m]}^t,z_t))^2\big]  \leq O(\frac{1}{T}).
\end{align}

\end{proof}

\end{appendices}

\end{onecolumn}

\end{document}